\documentclass[11pt,oneside, reqno, a4paper]{amsart} 

\usepackage[english]{babel}
\usepackage[T1]{fontenc}
\usepackage{amsmath,mathdots}
\usepackage{amssymb}
\usepackage{float}
\usepackage{amsfonts, bm}
\usepackage{mathtools}
\usepackage[utf8]{inputenc}
\usepackage[mathcal]{eucal}
\usepackage{enumerate}
\usepackage{mathrsfs}
\usepackage{tikz}
\usetikzlibrary{arrows, matrix}
\usepackage{amsthm}
\usepackage{esint}
\usepackage{hyperref}
\usepackage[totalwidth=14cm,totalheight=20.5cm]{geometry}
\usepackage{epsfig,fancyhdr,color}
\usepackage{multicol} 
\usepackage{csquotes}
\usepackage{tikz-cd}
\usepackage[toc]{appendix}

\newtheorem{cor}{Corollary}[section]

\newtheorem{theorem}[cor]{Theorem}

\newtheorem{prop}[cor]{Proposition}

\newtheorem{lemma}[cor]{Lemma}

\theoremstyle{definition}
\newtheorem{defi}[cor]{Definition}
\newtheorem{manualtheoreminner}{Theorem}
\newenvironment{manualtheorem}[1]{%
  \renewcommand\themanualtheoreminner{#1}%
  \manualtheoreminner
}{\endmanualtheoreminner}
\theoremstyle{remark}
\newtheorem{remark}[cor]{Remark}
\newtheorem*{remark*}{Remark}
\newtheorem{example}[cor]{Example}

\newcommand{\Pp}{\mathbb{P}}
\newcommand{\R}{\mathbb{R}}

\newcommand{\C}{\mathbb{C}}
\newcommand{\Z}{\mathbb{Z}}

\newcommand{\h}{\mathbb{H}}

\newcommand{\p}{\mathbf{P}}

\newcommand{\Diff}{\mathrm{Diff}}

\newcommand{\diag}{\mathrm{diag}}

\newcommand{\SL}{\mathrm{SL}}

\newcommand{\vl}{|}
\newcommand{\Ker}{\mathrm{Ker}}

\newcommand{\PSL}{\mathbb{P}\mathrm{SL}}

\newcommand{\Ima}{\mathrm{Im}}
\newcommand{\Sec}{\mathrm{Sec}}

\newcommand{\trace}{\mathrm{tr}}

\newcommand{\II}{\mathrm{I}\mathrm{I}}

\newcommand{\U}{\mathrm{U}}
\newcommand{\SU}{\mathrm{SU}}
\newcommand{\GL}{\mathrm{GL}}

\newcommand{\SO}{\mathrm{SO}}

\newcommand{\Hom}{\mathrm{Hom}}

\newcommand{\Id}{\mathrm{Id}}

\newcommand{\Ree}{\mathrm{Re}}

\newcommand{\Sg}{\Sigma}

\newcommand{\bigslant}[2]{{\raisebox{.2em}{$#1$}\left/\raisebox{-.2em}{$#2$}\right.}}
\newcommand{\g}{\mathbf{g}}
\newcommand{\ome}{\boldsymbol{\omega}}
\newcommand{\divi}{\mathrm{div}}
\newcommand{\dev}{\mathrm{dev}}

\DeclareMathAlphabet{\mathpzc}{OT1}{pzc}{m}{it}

\usepackage{mathtools}

\title[superconformal surfaces and cyclic Higgs bundles]{superconformal surfaces in para-complex hyperbolic space and cyclic Higgs bundles}

\allowdisplaybreaks

\begin{document}

\setcounter{secnumdepth}{3}
\setcounter{tocdepth}{2}

\title[Para-complex geometry and cyclic Higgs bundles]{Para-complex geometry and cyclic Higgs bundles} 

\author[Nicholas Rungi]{Nicholas Rungi}
\address{NR: Department of Mathematics, University of Turin, Italy.} \email{nicholas.rungi@unito.it} 

\author[Andrea Tamburelli]{Andrea Tamburelli}
\address{AT: Department of Mathematics, University of Pisa, Italy.} \email{andrea.tamburelli@libero.it}

\date{\today}

\begin{abstract}We introduce para-complex and pseudo-Riemannian geometric methods for the study of representations of surface groups in $\SL(2m+1,\R)$. For $m=1$ our techniques allow to recover several known results for Hitchin representations without any reference to convex projective geometry or hyperbolic affine spheres. In particular, we describe analytically the Guichard-Wienhard domain of discontinuity in the flag variety and the corresponding concave foliated flag structure of Nolte-Riestenberg. In higher rank, we obtain a one-to-one correspondence between stable cyclic $\SL(2m+1,\R)$-Higgs bundles (not necessarily in the Hitchin component) and a special class of surfaces, which we call isotropic $\mathbf{P}$-alternating, in the para-complex hyperbolic space $\mathbb{H}^{2m}_{\tau}$. As a result, we give a geometric interpretation to the holomorphic differential $q_{2m+1}$ in the Hitchin base in terms of harmonic sequences for immersions in para-complex manifolds.
    
\end{abstract}

\maketitle

\tableofcontents

\section{Introduction}
\noindent The study of smooth manifolds defined over the algebra of para-complex numbers \( \R_\tau := \{z = x + \tau y \ | \ \tau^2 = 1, \ x,y\in\R\} \) (\cite{cockle1849iii}), has gained significant interest in recent years, both from a purely mathematical perspective and for its applications in theoretical physics (see \cite{cruceanu1996survey} for a survey of both topics). On the one hand, results analogous to those in the theory of complex manifolds have been established, while on the other hand, many new phenomena have been discovered, partly due to the different behavior of the algebra \( \R_\tau \) compared to complex numbers. 
\\ \\
The main objective of this article is to show how para-complex and pseudo-Riemannian geometry find an application in the study of representations of fundamental groups of closed surfaces into \(\SL(2m+1,\R)\), although this does not preserve any bilinear form in $\mathbb{R}^{2m+1}$. More generally, this problem falls within the framework of \emph{higher Teichmüller theory} (\cite{Wienhard_intro}), which has developed following Hitchin’s seminal work on the existence of a special connected component in the representation space (\cite{hitchin1992lie}). Introducing a hyperbolic space \(\h_\tau^{2m}\) defined over the algebra \(\R_\tau\), and naturally endowed with a para-Kähler metric, one can compute its isometry group and prove it is isomorphic to \(\SL(2m+1,\R)\). Through the use of stable $\SL(2m+1,\R)$-Higgs bundles, we can establish the existence of certain \(\rho\)-equivariant surface immersions in \(\h_\tau^{2m}\), where \(\rho\) is a representation into \(\SL(2m+1,\R)\) that is not necessarily Hitchin. 
Moreover, a detailed analysis of the properties of these immersions allows us to derive geometric insights into their moduli space and, more broadly, into the representation space. Interestingly enough, the holomorphic differential \( q_{2m+1} \) of order \( 2m+1 \), naturally associated with the cyclic Higgs bundle — not necessarily of Hitchin type — admits a geometric interpretation in terms of these immersions. More generally, we show that there exists a method to derive such a holomorphic differential from minimal maps between a Riemann surface and a para-Kähler manifold with constant para-holomorphic sectional curvature. Finally, in the case \(m=1\), we recover the Labourie-Loftin parametrization of the Hitchin component (\cite{loftin2001affine, Labourie_cubic}) and construct geometric structures on the projectivised tangent bundle of the surface modeled on the flag variety $\mathcal{F}$ given by pairs of lines and planes $(l,H)$ in \( \mathbb{R}^3 \) with $l\subset H$, which we show to be precisely those found by Guichard-Wienhard (\cite{guichard2012anosov}) using the theory of domains of discontinuity for Anosov representations. \\ 

\noindent We believe that this setting may prove itself useful for the study of Anosov representations of more general discrete groups in $\SL(2m+1,\R)$ with respect to the parabolic subgroup that fixes a line and a hyperplane or for the development of a universal higher Teichm\"uller theory for $\SL(2m+1,\R)$. Indeed, the Anosov property furnishes a boundary map with values in $\mathcal{F}$, which is exactly the boundary at infinity of the para-complex hyperbolic space mentioned previously. It is conceivable that an asymptotic Plateau problem for these kinds of boundary data, not necessarily equivariant under a group action, may be solved, thus obtaining new examples of higher higher Teichm\"uller spaces, in the spirit of \cite{seppi2023complete, beyrer2023mathbb}.

\subsection{Immersions into a para-K\"ahler space}
A \emph{para-complex} structure on a smooth manifold \( M \) of real dimension \( 2n \) is an endomorphism \( \p \) of the tangent bundle such that:  
\begin{enumerate}  
\item[(i)] \( \p^2 = \mathrm{Id} \)  
\item[(ii)] The two eigendistributions \( D_{\pm} := \Ker\big(\p \mp \mathrm{Id}\big) \) both have rank \( n \) and are involutive.  
\end{enumerate}  
The pair $(M,\p)$ will be called a \emph{para-complex manifold} of para-complex dimension $n$. Moreover, if there exists a non-degenerate symmetric bilinear two-tensor \( g \) on \( (M,\p) \) such that  
\[
g(\p X, \p Y) = -g(X,Y), \quad \text{for any} \ X,Y\in\Gamma(TM)
\]  
then \( g \) defines a pseudo-Riemannian metric on \( M \) of signature \( (n, n) \), and the triple \( (M, \p, g) \) is called a \emph{para-Hermitian manifold}. It is easily verified that the pairing  
\[
\omega := g(\cdot, \p \cdot)
\]  
defines a 2-form on \( M \), known as the fundamental 2-form of \( (M,\p, g) \). If, in addition, \( \omega \) is a symplectic 2-form, i.e., \( \mathrm{d}\omega = 0 \), then the manifold \( (M,\p, g, \omega) \) is called \emph{para-Kähler}. \\

\noindent The para-Kähler homogeneous space we wish to introduce can be defined as a hyperbolic space over the algebra \( \mathbb{R}_\tau \): given a para-Hermitian bilinear form $\mathbf{q}$ on $\R^{n+1}_{\tau}$, the para-complex hyperbolic space is 
\[
    \h^{n}_{\tau}=\{ z \in \R^{n+1}_{\tau} \ | \ \mathbf{q}(z,z)=-1\} / \mathcal{U},
\]
where $\mathcal{U}$ is the subgroup of unitary para-complex numbers acting by scalar multiplication. There is a natural notion of boundary at infinity given by projective classes of $\mathbf{q}$-isotropic vectors. It can be shown that this space is homeomorphic to the partial flag manifold \(\mathcal{F}_{1,n}\), obtained as the quotient \(\SL(n+1,\R)/P_{1,n}\), where \(P_{1,n}\) is the stabilizer of a pair of line $\ell$ and a hyperplane $H$ in $\mathbb{R}^{n+1}$ with $\ell \subset H$. It is easy to see — by analogy with the complex case — that \( \mathbb{H}_\tau^n \) is a para-Kähler manifold of para-complex dimension \( n \). We will denote such structure by $(\p,g,\omega)$. Moreover, the group of para-holomorphic isometries of $\mathbb{H}^{n}_{\tau}$ is isomorphic to \( \mathrm{SL}(n+1, \mathbb{R}) \). \\ \\ To introduce special immersions of oriented and connected surfaces \( S \), we must restrict ourselves to the case where \( n = 2m \) with \( m \geq 1 \). In this setting, the isometry group of the space becomes \( \SL(2m+1,\R) \), and \( \h_\tau^{2m} \) is a para-Kähler manifold equipped with a pseudo-Riemannian metric of signature \( (2m,2m) \). The idea is to consider all immersions \( \sigma: S \to \h_\tau^{2m} \) whose pullback tangent bundle admits an orthogonal splitting  
\[
\sigma^*T\h_\tau^{2m} = \mathscr{L}_1 \oplus \cdots \oplus \mathscr{L}_{2m}
\]  
into real rank $2$ sub-bundles with $\mathscr{L}_1\cong TS$ and such that the following properties hold: \begin{itemize}
    \item[$\bullet$] The pull-back metric $\sigma^*g$ restricted to \( \mathscr{L}_{i} \) is non-degenerate and positive-definite when $i\in\{1,\dots,2m\}$ is odd; \item[$\bullet$] the pull-back of the Levi-Civita connection admits a decomposition 
$$
\nabla=\begin{psmallmatrix}
    \nabla^1 & -\eta_2^\dag &  \\ \eta_2 & \nabla^2 & -\eta_3^\dag && \\ & \eta_3 & \nabla^3 & \ddots \\ && \ddots & \ddots & -\eta_{2m}^\dag \\ &&& \eta_{2m} &\nabla^{2m} \end{psmallmatrix}$$
where $\nabla^i$ is a metric connection on $\mathscr{L}_{i}$, $\eta_i$ is a $1$-form with values in $\Hom(\mathscr{L}_{i-1},\mathscr{L}_i)$ and $\eta_i^\dag$ is its adjoint. Moreover, we assume $\eta_{i}(X):\mathscr{L}_{i-1} \rightarrow \mathscr{L}_{i}$ to be conformal for all $X \in \Gamma(TS)$ and for all $i=2, \dots, 2m$; 

\item[$\bullet$] the pull-back symplectic form $\sigma^*\omega$ vanishes when restricted to each $\mathscr{L}_{i}$;
\item[$\bullet$] the para-complex structure \( \p \) alternates the rank-2 subbundles, i.e.,  
\[
\p(\mathscr{L}_{2k-1}) \cong \mathscr{L}_{2m-2(k-1)},
\]  
so that \( \sigma^*g \) is negative definite on \( \mathscr{L}_{2k} \), for $k=1,\dots,m$.
\end{itemize} We will refer to the newly introduced immersions as \emph{isotropic \(\p\)-alternating} surfaces (Definition \ref{def:Frenet} and \ref{def:P_alternating_surfaces}), and they can actually be defined in any para-Kähler manifold of real dimension \(4m\). A similar approach has been used by Nie (\cite{Nie_alternating}) in the study of \emph{A-surfaces} in the \emph{pseudo-hyperbolic space} of signature \( (n,n) \) for even \( n \), or signature \( (n+1,n-1) \) for odd \( n \). Instead, Collier-Toulisse (\cite{collier2023holomorphic}) studied \emph{alternating pseudo-holomorphic curves} in the pseudo-hyperbolic space of signature $(4,2)$. We will explain the relationship with their case in the final paragraph of the introduction. \\ \\ The first result concerns the structure equations of the immersions that we will study in this article, which will be directly linked to the Hitchin equations of a certain Higgs bundle over the associated Riemann surface induced by the conformal class of the first fundamental form. 
\begin{manualtheorem}A(Theorem \ref{thm:hol_data} and Theorem \ref{thm:structure_eq_hol})\label{thm:A}
 \emph{Let $\sigma:S\to\h_\tau^{2m}$ be a smooth isotropic $\p$-alternating immersion with $m\ge 2$. Let us assume that $\sigma(S)$ is not contained in any para-complex hyperbolic subspace of real codimension $4$, then its structural data is completely determined by the datum of $m$ holomorphic $(1,0)$-forms $\eta_2,\dots,\eta_{m+1}$ with values in $\Hom_\C(L_{i-1},L_i)$ for $i=2,\dots,m$ and in $\Hom_\C(L_m,\overline{L_m})$ for $i=m+1$, where $L_1,\dots,L_m$ are holomorphic line bundles endowed with Hermitian metrics $\boldsymbol h_1,\dots,\boldsymbol h_m$ satisfying the following system of equations:}
 \begin{equation}
    \begin{cases}
        \partial^{2}_{z\bar{z}}\log\boldsymbol h_{m}-|\eta_{m}|^{2}\boldsymbol h_{m}\boldsymbol h_{m-1}^{-1}+|\eta_{m+1}|^{2}\boldsymbol h_{m}^{-2}=0 \\
        \partial^{2}_{z\bar{z}}\log\boldsymbol h_{j}-|\eta_{j}|^{2}\boldsymbol h_{j}\boldsymbol h_{j-1}^{-1}+|\eta_{j+1}|^{2}\boldsymbol h_{j+1}\boldsymbol h_{j}^{-1}=0, \ \ \ \ \ \ \ j=2, \dots, m-1 \\
        \partial^{2}_{z\bar{z}}\log\boldsymbol h_{1}-\boldsymbol h_{1}+|\eta_{2}|^{2}\boldsymbol h_{2}\boldsymbol h_{1}^{-1} =0 .
    \end{cases}
\end{equation}
\end{manualtheorem}
\noindent It is important to observe that the connection \( \nabla \) introduced above is determined by \( 2m \) vector bundle-valued 1-forms \( \eta_1,\dots,\eta_{2m} \). However, in the previously stated theorem, it follows that only \( m+1 \) of them are sufficient to express the structure equations of the immersions. This is a consequence of the \( \p \)-alternating property that we have imposed, which is made possible by the existence of the para-complex structure on \( \h_\tau^{2m} \). This fact will be crucial in establishing a connection between \( \p \)-alternating isotropic immersions and cyclic $\SL(2m+1,\R)$-Higgs bundles.
\subsection{The role of cyclic $\SL(2m+1,\R)$-Higgs bundles}
Given a fixed Riemann surface $X$ of genus $g\ge 2$ an $\SL(n,\C)$-Higgs bundle is a pair $(E,\phi)$ where $E\to X$ is a holomorphic bundle of rank $n$ and $\phi$ is a holomorphic $(1,0)$-form with values in trace-less endomorphisms bundle of $E$. In this article, we restrict ourselves to the case where the following decomposition holds:  
\[
E = L_{m}^{-1} \oplus \cdots \oplus L_{1}^{-1} \oplus \mathcal{O}_{X} \oplus L_{1} \oplus \cdots \oplus L_{m}, \quad
   \phi=\begin{psmallmatrix}
    0 & & & & & & & &\gamma_{m}\\
    \gamma_{m-1} & 0 & & & & & & & \\
    &\ddots & \ddots & & & & & &\\
    & & \gamma_{1} & 0 & & & & &\\
    & & & \mu & 0 & & & &\\
    & & & & \mu & 0 & & &\\
    & & & & & \gamma_{1} & 0 & & \\
    & & & & & & \ddots & \ddots & \\
    & & & & & & & \gamma_{m-1} & 0
    \end{psmallmatrix} \]where $L_i$ are holomorphic line bundles, $0 \neq \gamma_{i} \in H^{0}(X, L_{i}^{-1}L_{i+1}K)$ for $i=1, \dots, m-1$, $0 \neq \mu \in H^{0}(X, L_{1}K)$ and $\gamma_{m} \in H^{0}(X, L_{m}^{-2}K)$. In the literature, such pairs \((E, \phi)\) are referred to as cyclic \(\SL(2m+1,\R)\)-Higgs bundles. There is a notion of stability condition (\cite{hitchin1987self}) that ensures the existence and uniqueness of the solution to the Hitchin equations, which corresponds to a unique Hermitian metric \(H\) on \(E\) compatible with the splitting into line bundles (\cite{Collier_thesis}), namely:  
\(
H = \operatorname{diag}(h_{m}^{-1}, \dots, h_{1}^{-1}, 1, h_{1}, \dots, h_{m}),
\) 
where \(h_{i}\) is a Hermitian metric on \(L_{i}\) for each \(i = 1, \dots, m\). Thanks to the work of Baraglia (\cite[\S 3.4.2]{baraglia2010g2}), it is known that for \(m=1\), Higgs bundles of the form  
\[
E = K \oplus \mathcal{O}_X \oplus K^{-1}, \qquad \phi = \begin{pmatrix} 0 & 0 & q_3 \\ 1 & 0 & 0 \\ 0 & 1 & 0 \end{pmatrix}
\]
give rise to hyperbolic affine spheres in \(\mathbb{R}^3\), whose affine invariant, the Pick form, is closely related to the holomorphic cubic differential \(q_3\). As soon as \(m \geq 2\), the geometry associated with cyclic \(\SL(2m+1,\R)\)-Higgs bundles remains largely unexplored, as it is not clear how to construct affine immersions in higher codimension. The answer is provided by para-complex geometry, or more precisely by the space \(\h_\tau^{2m}\), as suggested by the following result.
\begin{manualtheorem}B (Theorem \ref{thm:existence})\label{thm:B}
\emph{Let $(E,\phi)$ be a stable cyclic $\mathrm{SL}(2m+1,\R)$-Higgs bundle over $X$ with holonomy $\rho$. Assume further that $L_{1}=K^{-1}$ and $\mu=1$. Then there exists a $\rho$-equivariant isotropic $\mathbf{P}$-alternating surface $\sigma:\widetilde{X} \rightarrow \mathbb{H}^{2m}_{\tau}$. Moreover, the structural data of $\sigma$ is exactly the holomorphic data $(L_{1}, \dots, L_{m}, \gamma_{1}, \dots, \gamma_{m})$ used to construct the Higgs bundle along with the harmonic metric $h_{i}$ on each $L_{i}$.}
\end{manualtheorem}
\noindent The hypothesis that \( L_1 \cong K^{-1} \) has been included to ensure a well-defined Riemannian metric on the immersed surface. The statement remains valid even without this assumption, provided that the Higgs bundle is stable. At the end of Section \ref{sec:moduli_space_fixed_X}, we study the case in which the Higgs bundle satisfies \(\gamma_{m-1} \equiv 0\) and is therefore strictly polystable. Since solutions to the Hitchin equations still exist, one obtains isotropic \(\p\)-alternating surfaces that are equivariant with respect to reducible representations, to which we are able to assign a geometric interpretation in terms of the immersion.
\\ \\ The main idea behind Theorem \ref{thm:B} has been the \emph{para-complexification} of a Higgs bundle, a construction that, as we will explain, holds in a fairly general setting. In short, starting from an \(\SL(n,\R)\)-Higgs bundle \((E,\phi,Q)\), one can construct a new triple \((E^\tau,\phi^\tau, \p)\), where \(E^\tau\) now has rank \(2n\), \(\phi^\tau = \phi \otimes e_+ - \phi \otimes e_-\), and \(\p\) is a para-complex structure on \(E^\tau\) (in addition to having a natural holomorphic structure induced by \(E\)), where \(e_\pm\) is the basis of idempotents of the algebra \(\R_\tau\). In the \(\SL(2m+1,\R)\)-cyclic case, we are able to construct a \(\rho\)-equivariant map \(\sigma:\widetilde{X} \to \h_\tau^{2m}\) by using parallel transport on the para-complexification of the trivial section inside $E^\tau$, which comes from the presence of \(\mathcal{O}_X\) in the decomposition of \(E\). Moreover, the geometry of $\h_\tau^{2m}$ allows us to provide a geometric interpretation for the holomorphic differential \( q_{2m+1} = \mu (\gamma_1 \dots \gamma_{m-1})^{2} \gamma_m \) of order \( 2m+1 \), which is naturally associated with any such cyclic Higgs bundle (see the end of Section \ref{sec:existence_surfaces}). This result is a special case of a more general theorem, presented in Appendix \ref{sec:appendix}, which applies to minimal maps from a Riemann surface into a para-Kähler manifold with constant para-holomorphic sectional curvature. In analogy with the complex case (\cite{wood1984holomorphic}), it is possible to define an \emph{isotropic order} for the minimal map and obtain an explicit definition of a holomorphic differential on the Riemann surface, whose order depends on the isotropic order of the map.  \\ \\ Restricting to the Higgs bundles such that \( L_i^{-1} \cong K^i \) and \( \gamma_{i-1} \equiv 1 \) for \( i = 1, \dots, m-1 \), their stability can be controlled using \( d := \deg(L_m^{-1}) \in \mathbb{Z} \), where \( L_m \) is the unique line bundle that is not isomorphic to a power of the canonical bundle (Lemma \ref{lem:stability_Higgs_bundles}). In particular, denoting by \( \mathcal{ML}_d(X)_m \) the moduli space of isotropic \(\p\)-alternating immersions constructed from the Higgs bundles just described with fixed degree $d\in\Z$, we obtain \begin{manualtheorem}C (Theorem \ref{prop:parameterization_moduli_space})
\emph{For any $m\ge 2$, \begin{enumerate}
    \item[(i)] if $0<d\le m(2g-2)$ the set $\mathcal{ML}_d(X)_m$ is biholomorphic to the total space of a holomorphic vector bundle of complex rank $2d+g-1$ over the $(m(2g-2)-d)$-th symmetric product of $X$;
    \item[(ii)] if $1-g\le d\le 0$ the set $\mathcal{ML}_d(X)_m$ is biholomorphic to a bundle over an $H^{1}(X,\Z_{2})$-cover of the $(2d+2g-2)$-symmetric product of $X$ whose fiber is $(\C^{(2m-1)(g-1)-d} \setminus \{0\}) / \{\pm \mathrm{Id}\}$.
    \item[(iii)] if $d\notin[1-g,\dots,m(2g-2)]\cap\Z$ then $\mathcal{ML}_d(X)_m$ is empty.
\end{enumerate}Moreover, the space $\mathcal{ML}_d(X)_m$ is smooth for any $d\in[1-g,m(2g-2)]\cap\Z$.}
\end{manualtheorem}
\noindent The condition imposed on the Higgs bundles in the previously stated theorem is always satisfied when \( m = 2 \). Moreover, since the stability condition can be controlled by examining the degrees of the various line bundles that constitute the decomposition of \( E \) (\cite[\S 2.5]{Dai_Li}), it is plausible that a similar result could be obtained in the general case, although the description of the moduli spaces as bundles would become significantly more intricate.
\subsection{Higher rank Teichm\"uller theory}
In the early 1990s, a fundamental result proved by Hitchin (\cite{hitchin1992lie}) initiated the study of representations of surface groups of compact, oriented surfaces $\Sigma$ of genus \(g \geq 2\) into real split semisimple Lie groups \(G\) of rank greater than or equal to 2. Classical theory deals with discrete and faithful representations into \(\PSL(2,\R)\) and is directly connected to the Teichmüller space \(\mathcal{T}(\Sigma)\), which represents the deformation space of hyperbolic structures defined on the surface, and it is diffeomorphic to a connected component of the space of such representations. The main theorem proved by Hitchin implies that, in the case of rank \(r \geq 2\), there exists a connected component $\mathrm{Hit}(\Sigma,G)$ in the representation space that contains \(\mathcal{T}(\Sigma)\) and is diffeomorphic to an open ball of dimension \((2g-2)\dim G\), which is nowadays referred to as the \emph{Hitchin component}. Despite the very general nature of the result, a geometric description of such representations had not been obtained, due to the differential-geometric and holomorphic nature of the tools used by Hitchin. As previously explained, we are interested in the case where \( G = \SL(2m+1,\R) \), whose representation space will be denoted by \( \chi_m(\Sigma) \), and the Hitchin component by \( \mathrm{Hit}_m(\Sigma) \). When \( m=1 \), it was shown that any $\rho\in\mathrm{Hit}_1(\Sigma)$ preserves a properly convex set \( \mathcal{C} \) in \( \R\mathbb P^2 \), and the quotient \( \mathcal{C}/\rho(\pi_1(\Sigma)) \) induces a projective structure on the surface (\cite{goldman1990convex}). Moreover, the deformation space of all such convex structures is diffeomorphic to $\mathrm{Hit}_1(\Sigma)$ (\cite{kuiper1953convex, choi1997classification}).  Later, it was established that every Hitchin representation in $\SL(3,\R)$ preserves a special convex immersion \( f:\widetilde\Sigma\to\R^3 \) known as \emph{hyperbolic affine sphere}, which allowed, among other things, the holomorphic parametrization of $\mathrm{Hit}_1(\Sigma)$ as a holomorphic bundle over $\mathcal{T}(\Sigma)$ (\cite{loftin2001affine, Labourie_cubic}). As soon as \( m \geq 2 \), no geometric description is known for Hitchin representations in terms of convex structures or equivariant surfaces. Our work aims to bridge this gap by using equivariant isotropic \( \p \)-alternating immersions in \( \h_\tau^{2m} \), which, as implied by Theorem \ref{thm:B}, exist both for cyclic Hitchin representations and for representations in the other two connected components of the representation space \( \chi_m(\Sigma) \) (\cite{hitchin1992lie}). Let us denote by \( \widetilde{\mathcal{ML}}(\Sigma)_m \) the moduli space of such equivariant surfaces, which naturally carries a fibre bundle structure over the Teichmüller space, given by  
\begin{align*}
\pi: \widetilde{\mathcal{ML}}&(\Sigma)_m \to \mathcal{T}(\Sigma) \\ & [\sigma,\rho]\longmapsto[J]
\end{align*}
where \( J \) is the complex structure on \( \Sigma \) induced by the first fundamental form. An additional technical hypothesis is required on the isotropic \( \p \)-alternating surfaces in order to obtain results on the representation space. From Theorem \ref{thm:A}, we know that \( \eta_2, \dots, \eta_{m+1} \) can be regarded as holomorphic \((1,0)\)-forms with values in a holomorphic vector bundle. Let us denote by \( \widetilde{\mathcal{ML}}(\Sigma)_m^* \) the subspace consisting of all immersions in \( \widetilde{\mathcal{ML}}(\Sigma)_m \) that satisfy  
\[
(\eta_2) \prec (\eta_3) \prec \dots \prec (\eta_{m+1}),
\]  
where \( (\eta_i) \) represents the divisor of the holomorphic \((1,0)\)-form.
\begin{manualtheorem}D (Thereom \ref{thm:holonomy} and Corollary \ref{cor:immersion})\label{thm:D}
\emph{The holonomy map \begin{align*}\mathrm{Hol}:\widetilde{\mathcal{ML}}&(\Sigma)_{m}^{*}\longrightarrow\chi_m(\Sigma) \\ &[\sigma,\rho]\longmapsto [\rho]\end{align*} is an immersion for any $m\ge 1$.}
\end{manualtheorem}
\noindent The hypothesis on divisors, also used by Nie for A-surfaces (\cite{Nie_alternating}), allows us to establish an infinitesimal rigidity property for isotropic \( \p \)-alternating immersions, which is always satisfied in the Hitchin case. Theorem \ref{thm:D} was previously proven by Labourie (\cite{labourie2017cyclic}) using the theory of cyclic surfaces, but only in the case of representations in \( \mathrm{Hit}_m(\Sigma) \). In particular, our proof simplifies his argument by providing a geometric description of the image of the holonomy map and extends the result to the other two connected components of the representation space. In the case \( m=1 \), we recover the holomorphic parametrization of the Hitchin component for \( \SL(3,\R) \) by Labourie-Loftin (\cite{loftin2001affine, Labourie_cubic}) using only the isotropic \( \p \)-alternating immersions in \( \h_\tau^2 \). In this case, these immersions correspond to maximal space-like Lagrangian surfaces, which are in fact equivalent to hyperbolic affine spheres in \( \R^3 \) (see \cite{RT_bicomplex} and \cite{hildebrand2011cross}). \\ \\ Focusing on the case \( m=1 \), the theory of Anosov representations allowed Guichard and Wienhard (\cite{guichard2012anosov}) to construct, for instance, from a Hitchin representation \( \rho:\pi_1(\Sigma)\to\SL(3,\R) \), a domain of discontinuity \( \Omega_\rho \) inside the unique flag variety \( \mathcal{F} \). In particular, the group \( \Gamma:=\rho\big(\pi_1(\Sigma)\big) \) acts properly discontinuously and cocompactly on \( \Omega_\rho \), such that the quotient \( \Omega_\rho/\Gamma \) is homeomorphic to \( \mathbb{P}T\Sigma \), thereby endowing it with a \(\big(\mathcal{F}, \SL(3,\R)\big)\)-geometric structure. The final result concerns the construction of geometric structures using maximal space-like Lagrangian surfaces in \( \h_\tau^2 \), which turn out to be precisely those found by Guichard and Wienhard. \begin{manualtheorem}E (Theorem \ref{thm:geometric_structure})
\emph{Given a maximal space-like Lagrangian immersion \( \sigma:\widetilde{\Sigma} \to \h_\tau^2 \) equivariant with respect to a Hitchin representation \( \rho:\pi_1(\Sigma) \to \SL(3,\R) \), there exists a \((\mathcal{F}, \SL(3,\R))\)-geometric structure on \( \mathbb{P}T\Sigma \) that coincides with the one found by Guichard-Wienhard.}
\end{manualtheorem}

\noindent The approach is based on the existence of a transverse section of a certain flat bundle with holonomy \( \rho \) over the projectivized tangent bundle of the surface. In particular, this section takes values in the subbundle of $\mathbf{q}$-isotropic vectors which, after projectivize over the algebra \( \mathbb{R}_\tau \), coincides with the boundary \( \partial_\infty \h_\tau^2 \) and thus with the flag variety \( \mathcal{F} \). Recently, Nolte and Riestenberg (\cite{nolte2024concave}) studied the domain \(\Omega_\rho\subset\mathcal F\) and described geometric structures they defined \emph{concave foliated flag structures} using entirely different techniques, which, however coincides with ours. It is highly likely that this approach using isotropic \(\p\)-alternating immersions can be generalized to the case \(m \ge 2\), thus constructing \((\partial_\infty\h_\tau^{2m}, \SL(2m+1,\R))\) geometric structures that coincide with those of Guichard-Wienhard. In this case, it could provide insight into the topology of the fiber of \(\Omega_\rho / \rho(\pi_1(\Sigma)) \to \Sigma\) for each \(\rho : \pi_1(\Sigma) \to \SL(2m+1,\R)\) Hitchin, which is generally a difficult problem (\cite{alessandrini2023fiber}).

\subsection{Minimal map in the symmetric space}
A fundamental result by Labourie (\cite{labourie2008cross}) states that for every Anosov representation \(\rho: \pi_1(\Sigma) \to G\), where \(G\) is a real semisimple Lie group, there exists a \(\rho\)-equivariant map \(f: \widetilde{\Sigma} \to \mathbb{X}_G\) into the symmetric space of \(G\) that is both conformal and harmonic, i.e., minimal. 
In our setting, we consider $G=\SL(2m+1,\R)$ and, after obtaining a new geometric model of the symmetric space \(\mathbb{X}_m := \SL(2m+1,\R)/\SO(2m+1)\) as the set of totally geodesic and Lagrangian subspaces of \(\h_\tau^{2m}\) for which the restriction of the pseudo-Riemannian metric is negative definite, we describe a way of constructing a $\rho$-equivariant minimal map \(f: \widetilde{\Sigma} \to \mathbb{X}_m\) as a Gauss map of an equivariant, isotropic, and \(\p\)-alternating immersion \(\sigma: \widetilde{X} \to \h_\tau^{2m}\):

\begin{manualtheorem}F (Theorem \ref{thm:Gauss_map})
\emph{The Gauss map of an isotropic $\mathbf{P}$-alternating immersion $\sigma:\widetilde{X} \rightarrow \mathbb{H}^{2m}_{\tau}$ is conformal and harmonic, thus it parameterizes a minimal surface in the symmetric space $\mathbb X_m$.}
\end{manualtheorem} \noindent When \( m=1 \), this result provides an alternative proof, different from the approach used by Labourie (\cite[\S 9.3]{Labourie_cubic}), that starting from a hyperbolic affine sphere in \( \mathbb{R}^3 \) equivariant for a Hitchin representation in \( \SL(3,\mathbb{R}) \), the induced map into the symmetric space \( \mathbb{X}_1 \) is both harmonic and conformal (see Example \ref{ex:maximal_Lagrangian}).

\subsection{Related projects}
As stated at the beginning of the introduction, Nie (\cite{Nie_alternating}) studied immersions, referred to as \emph{A-surfaces}, into the pseudo-hyperbolic space \(\mathbb{H}^{n,n}\) for even \(n\) or \(\mathbb{H}^{n+1,n-1}\) for odd \(n\), which are equivariant with respect to representations into \(\mathrm{SO}_0(n,n+1)\), naturally acting on these spaces by (anti)-isometries. He introduced the notion of Frenet splitting for a surface immersed in \(\mathbb{H}^{n,n}\) and imposed an alternating property concerning the signature of the pullback metric on each subbundle of the decomposition. Indeed, the space \(\mathbb{H}^{n,n}\) is endowed with a pseudo-Riemannian metric of constant negative sectional curvature, unlike our case, where we only have constant para-holomorphic sectional curvature. However, the presence of the para-complex structure \(\p\) allowed us to impose an alternating property using the action of \(\p\) on \(\sigma^*T\h_\tau^{2m}\), while the symplectic form ensured an isotropy property. Despite this, we are still able to establish an infinitesimal rigidity property, albeit with a slightly more intricate analysis due to the presence of non-constant curvature. \\ \\ Another recent work by Collier-Toulisse (\cite{collier2023holomorphic}) studied surfaces in \(\mathbb{H}^{4,2}\), called \emph{alternating pseudo-holomorphic curves}, which are equivariant with respect to representations into the split real form \(\mathrm G_2'\) of the complex exceptional group \(\mathrm G_2^\mathbb{C}\). This group \(\mathrm G_2'\) is considered as the isometry group of \(\mathbb{H}^{4,2}\), acting by biholomorphisms with respect to the almost-complex structure that can be introduced using the theory of split-octonions. These surfaces also satisfy an alternating property with respect to the signature of the pullback metric restricted to the tangent, normal, and bi-normal bundles of the surface.
Additionally, it is required the tangent bundle of the surface to be invariant under the action of the almost-complex structure. \\ \\ Finally, during the writing of this work, non-Hitchin representations in \(\mathrm{SL}(3, \mathbb{R})\) were studied by Bronstein-Davalo (\cite{bronstein2025anosov}), particularly those that can be obtained as deformations of Barbot’s representations (\cite{barbot2010three}) corresponding to the reducible copy of \(\mathrm{SL}(2, \mathbb{R})\). Although in our case the isotropic \(\p\)-alternating surfaces in \(\h_\tau^2\) are necessarily equivariant with respect to Hitchin representations in \(\mathrm{SL}(3, \mathbb{R})\), it would be interesting to investigate whether a similar construction can be carried out for the aforementioned representations.
\subsection*{Acknowledgement} 
N.R. is funded by the European Union (ERC, GENERATE, 101124349). Views and opinions expressed are however those of the author(s) only and do not necessarily reflect those of the European Union or the European Research Council Executive Agency. Neither the European Union nor the granting authority can be held responsible for them. A.T. acknowledges support from the National Science Foundation with grant DMS-2005551 and from the "Istituto Nazionale di Alta Matematica" with the INdAM project "Geometric limits in higher Teichm\"uller theory".

\section{Para-complex hyperbolic $n$-space}
\subsection{Para-complex numbers}\label{sec:2.1}
Let us recall the definition and some basic properties of para-complex numbers (also known as split-complex or hyperbolic numbers), initially introduced in \cite{cockle1849iii} and subsequently studied by many authors in various different contexts. \\ \\ The set of para-complex numbers, denoted with $\R_\tau$, is defined as the commutative algebra over $\R$ generated by $\{1,\tau\}$ where $\tau$ is a non-real number such that $\tau^2=1$. In particular, any number $z\in \R_\tau$ can be written as $z=x+\tau y$ for $x,y\in\R$. There is a notion of conjugate para-complex number $\bar z^{\tau}:=x-\tau y$ which permits to define the absolute value as $$\vl z \vl_\tau^2=z\bar z^{\tau}=x^2-y^2\in\R.$$ The main difference compared to complex numbers is that the absolute value $\vl\cdot\vl_\tau^2$ does not have a sign and can also be zero. In fact, it can be shown that para-complex numbers with zero absolute value are exclusively the zero divisors. Among these, the two numbers $$e_+:=\frac{1+\tau}{2},\qquad e_-:=\frac{1-\tau}{2},$$ known as idempotent units\footnote{In fact, an easy computation shows that $e_{\pm}\cdot e_\pm=e_\pm$}, enable the algebra $\R_\tau$ to be realized as a direct product: \begin{align*}
    &\qquad\quad\R_\tau\longrightarrow\R\oplus\R \\ & z=x+\tau y\longmapsto (x+y,x-y).
\end{align*}
In particular, since $e_+$ and $e_-$ are linearly independent, any arbitrary para-complex number $z=x+\tau y$ can also be expressed as $z=z^+e_++z^-e_-$, where $z^+:=x+y$ and $z^-:=x-y$. We conclude by noting that the set of invertible numbers with absolute value equal to one $$\mathcal{U}:=\{z\in\R_\tau \ | \ \vl z\vl_\tau^2=1\}$$ is equal, as a set, to two hyperbolas in the real plane, and therefore to two copies of $\R$ as a group. This stands in stark contrast to the complex case, where the counterpart of the group $\mathcal U$ defined above is a copy of $S^1$, and is therefore compact.

\subsection{The hyperboloid model}\label{sec:2.2}
In this section, we introduce the para-complex hyperbolic space in a manner similar to how complex hyperbolic space is treated in the literature. To the best of our knowledge, the definition we are about to present, appears for the first time in \cite{trettel2019families} (see also \cite[\S 2.2]{RT_bicomplex} for a generalization using bi-complex numbers).
 \\ \\ Let us consider the $2(n+1)$-dimensional real vector space $\R^{n+1}_\tau$ consisting of $(n+1)$-uples of para-complex numbers, endowed with the para-hermitian form $$\mathbf{q}(z,w):=z^tQ\bar w^{\tau},$$ where $z=(z_1,\dots,z_{n+1}), w=(w_1,\dots,w_{n+1})\in\R^{n+1}_\tau$ and 
 \begin{equation}\label{Q_bilinear}
    Q=\begin{pmatrix}
     & & & 1\\
     & & \iddots & \\
     & \iddots &  &\\
    1 & & & 
    \end{pmatrix} .
\end{equation}
 \begin{defi}\label{def:paracomplex_hyperbolic_space}
The \emph{para-complex hyperbolic space} of dimension $n$ is defined as $$\mathbb H^n_\tau:=\{z\in\R^{n+1}_\tau \ | \ \mathbf{q}(z,z)=-1 \}/_\sim $$ where $z\sim w$ if and only if there exists $\lambda\in\mathcal U$ (i.e. a para-complex number of absolute value equal to one) such that $w=\lambda z$.
 \end{defi} \noindent It can be easily deduced that it is a smooth real manifold of dimension $2n$ which is, in fact, endowed with a special structure that we will now describe. For any point $z\in\mathbb{H}_\tau^n$ one can identify the tangent space $T_z\mathbb H^n_\tau$ with the $\mathbf{q}$-orthogonal complement of $z$, so that \begin{equation}\label{eq:decomposition_R_tau}T_z\R_\tau^{n+1}=T_z\mathbb{H}_\tau^n\oplus\R_\tau z.\end{equation} We deduce that $T_z\mathbb{H}_\tau^n$ can be naturally identified with $\R_\tau^n$ and thus can be equipped with a para-complex structure 
 $\mathbf{P}$ corresponding to the multiplication by $\tau$ on each entry. Moreover, the restriction of $\Ree(\mathbf q)$ on $T_z\h^n_\tau$ defines a pseudo-Riemannian metric $g$ of neutral signature $(n,n)$ so that $g(\p\cdot,\p\cdot)=-g(\cdot,\cdot)$. Therefore, we get an induced symplectic form $\omega=g(\cdot,\mathbf P\cdot)$ giving rise to a \emph{para-K\"ahler} metric on $\h^n_\tau$. A priori, the endomorphism $\p$ of the tangent bundle defines an \emph{almost} para-complex structure; however, as we will show in the next section through a symplectic reduction process, $\p$ is actually integrable, allowing us to refer to a para-Kähler structure on $\h_\tau^n$, which has para-complex dimension equal to $n$. In particular, such a structure is entirely determined by the para-Hermitian form $\mathbf q$ through the decomposition $\mathbf q(\cdot,\cdot)=g(\cdot,\cdot)+\tau\omega(\cdot,\cdot)$. \\ \\ Similarly to what occurs in the context of complex hyperbolic space, we can consider the group formed by all $(n+1)\times(n+1)$ matrices with para-complex entries that preserve the para-Hermitian form $\mathbf q$. More precisely, the unitary group $$\mathrm U(n+1,\R_\tau,Q):=\{A\in\GL(n+1,\R_\tau) \ | \ \mathbf q\big(A\cdot z,A\cdot w\big)=\mathbf q(z,w), \ \text{for all} \ z,w\in\R^{n+1}_\tau\}$$ acts by isometry on $\h_\tau^n$. In the next proposition, we will provide an alternative description of the group just introduced, which will be useful for another model of para-complex hyperbolic space that we will present later. A similar statement can be found in \cite{RT_bicomplex}, where it was carried out for $n=2$ and using bi-complex numbers. Since the proof follows the same idea with no significant modifications, it will not be included.
 \begin{prop}[{\cite[Proposition 2.3]{RT_bicomplex}}]\label{prop:iso_group}
 The map \begin{align*}\Psi:\GL(n&+1,\R)\longrightarrow\mathrm U(n+1,\R_\tau,Q) \\ & A\longmapsto Ae_++Q(A^{-1})^tQe_- \end{align*} is a group isomorphism. Moreover, the subgroup $\mathrm{SU}(n+1,\R_\tau,Q)$ of unitary matrices with determinant equal to one is isomorphic to $\SL(n+1,\R)$.
 \end{prop}


\begin{remark}
The unitary groups $\SU(n,\R_{\tau}, Q)$ defined using different para-Hermitian forms $Q$ are, in fact, all conjugate, a phenomenon similar to what occurs with orthogonal groups over complex numbers. Indeed, the Gram-Schmidt algorithm to find an orthogonal basis of $\R^{n}_{\tau}$ still applies and we can always renormalize to make it orthonormal, since multiplication by $\tau$ changes the sign of the norm of a vector. For this reason, from now on, we will remove $Q$ from the notation of the unitary groups $\SU(n,\R_{\tau}, Q)$.
\end{remark}

\noindent The action of $\SU(n+1,\R_\tau)$ on $\h^n_\tau$ is isometrically transitive and the stabilizer is given by $\mathrm{S}\big(\U(n,\R_\tau)\times\R_\tau^\ast\big)$, where $\R_\tau^\ast:=\{z\in\R_\tau \ | \ \vl z\vl_\tau^2\neq 0\}$ is isomorphic to $\mathrm U(1,\R_\tau)$. In other terms, the para-complex hyperbolic space has a description as homogeneous space $$\h_\tau^n\cong\bigslant{\SU(n,1,\R_\tau)}{\mathrm{S}\big(\U(n,\R_\tau)\times\R_\tau^\ast\big)}$$ or, according to Proposition \ref{prop:iso_group}, as \begin{equation}\label{Htau_homogeneous}\h_\tau^n\cong\bigslant{\SL(n+1,\R)}{\mathrm{S}\big(\GL(n,\R)\times\R^\ast\big)}.\end{equation}
The latter description will be directly related to the incidence geometry in $\R^{n+1}$ which we will explain in Section \ref{sec:2.4}.
\subsection{The para-K\"ahler metric}\label{sec:2.3}
The aim of this section is to present a construction of the para-K\"ahler metric on $\h_\tau^n$ and study some of its properties (see \cite[\S 2.5]{RT_bicomplex} for a similar case). \\ \\ Recall that on $\R_\tau^{n+1}$ we introduced a para-Hermitian form $\mathbf q$ such that $\Ree(\mathbf q)$ and $\Ima(\mathbf q)$ define, respectively, a pseudo-Riemannian metric $\hat g$ of signature $(n+1,n+1)$ such that $\hat g(\p\cdot,\p\cdot)=-\hat g$ and a symplectic form $\hat \omega=\hat g(\cdot,\p\cdot)$. Let us define the function $f:\R_\tau^{n+1}\to\R$ as $f(z):=\mathbf q(z,z)$, so that the action of $\mathcal{U}$, namely the group of invertible para-complex numbers with norm equal to one, preserves its level sets. In fact, if $u\in\mathcal U$ then $$f(uz)=\mathbf q(uz,uz)=u\bar u\mathbf q(z,z)=f(z).$$ Noting that any $u=u_1+\tau u_2\in\mathcal U$ satisfies $(u_1)^2-(u_2)^2=1$ we obtain the following characterization $$\mathcal{U}=\{ \pm \cosh{t}+\tau\sinh{t} \ | \ t\in\R\}.$$ \begin{prop}
The action of $\mathcal U$ is Hamiltonian with respect to $\hat\omega=\Ima{(\mathbf q)}$ and with moment map $\mu(z):=\frac{1}{2}f(z)$. In particular, there is an induced para-K\"ahler structure $\big(\p, g:={\hat g}|_{\h_n^\tau}, \omega:=\hat{\omega}|_{\h_\tau^n}\big)$ on $\h_\tau^n$.
\end{prop} \begin{proof}
As an initial observation, we note that the restriction of the pseudo-Riemannian metric $\hat g:=\Ree(\mathbf q)$ on the $\mathcal U$-orbit of the action on the non-zero level sets of $f$ is non-degenerate. Indeed, if $k\in\R\setminus\{0\}$ and $z\in f^{-1}(k)$ then $$\frac{\mathrm d}{\mathrm dt}\Big|_{t=0}\big( \pm \cosh{t}+\tau\sinh{t}\big)z=\tau z,$$ and $\hat g(\tau z,\tau z)=\Ree(\mathbf{q})(\tau z,\tau z)=-k\neq 0$. Hence, the infinitesimal generator of the action of $\mathcal U$ is the vector field $U(z):=\tau z$. Moreover, for any $v\in\R_\tau^{n+1}$ we have \begin{align*}
    \hat\omega(U(z),v)&=\Ima(\mathbf{q}(\tau z,v)) \\ &=\Ree(\mathbf{q}(z,v)) \\ &=\frac{1}{2}\frac{\mathrm d}{\mathrm ds }\Big|_{s=0}\mathbf q(z+sv,z+sv) \\ &=\mathrm d_z\mu(v).
\end{align*} Using the standard theory of symplectic reduction (\cite{marsden1974reduction, weinstein1980symplectic}) we obtain a symplectic form $\omega$ on the quotient space $f(-1)/\mathcal U$, which is identified with $\h_\tau^n$. Furthermore, since the pseudo-Riemannian metric $\hat g$ is non-degenerate on the $\mathcal U$-orbit and we have the decomposition $$T_z\R_\tau^{n+1}=T_z\h_\tau^n\oplus T(\mathcal U\cdot z)\oplus\tau T(\mathcal U\cdot z)$$it can be deduced, through identical computations as in the K\"ahler setting (see \cite[Theorem B.6]{rungi2023pseudo} for a detailed proof), that we obtain a para-complex structure $\p$ and a pseudo-Riemannian metric $g$ of signature $(n,n)$ on $\h_\tau^n$ such that $$\nabla^g_{\cdot}\p=0,\qquad g(\p\cdot,\p\cdot)=-g(\cdot,\cdot),\qquad \omega(\cdot,\cdot)=g(\cdot,\p\cdot),$$ where $\nabla^g$ is the Levi-Civita connection with respect to $g$.
\end{proof} 
\noindent Even in the pseudo-Riemannian case, there is a natural notion of curvature for the Levi-Civita connection $\nabla^g\equiv\nabla$ whose Riemann tensor is defined as: $$R^\nabla(X,Y)Z:=\nabla_X\nabla_YZ-\nabla_Y\nabla_XZ-\nabla_{[X,Y]}Z \qquad X,Y,Z\in\Gamma(T\h_\tau^n).$$ The \emph{sectional curvature} of $g$ is defined as $$\Sec(X,Y):=-\frac{R^\nabla(X,Y,X,Y)}{g(X,X)g(Y,Y)-g(X,Y)^2},$$ where $X,Y\in\Gamma(T\h_\tau^n)$ are vector fields that at each point span a non-degenerate $2$-plane for $g$. Furthermore, for every non-isotropic $X$, one can consider the $\p$-invariant $2$-plane generated by $X$ and $\p X$, and we will refer to the quantity $\Sec(X,\p X)$ as the \emph{para-holomorphic sectional curvature}.
\begin{prop}
The para-complex hyperbolic space $\h_\tau^n$ has constant para-holomorphic sectional curvature equal to $-2$.
\end{prop}
\begin{proof} See for instance \cite[Lemma 2.7]{RT_bicomplex}.
\end{proof}

\noindent We conclude the section by presenting a result that holds more generally for para-K\"ahler manifolds with constant para-holomorphic curvature since will be useful later in the article.
\begin{lemma}[\cite{bonome1998paraholomorphic}]\label{lem:para_hol_curvature}
Let $(M,\p,\g,\ome)$ be a para-K\"ahler manifold of constant para-holomorphic sectional curvature $\kappa$, then its Riemann tensor satisfies \begin{align*}
    R^\nabla(X,Y,Z,W)=&-\frac{\kappa}{4}\big(\g(X,Z)\g(Y,W)-\g(Y,Z)\g(X,W)+\g(X,\p Z)\g(\p Y,W) \\ &-\g(Y,\p Z)\g(\p X,W)+2\g(X,\p Y)\g(\p Z,W)\big).
\end{align*}
\end{lemma}
\begin{remark}
According to the above lemma and to simplify the calculations, from now on we will rescale the metric $g$ of $\h_\tau^n$ so that its para-holomorphic sectional curvature is equal to $-4$.  
\end{remark}

\subsection{The incidence geometry model}\label{sec:2.4}
We will now introduce a new model for $\h_\tau^n$ using incidence geometry. In particular, we will prove how the para-complex hyperbolic space can be interpreted as the space of projective classes of pairs given by a line and a transverse hyperplane in $\R^{n+1}$. This was initially observed in \cite[\S 8.5]{trettel2019families}. \\ \\
Recall that the idempotents $\{e_+,e_-\}$ introduced in Section \ref{sec:2.1} allow us to find an isomorphism between the algebra $\R_\tau$ and $\R e_+\oplus\R e_-$. In particular, for any $n\ge 2$ we get an induced isomorphism $\alpha:\R_\tau^{n+1}\to\R^{n+1}e_+\oplus\R^{n+1}e_-$ simply by applying the one described above component-wise. Let us identify the second copy of $\R^{n+1}$ in the direct sum with its dual by using the bilinear form $Q$ defined in (\ref{Q_bilinear}). In other terms, \begin{align*}
    (\R^{n+1})^{*} &\longleftrightarrow \R^{n+1} \\
        \varphi &\longleftrightarrow v_{\varphi} \ ,
\end{align*}
where $v_{\varphi}$ is the unique vector in $\R^{n+1}$ such that $\varphi(w)=v_{\varphi}^tQw$ 
for all $w\in \R^{n+1}$. \begin{prop}[{\cite{trettel2019families},\cite{RT_bicomplex}}]
The space $\h_\tau^n$ is diffeomorphic to the set $$\{(v,\varphi)\in \R\mathbb P^n\times(\R\mathbb P^n)^{*} \ | \ \varphi(v)\neq 0\},$$ which consists of pairs of projective classes of transverse lines and hyperplanes in $\R^{n+1}$.
\end{prop}
\noindent The above result is consistent, giving a geometric description, of the representation of $\h_\tau^n$ as a homogeneous space for $\SL(n+1,\R)$ (see (\ref{Htau_homogeneous})) that we obtained through isomorphisms of Lie groups. Moreover, the quadric $\widehat{\h}_\tau^n:=\{z\in\R_\tau^{n+1} \ | \ \mathbf q(z,z)=-1\}$, which has the structure of a principal $\mathcal{U}$-bundle over $\h_\tau^n$, is diffeomorphic to the set of pairs given by a vector and a transverse hyperplane in $\R^{n+1}$. Specifically, we have a diffeomorphism $$\widehat{\h}_\tau^n\cong\bigslant{\SL(n+1,\R)}{\SL(n,\R)}.$$
Finally, using the incidence geometry model, we can introduce a natural notion of boundary at infinity for $\h_\tau^n$. In summary, it is given by all possible "degenerate" configurations of a line and a hyperplane in $\R^{n+1}$. For dimensional reasons, this occurs only when the line is contained within the hyperplane. We thus define the boundary at infinity of the para-complex hyperbolic space as the set $$\partial_\infty\h_\tau^n:=\{(v,\varphi)\in \R\mathbb P^n\times(\R\mathbb P^n)^{*} \ | \ \varphi(v)=0\}.$$ It is a compact manifold of real dimension $2n-1$ and is diffeomorphic to the flag manifold $\SL(n+1,\R)/P_{1,n}$, where $P_{1,n}$ is the parabolic subgroup stabilizing the pair formed by a line transverse to a hyperplane. In the hyperboloid model it is given by \begin{equation}
    \partial_\infty\h_\tau^n=\{z\in\R_\tau^{n+1} \ | \ \boldsymbol{q}(z,z)=0\}/_{\sim}
\end{equation} where the equivalence relation is the same as Definition \ref{def:paracomplex_hyperbolic_space}.
\section{Superconformal surfaces in $\mathbb{H}^{2m}_{\tau}$}
\noindent The purpose of this section is to introduce the main geometric object of this work and to study some of its properties, which will be fundamental for its application to the representation space for $\SL(2m+1,\R)$. As explained in the introduction, a similar phenomenon was described by Nie (\cite{Nie_alternating}) for immersions into the pseudo-hyperbolic space. Indeed, our surfaces will be connected to cyclic $\SL(2m+1,\R)$-Higgs bundles, and it can be shown that Nie's surfaces correspond to the sub-cyclic ones according to the inclusion $\SO(m,m+1)<\SL(2m+1,\R)$.
\subsection{Definition and main properties}\label{sec:def_superconformal}\label{sec:definition_surfaces}
Let $S$ be a connected oriented surface and let $(M,\g,\p,\ome)$ be a para-K\"ahler manifold of para-complex dimension $2m\ge 2$. In particular, it can be seen as a pseudo-Riemannian manifold $(M,\g)$ of signature $(2m,2m)$. In the following, we will consider smooth space-like immersions $\sigma:S\longrightarrow M$, namely those whose first fundamental form $h|_{TS}:=\sigma^*\g|_{TS}$ is positive definite, so that defines a conformal structure on the immersed surface. Let us denote by $\nabla$ the pullback of the Levi-Civita connection of $(M,\g)$. We now recall the definition of Frenet splittings for surfaces in pseudo-Riemannian manifolds with neutral signature (\cite{Nie_alternating}). \begin{defi}\label{def:Frenet}
Let $\sigma: S\longrightarrow (M,\g)$ be a smooth spacelike immersion, then a \emph{Frenet splitting} of $\sigma^*TM$ is an orthogonal decomposition $\sigma^*TM=\mathscr{L}_1\oplus\mathscr{L}_2\oplus\cdots\oplus\mathscr{L}_{2m-1}\oplus\mathscr{L}_{2m}$ such that: \begin{enumerate}
    \item[(i)] Each $\mathscr{L}_i$ is a rank $2$ vector bundle over $S$ for $i=1,\dots,2m$ and $\mathscr{L}_1\cong TS$; \item[(ii)] The pull-back metric $h$ is positive or negative definite when restricted to $\mathscr{L}_i$ for $i=1,\dots,2m$; \item[(iii)] For any $X\in\Gamma(TS)$ and any $\nu\in\Gamma(\mathscr{L}_j)$ the projection of $\nabla_X\nu$ onto $\mathscr{L}_i$ is zero whenever $|i-j|\ge 2$.
\end{enumerate}
\end{defi}
\noindent Since the pull-back metric $h$ is parallel for $\nabla$, we obtain the following decomposition \begin{equation}\label{eq:nabla_decomposition}\renewcommand\arraystretch{1.5}
    \nabla=\begin{pmatrix}
        \nabla^1 & -\eta_2^\dag &  \\ \eta_2 & \nabla^2 & -\eta_3^\dag && \\ & \eta_3 & \nabla^3 & \ddots \\ && \ddots & \ddots & -\eta_{2m}^\dag
        \\ &&& \eta_{2m} &\nabla^{2m}
    \end{pmatrix} 
\end{equation}
where $\nabla^i$ is a connection on $\mathscr{L}_i$ compatible with $h|_{\mathscr L_i}$ and $\eta_i$ is a $1$-form with values in $\Hom(\mathscr{L}_{i-1},\mathscr{L}_i)$. Moreover, the element $\eta^\dag_i\in\Omega^1\big(\Hom(\mathscr{L}_{i},\mathscr{L}_{i-1})\big)$ is completely determined by the following relation $$h\big(\eta_i(X)\xi_{i-1},\xi_{i}\big)=h\big(\xi_{i-1},\eta_i^\dag(X)\xi_i\big),$$ for $X\in\Gamma(TS)$ and $\xi_{i-1},\xi_i$ smooth sections of $\mathscr{L}_{i-1},\mathscr{L}_i$. To be more precise, $\eta_{i}^{\dag}$ is the adjoint of $\eta_{i}$ hence is an element in $\Omega^{1}(\Hom(\mathscr{L}_{i}^{*}, \mathscr{L}_{i-1}^{*}))$ but we then identify $\mathscr{L}_{i}^{*}$ with $\mathscr{L}_{i}$ using the metric $h$.  \\

\noindent Recall that a linear map of the Euclidean plane is said to be conformal if it sends any circle centered at the origin to a circle (possibly reduced to $\{0\}$), or equivalently, if it is either a scaled rotation, a scaled reflection, or the zero map. A homomorphism between rank $2$ metric vector bundles
is said to be conformal if it is a conformal linear map on each fiber. 

\begin{defi}\label{def:superconformal} A smooth spacelike immersion $\sigma:S\rightarrow (M,\mathbf{g})$ as in Definition \ref{def:Frenet} is said to be \emph{superconformal} if $\eta_{i}(X):\mathscr{L}_{i-1} \rightarrow \mathscr{L}_{i}$ is conformal for all $X \in \Gamma(TS)$ and for all $i=2, \dots, 2m$.
\end{defi}

\noindent At this point, by using the additional structure on the ambient para-Kähler manifold, namely the symplectic form $\ome$ and the para-complex structure $\p$, we introduce a class of special surfaces that we will study in the following. \begin{defi}\label{def:P_alternating_surfaces}
A superconformal immersed surface $\sigma: S\longrightarrow (M,\g,\p,\ome)$ admitting a Frenet splitting is called an \emph{isotropic $\p$-alternating} surface if \begin{enumerate}
\item[(i)] the pull-back metric $h$ is positive-definite when restricted to $\mathscr{L}_i$ for $i$ odd.
    \item[(ii)] $\p(\mathscr{L}_{2k-1})\cong\mathscr{L}_{2m-2(k-1)}$ for $k=1,\dots,m$; \item[(iii)] the pull-back form $\sigma^*\ome$ is everywhere zero when restricted to $\mathscr{L}_i$.
\end{enumerate}
\begin{remark}\label{rem:metric_positive_negative_definite}
As a consequence of point $(i)$ and $(ii)$ in Definition \ref{def:P_alternating_surfaces} we get that the pull-back metric $h$ is negative-definite when restricted to $\mathscr{L}_{i}$ with $i$ even, being it isomorphic to $\p(\mathscr{L}_{2m-i+1})$. In fact, the para-K\"ahler condition implies $h(\p\cdot,\p\cdot)=-h(\cdot,\cdot)$.
\end{remark}
\end{defi}\noindent 
The adjective \emph{$\p$-alternating} is encapsulated in properties $(i)-(ii)$, and \emph{isotropic} in point $(iii)$. The $\p$-alternating property is outlined as follows: 

 
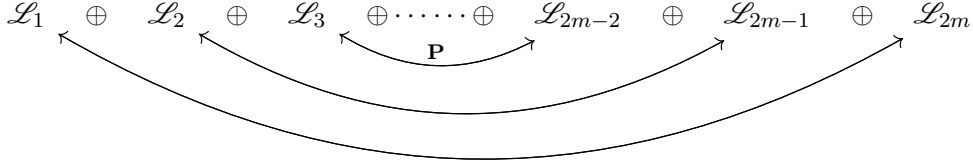
\[
\begin{tikzcd}[column sep=0.5em, row sep=1em]
 & \arrow[rrrrrrrrrr, bend right]\mathscr{L}_1  & \oplus & \arrow[rrrrrr, bend right]\mathscr{L}_2  & \oplus & \arrow[rr, bend right, "\p"]\mathscr{L}_3  & \oplus \cdots\cdots \oplus & \mathscr{L}_{2m-2} \arrow[ll, bend left] & \oplus &\mathscr{L}_{2m-1} \arrow[llllll, bend left] & \oplus & \mathscr{L}_{2m} \arrow[llllllllll, bend left]
\end{tikzcd}
\]    
\begin{remark}\label{rem:maximal_spacelike}
The smooth section $\eta_2$ is a $1$-form with values in $\Hom(TS,\mathscr{L}_2)$ and therefore coincides with the second fundamental form $\II$ of the immersion. In particular, requiring that $\eta_2(X):TS\to\mathscr{L}_2$ is conformal for any $X\in\Gamma(TS)$ implies that $\trace_h\II=0$. Thus any smooth superconformal immersion as in Definition \ref{def:superconformal} is automatically maximal (see also \cite[Remark 4.4]{Nie_alternating}). Moreover, the ones given in Definition \ref{def:P_alternating_surfaces} are spacelike by point (i). In light of this observation, and to simplify the notation, from now on we will refer to those surfaces as \emph{isotropic $\p$-alternating} immersions since it is understood that they are also superconformal, maximal, and space-like. 
\end{remark}
\noindent We now provide an example of a surface as in Definition \ref{def:P_alternating_surfaces}, which, on the one hand, illustrates a natural context in which they can be found, and, on the other, served as the motivating example that led to the general definition.
\begin{example}\label{ex:maximal_Lagrangian}
Let $\sigma:S\to(M,\g,\p,\ome)$ be a space-like maximal Lagrangian surface into a para-K\"ahler manifold of real dimension $4$. By definition, $h|_{TS}$ is positive-definite, $\trace_{h}\II=0$ and $\sigma^*\omega|_{TS}\equiv0$. The Lagrangian condition allows us to identify the normal bundle of the surface with $\p(TS)$, so that we get a $h$-orthogonal decomposition $$\sigma^*TM=TS\oplus\p(TS).$$ Moreover, using $h(\p\cdot,\p\cdot)=-h$ we conclude that $h$ is negative-definite when restricted to $\p(TS)$. It is now easy to deduce that $\sigma$ satisfies all the hypotheses of Definition \ref{def:P_alternating_surfaces}. If we choose the para-complex hyperbolic space $(\h_\tau^2,g,\p,\omega)$ as the ambient manifold, the space-like maximal Lagrangian surfaces correspond to hyperbolic affine spheres in $\R^3$ which are extensively studied in the framework of affine differential geometry (\cite{nomizu1994affine}). The correspondence was initially proved by Hildebrand in a slightly different context (\cite{hildebrand2011cross}), and later proved in (\cite[\S 4]{RT_bicomplex}) using directly the geometry of the para-complex hyperbolic space. It would be interesting to determine whether surfaces into $\h_\tau^{2m}$, with $m\ge 2$, as in Definition \ref{def:P_alternating_surfaces}, also have an equivalent description in terms of affine differential geometry in codimension greater than one (\cite{nomizu1986cubic}).
\end{example}
\begin{prop}\label{prop:some_properties}
Let $\sigma:S\longrightarrow (M,\g,\p,\ome)$ be an isotropic $\p$-alternating surface, then \begin{itemize}
    \item[(1)] the real vector bundle $\p(\mathscr L_k)$ can be identified with $\mathscr{L}_k^{*_{\ome}}$ using the symplectic form $\ome$ for $k=1,\dots,m$; \item[(2)] the $\Hom$-valued $1$-forms $\eta_2,\dots,\eta_{2m}$ appearing in the decomposition of $\nabla$ (see \ref{eq:nabla_decomposition}) satisfy the following relations: \begin{align*}
        &\eta_{2k}(X)\p\xi_{2m-2k+2}=-\p\Big(\eta_{2m-2k+2}^\dag(X)\xi_{2m-2k+2}\Big), \\ & -\eta_{2k-1}^\dag(X)\p\xi_{2m-2k+2}=\p\Big(\eta_{2m-2k+3}(X)\xi_{2m-2k+2}\Big),
    \end{align*}for $X\in\Gamma(TS), \xi_{2m-2k+2}\in\Gamma(\mathscr{L}_{2m-2k+2})$ and $k=1,\dots,m$. In particular, only $m$ of the $(\eta_i)_{i=1,\dots,2m}$ are not determined by the others and by the ambient para-K\"ahler structure.
\end{itemize}
\end{prop}
\begin{proof}
Regarding point $(1)$ we observe that the map \begin{align*}
    \p(&\mathscr{L}_k)\longrightarrow\mathscr{L}_k^* \\ & \p\xi_k\longmapsto\ome|_{\mathscr{L}_k}(\cdot,\p\xi_k)
\end{align*}induces a linear map on each fiber over a point $p\in S$. Furthermore, since $\g(\cdot,\cdot)=\ome(\cdot,\p\cdot)$ and the pseudo-Riemannian metric is non-degenerate whenever restricted to $\mathscr{L}_k$, we obtain that the induced map on each fiber in an isomorphism. \newline The claim in point $(2)$ is simply a calculation that we will explain as follows: let $k=1,\dots,m$ be fixed and let $\xi_{2m-2k+2}\equiv\xi$ be a smooth section of $\mathscr{L}_{2m-2k+2}$; then for any $X\in\Gamma(TS)$, according to the decomposition (\ref{eq:nabla_decomposition}), we have $$\nabla_X\p\xi=-\underbrace{\eta_{2k-1}^\dag(X)\p\xi}_{\in\mathscr{L}_{2k-2}}+\underbrace{\nabla^{2k-1}_X\p\xi}_{\in\mathscr{L}_{2k-1}}+\underbrace{\eta_{2k}(X)\p\xi}_{\in\mathscr{L}_{2k}}.$$On the other hand, $\p$ is parallel with respect to $\nabla$, hence $$\nabla_X\p\xi=\p\big(\nabla_X\xi\big)=-\underbrace{\p\big(\eta_{2m-2k+2}^\dag(X)\xi\big)}_{\in\p(\mathscr{L}_{2m-2k+1})}+\underbrace{\p\big(\nabla_X^{2m-2k+2}\xi\big)}_{\in\p(\mathscr{L}_{2m-2k+2})}+\underbrace{\p\big(\eta_{2m-2k+3}(X)\xi\big)}_{\in\p(\mathscr{L}_{2m-2k+3})}.$$ Using the $\p$-alternating property of Definition \ref{def:P_alternating_surfaces}, namely $\p(\mathscr{L}_{2k-1})\cong\mathscr{L}_{2m-2(k-1)}$ for any $k=1,\dots,m$, we get that $$\p(\mathscr{L}_{2m-2k+1})\cong\mathscr{L}_{2k},\quad\p(\mathscr{L}_{2m-2k+2})\cong\mathscr{L}_{2k-1},\quad\p(\mathscr{L}_{2m-2k+3})\cong\mathscr{L}_{2k-2}$$ and we conclude by comparing the two orthogonal decompositions above.
\end{proof}

\noindent As a consequence, the Frenet splitting of an isotropic $\mathbf{P}$-alternating immersion $\sigma:S \rightarrow (M,\mathbf{g},\mathbf{P},\ome)$ can be written as
\[
    \sigma^*TM\cong \mathscr{L}_1 \oplus \mathbf{P}(\mathscr{L}_{2m-1}) \oplus \mathscr{L}_{3} \oplus \cdots \oplus \p(\mathscr{L}_{3}) \oplus \mathscr L_{2m-1} \oplus\p(\mathscr L_1),\quad \mathscr L_1\cong TS
\]
and the connection $\nabla$ decomposes as
\begin{equation}\label{eq:nabla_decomposition2}\renewcommand\arraystretch{1.5}
    \nabla=\begin{pmatrix}
        \nabla^1 & -\eta_2^\dag &  \\ \eta_2 & \nabla^2 & -\eta_3^\dag && \\ & \eta_3 & \nabla^3 & \ddots \\ && \ddots & \ddots & -\eta_m^\dag
        \\ &&& \eta_{m} &\nabla^{m} & -\eta_{m+1}^\dag \\
         &&&& \eta_{m+1} & \nabla^{m} & \eta_m \\
         &&&&&-\eta_{m}^\dag & \nabla^{m-1} & \ddots
        \\ &&&&&& \ddots & \ddots & \eta_{2}
        \\ &&&&&&& -\eta_{2}^\dag &\nabla^{1}
    \end{pmatrix} 
\end{equation}
\begin{remark}
To obtain the Frenet splitting above, point $(iii)$ of Definition \ref{def:P_alternating_surfaces} has been used. Moreover, the new decomposition of the pull-back connection should be understood as follows: the \( m \times m \) block in the top left remains unchanged with respect to the decomposition in (\ref{eq:nabla_decomposition}), while the \( m \times m \) block in the bottom right has been rewritten using point $(2)$ of Proposition \ref{prop:some_properties}. For instance, the smooth section $\eta_{2m}$ has been replaced by $-\eta_2^\dag$ according to $$\eta_{2m}(X)\p\xi=-\p\big(\eta_2^\dag(X)\xi\big), \quad \xi\in\mathscr L_2$$ and the same applies to $\eta_{m+2},\dots,\eta_{2m-1}$. Finally, we have chosen not to include the action of the para-complex structure $\p$ to avoid complicating the notation.
\end{remark}
\subsection{Structure equations}\label{sec:3.2}
In the following section, we fix the ambient para-Kähler manifold $(M,\g,\p,\ome)$ to be the para-complex hyperbolic space $(\h_\tau^{2m}, g,\p,\omega)$ introduced in Section \ref{sec:2.2}. Let $\sigma:S\to\h_\tau^{2m}$ be an isotropic $\p$-alternating smooth immersion with Frenet splitting 
$$\sigma^*T\h_\tau^{2m}\cong \mathscr{L}_1 \oplus \mathbf{P}(\mathscr{L}_{2m-1}) \oplus \mathscr{L}_{3} \oplus \cdots \oplus \p(\mathscr{L}_{3}) \oplus \mathscr L_{2m-1} \oplus\p(\mathscr L_1),\quad \mathscr L_1\cong TS.$$
Let $h:=\sigma^*g$ be the pull-back metric which is positive (resp. negative) definite on $\mathscr{L}_{2i-1}$ (resp. $\p(\mathscr{L}_{2i-1})$) for $i=1,\dots,m$. We will denote with $h_{i}$ its restriction to $\mathscr{L}_{i}$. Being $\sigma$ a space-like immersion, the surface is naturally equipped with an (almost)-complex structure $J_1$ induced by the conformal structure of the first fundamental form, and thus with a Riemann surface structure $X=(S,J_1)$. Since the vector bundles $\mathscr{L}_{2i-1}$ are all of real rank $2$, and the connection $\nabla^{2i-1}$ induced by the decomposition (\ref{eq:nabla_decomposition}) is compatible with 
$h_{2i-1}$, they can be regarded as a pair $(L_{2i-1}, \boldsymbol{h}_{2i-1})$ where $L_{2i-1}\to X$ is a holomorphic line bundle and $\boldsymbol{h}_{2i-1}$ is an Hermitian metric. They are naturally endowed with one of the two possible orientations of the Euclidean plane, so that $\overline{\mathscr{L}_{2i-1}}$ corresponds to the pair $(L_{2i-1}^{-1},\boldsymbol{h}_{2i-1}^{-1})$ (see \cite{huybrechts2005complex} for more details). Once an orientation on $\mathscr{L}_{2i-1}$, with $i=1,\dots,m$, is fixed, the real vector bundles $\p(\mathscr{L}_{2i-1})$ inherits an orientation by requiring that $\p$ be orientation-preserving. In this way, $\p(\mathscr{L}_{2i-1})$ corresponds to the line bundle $L_{2i-1}$ as well. In particular, since $\mathscr{L}_{1}=TS$, we have $L_1\cong K^{-1}$ where $K:=(T^{1,0}X\big)^*$ is the canonical bundle of $X$. 

\begin{remark} By the discussion above, we obtain $m$ holomorphic line bundles $L_{2i-1}$ endowed with Hermitian metric $\boldsymbol{h}_{2i-1}$, for $i=1, \dots, m$, associated to an isotropic $\p$-alternating surface in $\h^{2m}_{\tau}$. Since $\p(\mathscr{L}_{2m-2i+1}) \cong \mathscr{L}_{2i}$, we are going to re-index the line bundles so that $L_{2m-2i+1}:=L_{2i}$ for $1\leq i \leq \frac{m}{2}$.    
\end{remark}

\noindent The next result provides a holomorphic interpretation of the smooth sections $\eta_2,\dots,\eta_{m+1}$ appearing in the decomposition of the pull-back connection of the ambient space (see $(\ref{eq:nabla_decomposition})$), similar to what occurs for the case of \emph{A-surfaces} defined by Nie (see \cite[Theorem 4.6]{Nie_alternating}). 

\begin{theorem}\label{thm:hol_data} 
 Let $\sigma:S\to\h_\tau^{2m}$ be a smooth isotropic $\p$-alternating immersion with $m\ge 2$. Let us assume that $\sigma(S)$ is not contained in any para-complex hyperbolic subspace of real codimension $4$, then: \begin{enumerate}
    \item[(1)] none of the smooth sections $\eta_2,\dots,\eta_m$ is identically zero;
    \item[(2)] once an orientation on $\mathscr{L}_1\cong TS$ is chosen, the real vector bundles $\mathscr{L}_2,\dots,\mathscr{L}_{2m}$ can be endowed with an orientation as well, so that they correspond to holomorphic line bundles $L_1,L_2,\dots, L_{m}, L_{m}, \dots, L_2, L_1$ as explained previously; 
     \item[(3)] any $\eta_i$ can be seen as a holomorphic $(1,0)$-form with values in $\Hom_{\C}(L_{i-1},L_i)$, for $i=2,\dots,m$ and $\eta_{m+1}$ can be seen as a holomorphic $(1,0)$-form with values in $\Hom_{\C}(L_{m},\overline{L_{m}})$;
     \item [(4)] The Frenet splitting $\sigma^*T\h_\tau^{2m}\cong \mathscr{L}_1 \oplus \mathbf{P}(\mathscr{L}_{2m-1}) \oplus \mathscr{L}_{3} \oplus \cdots \oplus \p(\mathscr{L}_{3}) \oplus \mathscr L_{2m-1} \oplus\p(\mathscr L_1)$ is unique;
 \end{enumerate}
\end{theorem}
\begin{proof}
We begin with point $(1)$. This is similar to the pseudo-hyperbolic case considered by Nie (\cite{Nie_alternating}). In fact, immersed copies of para-complex hyperbolic space in $\h_\tau^{2m}$ are obtained by taking a non-degenerate linear $\mathbb{R}_{\tau}$-invariant subspace $V\subset\R_\tau^{2m+1}$, then consider the intersection with the quadric formed by vectors of $\boldsymbol q$-norm equal to $-1$ and then projectivize as in Definition \ref{def:paracomplex_hyperbolic_space}. The only difference is that if there exists an index $k=2,\dots,m$ such that $\eta_k\equiv 0$, then by Proposition \ref{prop:some_properties} the corresponding $\eta_{2m-2j+2}$ (if $k=2j$ is even) or $\eta_{2m-2j+3}$ (if $k=2j-1$ is odd) is identically zero as well. Using the decomposition (\ref{eq:nabla_decomposition}) of the pull-back connection $\nabla$, we obtain a projection onto a copy of $\h_\tau^{2m-2}$ which has real codimension $4$ in $\h_\tau^{2m}$. 
\newline For point $(2)$ and $(3)$ we shall write the Gauss-Codazzi equations for the immersed surface in an appropriate moving frame. First, let us choose a lift $\hat\sigma:S\to\widehat\h_\tau^{2m}\subset\R_\tau^{2m+1}$ inside the quadric formed by all vectors with $\boldsymbol q$-norm equal to $-1$. Then, by using the decomposition \(\mathbb{R}_\tau^{2m+1} = \hat\sigma^*T\mathbb{H}_\tau^{2m} \oplus \mathbb{R}\hat\sigma \oplus \tau\mathbb{R}\hat\sigma\) (see (\ref{eq:decomposition_R_tau})), we can consider, in a neighborhood of each point of \(S\), locally defined functions \((u_1, v_1, \dots, u_{2m}, v_{2m})\) with values in \(\mathbb{R}^{4m+2} \equiv \mathbb{R}_\tau^{2m+1}\), such that \(u_{2i-1}, v_{2i-1}\) is an orthonormal frame of $\mathscr{L}_{2i-1}$ and $(u_{2m-2i+2},v_{2m-2i+2})=(\p u_{2i-1}, \p v_{2i-1})$ (see Proposition \ref{prop:some_properties}). In other words, if we consider the functions \((u_1, v_1, \dots, u_{2m}, v_{2m})\) as column vectors, we have a map \(F = (\sigma, u_1, v_1, \dots, u_{2m}, v_{2m}, \tau \sigma)\) that is locally well-defined around each point of the surface, with values in \(\mathrm{SU}(2m+1, \mathbb{R}_\tau) \cong \mathrm{SL}(2m+1, \mathbb{R})\) (Proposition \ref{prop:iso_group}). Using the decomposition (\ref{eq:nabla_decomposition2}) of the pull-back connection \(\nabla\), we find that the differential of the moving frame can be expressed as $\mathrm{d}F = F \cdot \Omega$, 
with 
\[  \renewcommand\arraystretch{1.5}
    \Omega = \begin{pmatrix}  
        \ & \theta^t \\ 
        \theta & \Omega_1 & \boldsymbol\eta_2^t \\
        \ & \boldsymbol\eta_2 & \Omega_2 & \boldsymbol\eta_3^t \\ 
        \ & \ & \boldsymbol\eta_3 & \Omega_3 & \ddots \\
        \ & \ & \ & \ddots & \ddots & \boldsymbol\eta_m^t \\ 
        \ & \ & \ & \ & \boldsymbol\eta_{m} &\Omega_{m} & \boldsymbol\eta_{m+1}^t \\
        \ & \ & \ & \ & \ & \boldsymbol\eta_{m+1} & \Omega_{m} & \boldsymbol\eta_m \\
        \ & \ & \ & \ & \ & \ &\boldsymbol\eta_{m}^t & \Omega_{m-1} & \ddots 
        \\ 
        \ & \ & \ & \ & \ & \ & \ & \ddots & \ddots & \boldsymbol\eta_{2} \\ 
        \ & \ & \ & \ & \ & \ & \ & \ & \boldsymbol\eta_{2}^t & \Omega_{1} & \theta \\ 
        & \ & \ & \ & \ & \ & \ & \ & \ & \theta^t & 
        \end{pmatrix}
\] 
where $\theta^t:=(\theta_1,\theta_2)$ is the local frame of $T^*S$ dual to $(u_1,v_1)$ and $\Omega_i$ is the matrix of the connection $\nabla^i$ in the frame $(u_i,v_i)$. Moreover, $\boldsymbol{\eta}_i$ (resp. $\boldsymbol{\eta}_i^t$) are the locally defined $2\times 2$ matrices of $1$-forms corresponding to the section $\eta_i$ (resp. $-\eta_i^\dag$) for $i=1,\dots,m+1$. In the local frame $(u_i,v_i)$ the matrices $\Omega_i$ can be written as $$\Omega_i=\begin{pmatrix}
    0 & \omega_i \\ -\omega_i & 0
\end{pmatrix}$$ for a $1$-form $\omega_i$, since $\nabla^i$ is metric preserving. The Gauss-Codazzi equation of the immersed surface are: $\mathrm{d}\Omega+\Omega\wedge\Omega=0$, and, since the expression for \(\Omega\) is formally similar to that provided by Nie in \cite[Theorem 4.6]{Nie_alternating}, we can employ formulas (4.6), (4.7), (4.8) from \cite{Nie_alternating} to derive the following equations in our case: 
\begin{align} 
&\begin{cases}\label{eq:Maurer_Cartan_first_block}\mathrm{d}\Omega_1+\theta\wedge\theta^t+\boldsymbol\eta_2^t\wedge\boldsymbol\eta_2=0 \\ \mathrm{d}\Omega_i+\boldsymbol\eta_i\wedge\boldsymbol\eta_i^t+\boldsymbol\eta_{i+1}^t\wedge\boldsymbol\eta_{i+1}=0 \quad (2\le i\le m) \\ \mathrm{d}\Omega_{m}+\boldsymbol\eta_{m+1}\wedge\boldsymbol\eta_{m+1}^t+\boldsymbol\eta_m\wedge\boldsymbol\eta_m^t=0 
    \end{cases}\\ &\begin{cases}\label{eq:Maurer_Cartan_second_block}
        \mathrm d\theta+\Omega_1\wedge\theta=0 \\ \mathrm{d}\boldsymbol\eta_i+\boldsymbol\eta_i\wedge\Omega_{i-1}+\Omega_i\wedge\boldsymbol\eta_i=0 \quad (2\le i\le m+1)
    \end{cases} \\ &\begin{cases}\label{eq:Maurer_cartan_third_block}
        \boldsymbol\eta_2\wedge\theta=0 \\ \boldsymbol\eta_i\wedge\boldsymbol\eta_{i-1}=0 \quad (3\le i\le m+1)
    \end{cases}
\end{align}It is important to note that part of the equations has not been written due to the symmetry of the matrix \(\Omega\), which, we recall, is a consequence of the \(\p\)-alternating property of the surfaces under consideration. In what follows we shall employ some standard arguments from complex differential geometry, which we recall for the sake of clarity. Let $L$ be a Hermitian line bundle and let $E$ be a real rank $2$ vector bundle endowed with a metric and a compatible connection. It is possible to define a natural structure of Hermitian holomorphic bundle on $\Hom_\R(L,E)$, where the complex structure is obtained by pre-composing each homomorphism $L\to E$ with the complex structure on $L$, while the metric and compatible connection are induced by those on $L$ and $E$. We will denote such holomorphic Hermitian bundle as $\mathscr{Hom}_\R(L,E)$. Furthermore, whenever we have a non-zero holomorphic $(1,0)$-form $\boldsymbol\eta$ with values in $\mathscr{Hom}_\R(L,E)$ and such that $\boldsymbol\eta(X)$ is conformal for any $X\in\Gamma(TS)$, then there exists a unique orientation on the real vector bundle $E$ such that $\boldsymbol\eta$ takes values in $\Hom_\C(L,E)$ (see \cite[Lemma 4.9]{Nie_alternating} for instance). The holomorphicity of $\boldsymbol\eta$ is equivalent to require that $\mathrm{d}^{\nabla^\Hom}\boldsymbol\eta =0$ where $\nabla^\Hom$ is the induced metric connection on $\Hom_\R(L,E)$ and $\mathrm{d}^{\nabla^\Hom}\boldsymbol{\eta}$ is understood as a smooth $2$-form with values in $\Hom_\R(L,E)$. Bearing this in mind, we can now prove the existence of an orientation on the real vector bundles \(\mathscr{L}_i\) that form the Frenet splitting. We begin by choosing an orientation on \(\mathscr{L}_1 \cong TS\) and treating it as a Hermitian line bundle \(K^{-1}\) so that the complex structure \(J_1\), induced by the first fundamental form, is compatible with this orientation. In particular, since these surfaces are maximal (see Remark \ref{rem:maximal_spacelike}), it follows that the second fundamental form, and hence \(\eta_2\), satisfies the following relation:  
\[
\II(J_1 X, Y) + \II(X, J_1 Y) = 0, \quad \text{for any} \ X, Y \in \Gamma(TS).
\]  
In other words, the associated matrix \(\boldsymbol\eta_2\) can be interpreted as a \(\Hom_\R(K^{-1}, \mathscr{L}_2)\)-valued \((1,0)\)-form. The second equation in (\ref{eq:Maurer_Cartan_second_block}) with \(i=2\) is equivalent to requiring that \(\mathrm{d}^{\nabla^\Hom} \boldsymbol\eta_2 = 0\), and thus \(\boldsymbol\eta_2\) is in fact holomorphic. As discussed above, there exists a unique orientation on \(\mathscr{L}_2\) that induces a Hermitian line bundle structure \(L_2\) and such that \(\boldsymbol\eta_2\) becomes a \((1,0)\)-form taking values in \(\Hom_\C(K^{-1}, L_2)\). The next step is to obtain an analogous result for \(\boldsymbol\eta_i\) with \(i = 3, \dots, m+1\) using the second equations in (\ref{eq:Maurer_Cartan_second_block}) and (\ref{eq:Maurer_cartan_third_block}). We will explain the procedure only for \(\boldsymbol\eta_3\), as it is analogous for the remaining cases. Being a smooth $1$-form with values in $\Hom_\R(L_2,\mathscr{L}_3)$, we can decompose \(\boldsymbol\eta_3 = \boldsymbol\eta_3^{(1,0)} + \boldsymbol\eta_3^{(0,1)}\) into its \((1,0)\)-part and \((0,1)\)-part. Since, by the preceding argument, \(\boldsymbol\eta_2\) is a matrix of \((1,0)\)-forms, the second equation in (\ref{eq:Maurer_cartan_third_block}) with \(i=3\) becomes \(\boldsymbol\eta_3^{(0,1)} \wedge \boldsymbol\eta_2 = 0\), which implies that \(\boldsymbol\eta_3^{(0,1)}\) must vanish on the set of points in \(S\) where \(\boldsymbol\eta_2 \neq 0\). Moreover, since \(0 \neq \boldsymbol\eta_2\) is holomorphic, this set is dense in \(S\), and by continuity, \(\boldsymbol\eta_3^{(0,1)} = 0\) everywhere. In other words, \(\boldsymbol\eta_3\) is a matrix of \((1,0)\)-forms. Finally, the second equation in (\ref{eq:Maurer_Cartan_second_block}) with \(i=3\) implies that \(\boldsymbol\eta_3\) is holomorphic, allowing us to apply the same reasoning as previously used for \(\boldsymbol\eta_2\). There is a caveat only for $\boldsymbol{\eta}_{m+1}$: we explain the argument when $m$ is odd, the other being analogous. Equation \eqref{eq:Maurer_Cartan_second_block} still implies that $\boldsymbol{\eta}_{m+1}$ is a $(1,0)$-form with values in $\Hom_{\R}(\mathscr{L}_{m}, \p(\mathscr{L}_{m}))$ (if $m$ were even, it would take values in $\Hom_{\R}(\p(\mathscr{L}_{m}), \mathscr{L}_{m})$, but that is the only difference), however now the orientation on $\mathscr{L}_{m}$ has already been fixed when considering $\boldsymbol{\eta}_{m}$ and there is a natural orientation on $\p(\mathscr{L}_{m})$ by requiring that $\p$ be orientation preserving. Since $\boldsymbol{\eta}_{m+1}$ is conformal, we can write $\boldsymbol{\eta}_{m+1}=\boldsymbol{\eta}_{m+1}'+\boldsymbol{\eta}_{m+1}''$, where $\boldsymbol{\eta}_{m+1}' \in \Hom_{\C}(L_{m}, L_{m})$ and $\boldsymbol{\eta}_{m+1}''\in \Hom_{\C}(L_{m}, \overline{L}_{m})$. We shall prove that $\boldsymbol{\eta}_{m+1}'=0$ by showing that in orthonormal basis compatible with the orientations of $\mathscr{L}_{m}$ and $\p(\mathscr{L}_{m})$ the matrix $\boldsymbol{\eta}_{m+1}$ is symmetric. Let $\{u_{m}, v_{m}\}$ be an oriented orthonormal basis of $\mathscr{L}_{m}$ and let $\{\p u_{m}, \p v_{m}\}$ be the corresponding oriented basis of $\p(\mathscr{L}_{m})$. We already observed at the beginning of the proof that $-\eta^{\dag}_{i}$ is represented by the matrix $\boldsymbol{\eta}_{i}^{t}$ in the chosen frames. On the other hand, by Proposition \ref{prop:some_properties}, for all $X\in \Gamma(TS)$ and $\xi \in \{\p u_{m}, \p v_{m}\}$, we have 
\[
    \eta_{m+1}(X)\p\xi = -\p(\eta_{m+1}^{\dag}(X)\xi)
\]
which implies that $\eta_{m+1}$ and $-\eta^{\dag}_{m+1}$ are represented by the same matrices in these frames, hence $\boldsymbol{\eta}_{m+1}$ is symmetric. \\
We finally prove point $(4)$. The vector bundles $\mathscr{L}_1\cong TS$ and $\mathscr{L}_{2m}\cong\p(TS)$ are uniquely determined by the surface and the para-complex structure of the ambient space. Iteratively, the bundles $\mathscr{L}_{i}$ for $i=2, \dots, m$ are uniquely determined by $\mathscr{L}_{i-1}$ as the image of the map $\eta_{i}$, which has full rank except at isolated points by item $(1)$ and $(3)$. The bundles $\mathscr{L}_{i}$ with $i=m+1, \dots, 2m$ can be recovered using the para-complex structure $\mathbf{P}$ by Definition \ref{def:P_alternating_surfaces} part $(iii)$.

\end{proof}
\begin{remark}
In the statement of the theorem, the case \(m=1\) has not been considered, as it falls within the study of maximal spacelike Lagrangian immersions in \(\h_\tau^2\), which have been analyzed in \cite{RT_bicomplex} (see also \cite{hildebrand2011cross,dorfmeister2024half}). In this case, the techniques used to study these surfaces are simplified.
\end{remark}
\noindent We will refer to the holomorphic line bundles $L_{1}, \dots, L_{m}$ endowed with the Hermitian metrics $\boldsymbol{h}_{1}, \dots, \boldsymbol{h}_{m}$ together with the holomorphic $(1,0)$-forms $\eta_{2}, \dots, \eta_{m+1}$ as the holomorphic immersion data associated to a smooth isotropic $\p$-alternating immersion $\sigma:S \rightarrow \mathbb{H}^{2m}_{\tau}$. 

\begin{theorem}\label{thm:structure_eq_hol}
Let $\sigma:S\to\h_\tau^{2m}$ be a smooth isotropic $\p$-alternating immersion with $m\ge 2$. Let us assume that $\sigma(S)$ is not contained in any para-complex hyperbolic subspace of real codimension $4$, then its holomorphic immersion data satisfy \begin{equation}\label{eq:immersion_data_eta}
    \begin{cases}
        \partial^{2}_{z\bar{z}}\log\boldsymbol h_{m}-|\eta_{m}|^{2}\boldsymbol h_{m}\boldsymbol h_{m-1}^{-1}+|\eta_{m+1}|^{2}\boldsymbol h_{m}^{-2}=0 \\
        \partial^{2}_{z\bar{z}}\log\boldsymbol h_{j}-|\eta_{j}|^{2}\boldsymbol h_{j}\boldsymbol h_{j-1}^{-1}+|\eta_{j+1}|^{2}\boldsymbol h_{j+1}\boldsymbol h_{j}^{-1}=0, \ \ \ \ \ \ \ j=2, \dots, m-1 \\
        \partial^{2}_{z\bar{z}}\log\boldsymbol h_{1}-\boldsymbol h_{1}+|\eta_{2}|^{2}\boldsymbol h_{2}\boldsymbol h_{1}^{-1} =0 .
    \end{cases}
\end{equation}
\end{theorem}
\begin{proof} In the proof of Theorem \ref{thm:hol_data}, we obtained equations \eqref{eq:Maurer_Cartan_first_block}, \eqref{eq:Maurer_Cartan_second_block} and \eqref{eq:Maurer_cartan_third_block} after choosing an orthonormal local frame $(u_i,v_i)$ for each $\mathscr{L}_{i}$. Under the special choice $(u_i,v_i)=\boldsymbol{h}_i^{-\frac{1}{2}}(\ell_i,i\ell_i)$, for a holomorphic section $\ell_{i}$ of $L_{i}$ such that $\boldsymbol{h}_{i}(\ell_{i}, \ell_{i})=\boldsymbol{h}_{i}$, we have
\begin{align}\label{eq:matrices_eta}
\boldsymbol{\eta}_i &= \sqrt{\frac{\boldsymbol{h}_{i}}{\boldsymbol{h}_{i-1}}}\begin{pmatrix}
            \Ree(\eta_{i}dz) & -\Ima(\eta_{i}dz) \\
            \Ima(\eta_{i}dz) & \Ree(\eta_{i}dz)
    \end{pmatrix}    \ \ \  \text{for $2\leq j\leq m$} \\
\boldsymbol{\eta}_{m+1} &= \boldsymbol{h}_{m}^{-1}\begin{pmatrix}
            \Ree(\eta_{m}dz) & -\Ima(\eta_{m}dz) \\
            -\Ima(\eta_{m}dz) & -\Ree(\eta_{m}dz) 
    \end{pmatrix}. \notag
\end{align}
We shall show that Equation \eqref{eq:Maurer_Cartan_first_block} reduces to the required PDEs. After expanding the $1$-forms in Equation \eqref{eq:matrices_eta} in the real coordinates $(x,y)$ with $z=x+iy$, we obtain by direct computations that
\begin{align}\label{eq:wedge_etas}
    \boldsymbol{\eta}_{i}\wedge \boldsymbol{\eta}_{i}^{t} & = - \boldsymbol{\eta}_{i}^{t} \wedge \boldsymbol{\eta}_{i} = -2\frac{\boldsymbol{h}_{i}}{\boldsymbol{h}_{i-1}}|\eta_{i}|^{2}  \begin{pmatrix} 0 & -1 \\ 1 & 0   \end{pmatrix} dx \wedge dy \\
    \boldsymbol{\eta}_{m+1}\wedge \boldsymbol{\eta}_{m+1}^{t} & =  \boldsymbol{\eta}_{m+1}\wedge \boldsymbol{\eta}_{m+1} = 2\boldsymbol{h}_{m}^{-2}|\eta_{m+1}|^{2}  \begin{pmatrix} 0 & -1 \\ 1 & 0   \end{pmatrix} dx \wedge dy. \notag
\end{align}
Moreover, by the proof of Theorem \ref{thm:hol_data} we know that $\Omega_i\wedge\Omega_i=0$, hence the exterior differential $d\Omega_i$ is the matrix of the curvature $2$-form of $\nabla^i$ under the frame $(u_i,v_i)$. But since $\nabla^i$ is the Chern connection of $L_i$, its curvature is expressed in the holomorphic frame $\ell_i$ as
$\bar{\partial}\partial\log \boldsymbol{h}_i=2i \partial_{z\bar{z}}^{2}\log \boldsymbol{h}_i \  dx\wedge dy$. Changing to the real frame $(u_i,v_i)$, we get 
\begin{equation}\label{eq:curvature}
d\Omega_i=2 \partial^{2}_{z\bar{z}}\log \boldsymbol{h}_i \begin{pmatrix}
    0 & -1 \\ 1 & 0
\end{pmatrix} dx \wedge dy.
\end{equation}
Finally, since $\theta^{t}=(\theta_1,\theta_2)$ is the frame of $T^{*}S$ dual to the orthonormal frame $(u_1,v_1)$ of $\mathscr{L}_{1}=TS$, the wedge product $\theta_1\wedge\theta_2$ equals the volume form of the first fundamental form $2\Ree(\boldsymbol{h}_{1})$. It follows that 
\begin{equation} \label{eq:wedge_theta}
\theta \wedge \theta^{t}=- \begin{pmatrix} 0 & -1 \\ 1 & 0   
\end{pmatrix} dx \wedge dy = -2\boldsymbol{h}_{1} \begin{pmatrix} 0 & -1 \\ 1 & 0   \end{pmatrix} dx \wedge dy .
\end{equation}
Inserting Equations \eqref{eq:wedge_etas}, \eqref{eq:curvature} and \eqref{eq:wedge_theta} into Equation \eqref{eq:Maurer_Cartan_first_block}, we get the required equations.
\end{proof}

\subsection{Infinitesimal rigidity of isotropic $\mathbf{P}$-alternating surfaces}\label{sec:inf_rigidity}
In this section, we restrict to equivariant isotropic $\p$-alternating surfaces in $\h^{2m}_{\tau}$ satisfying an extra technical assumption first introduced by Dai and Li (\cite{Dai_Li}) and also appearing in Nie's work (\cite{Nie_alternating}). We will show that this class of surfaces are infinitesimally rigid. \\

\noindent Let $\Sigma$ be a closed, connected, oriented surface of genus at least $2$ and let $\sigma: \widetilde{\Sigma} \rightarrow \mathbb{H}^{2m}_{\tau}$ be an isotropic $\p$-alternating surface that is equivariant under a representation $\rho:\pi_{1}(\Sigma) \rightarrow \SL(2m+1,\R)$. By equivariance, we can interpret the holomorphic data $(L_{1}, \dots, L_{m}, \boldsymbol{h}_{1}, \dots, \boldsymbol{h}_{m}, \eta_{2}, \dots, \eta_{m+1})$ as being defined on the Riemann surface $X=(\Sigma, J_{1})$, where $J_{1}$ is the complex structure compatible with the induced metric on $\sigma(\widetilde{\Sigma})$. Moreover, by Theorem \ref{thm:hol_data}, we can interpret $\eta_{i}$ as holomorphic $1$-forms on $X$. Recall the following definition:

\begin{defi} Given a holomorphic $1$-form $\alpha$ on $X$, we denote by $(\alpha)$ its divisor, i.e. the $\mathbb{Z}_{\geq 0}$-valued function on $X$ assigning to every $z\in X$ the multiplicity of the zero of $\alpha$ at $z$. If $\beta$ is another holomorphic $1$-form we say that $(\alpha) \prec (\beta)$ if the divisor $(\alpha)$ is strictly smaller than $(\beta)$ at every point where $\alpha$ vanishes. 
\end{defi}

\noindent In this section we consider equivariant isotropic $\p$-alternating surfaces in $\mathbb{H}^{2m}_{\tau}$ whose holomorphic data satisfy the extra technical assumption 
\begin{equation}\label{eq:assumption}
    (\eta_{2}) \prec (\eta_{3}) \prec \dots \prec (\eta_{m+1}) \ .
\end{equation}
This condition, as first observed by Dai-Li (\cite{Dai_Li}) and then Nie (\cite{Nie_alternating}), guarantees a better control on the behavior of the solutions of the system of PDEs in Equation \eqref{eq:immersion_data_eta}, as the following theorem shows (see also \cite{sagman2024hitchin}). We do not know whether this assumption is necessary in order for the result to hold. 

\begin{theorem} \label{thm:estimates}
Let $(L_{1}, \dots, L_{m}, \boldsymbol{h}_{1}, \dots, \boldsymbol{h}_{m}, \eta_{2}, \dots, \eta_{m+1})$ be the holomorphic data associated to an equivariant isotropic $\p$-alternating surface in $\h^{2m}_{\tau}$ that is not contained in any para-complex hyperbolic subspace of real codimension $4$. Assume that Equation \eqref{eq:assumption} holds, then $\partial^{2}_{z\bar{z}}\log\boldsymbol{h}_{i} \geq 0$ for all $i=1, \dots, m$ with strict inequality away from the zeros of $\eta_{i}$.
\end{theorem}
\begin{proof} The proof is a straightforward application of the maximum principle for systems proved by Dai-Li (\cite{Dai_Li}) and is an easy adaptation of the proof of a similar result appeared in the aformentioned paper. It is reported here mostly for reader's convenience. We will explain the proof for $m\geq 4$: the other simpler cases can be easily deduced using the same ideas. We first assume that $\eta_{m+1}$ is not identically zero and at the end we will point out how the proof has to be modified if $\eta_{m+1} =0$. By Theorem \ref{thm:structure_eq_hol}, the following system of locally defined PDEs holds
\[
\begin{cases}
        \partial^{2}_{z\bar{z}}\log\boldsymbol h_{m}-|\eta_{m}|^{2}\boldsymbol h_{m}\boldsymbol h_{m-1}^{-1}+|\eta_{m+1}|^{2}\boldsymbol h_{m}^{-2}=0 \\
        \partial^{2}_{z\bar{z}}\log\boldsymbol h_{j}-|\eta_{j}|^{2}\boldsymbol h_{j}\boldsymbol h_{j-1}^{-1}+|\eta_{j+1}|^{2}\boldsymbol h_{j+1}\boldsymbol h_{j}^{-1}=0, \ \ \ \ \ \ \ j=2, \dots, m-1 \\
        \partial^{2}_{z\bar{z}}\log\boldsymbol h_{1}-\boldsymbol h_{1}+|\eta_{2}|^{2}\boldsymbol h_{2}\boldsymbol h_{1}^{-1} =0 .
\end{cases}
\]
We observe that, after dividing each equation by $\boldsymbol{h}_{1}$, which is a Hermitian metric on $K^{-1}$, the equations above are actually defined globally on $X=(S,J_{1})$. We also set $\|\eta_{j+1}\|^{2}=\boldsymbol{h}_{j+1}\boldsymbol{h}_{1}^{-1}\boldsymbol{h}_{j}^{-1}|\eta_{j+1}|^{2}$ for $j=1, \dots, m-1$ and $\| \eta_{m+1}\|^{2}=\boldsymbol{h}_{1}^{-1}\boldsymbol{h}_{m}^{-2}|\eta_{m+1}|^{2}$ so that we obtain
\[
\begin{cases}
    \frac{1}{4}\Delta_{\boldsymbol{h}_{1}} \log \boldsymbol{h}_{m} = \| \eta_{m}\|^{2} - \|\eta_{m+1}\|^{2} \\
    \frac{1}{4}\Delta_{\boldsymbol{h}_{1}} \log \boldsymbol{h}_{j} = \| \eta_{j}\|^{2} - \|\eta_{j+1}\|^{2}, \ \ \ \ \ \ \ j=2, \dots, m-1 \\ 
    \frac{1}{4}\Delta_{\boldsymbol{h}_{1}} \log \boldsymbol{h}_{1} = 1 - \| \eta_{2} \|^{2}.
\end{cases}
\]
We now set $v_{j}=\log\|\eta_{j}\|^{2}-\log\|\eta_{j+1}\|^{2}$ for $j=2, \dots, m$ and $v_{1}=-\log\|\eta_{2}\|^{2}$. These functions satisfy the following equations away from $P=\cup_{j}P_{j}$, where $P_{j}$ is the (discrete) set of zeros of $\eta_{j}$:
\begin{equation}\label{eq:coop_system}
    \begin{cases}
        \frac{1}{4}\Delta_{\boldsymbol{h}_{1}}v_{1} = -c_{2}v_{2}+2c_{1}v_{1} \\
        \frac{1}{4}\Delta_{\boldsymbol{h}_{1}}v_{j} = -c_{j+1}v_{j+1}+2c_{j}v_{j}-c_{j-1}v_{j-1}, \ \ \ \ \ \ \ j=2, \dots, m-1 \\
        \frac{1}{4}\Delta_{\boldsymbol{h}_{1}}v_{m} = 3c_{m}v_{m}-c_{m-1}v_{m-1},
    \end{cases}
\end{equation}
with $c_{k}=\frac{\|\eta_{k}\|^{2}-\|\eta_{k+1}\|^{2}}{v_{k}}\geq 0$ for $j=2, \dots, m-1$ and $c_{1}=\frac{1-\|\eta_{2}\|^{2}}{v_{1}} \geq 0$. Note that the assumption in Equation \eqref{eq:assumption} implies that $v_{j}$ is a smooth function on $X\setminus P_{j+1}$ that tends to $+\infty$ when $z$ approaches a point in $P_{j+1}$ and the functions $c_{k}$ are smooth on $X\setminus P$ and bounded on $X$ for all $k$. Moreover, $P$ is non-empty because $q=(\eta_{2}\dots \eta_{m})^{2}\eta_{m+1}$ is a non-zero holomorphic $(2m+1)$-differential on $X$ vanishing in $P$. Therefore, the system in Equation \eqref{eq:coop_system} satisfies the assumption of the maximum principle for systems (\cite[Lemma 3.1, Condition (2)]{Dai_Li}) which implies that $v_{j}>0$ and hence $\partial_{z\bar{z}}^{2}\log\boldsymbol{h}_{j}\geq 0$ for all $j=1,\dots, m$ with strict inequality away from the zeros of $\eta_{j}$ \\
If $\eta_{m+1} =0$, the proof follows the same lines after simply removing $\|\eta_{m+1}\|^{2}=0$ and $v_{m}$ in all the equations above, which thus become
\[
    \begin{cases}
    \frac{1}{4}\Delta_{\boldsymbol{h}_{1}}v_{1} = -c_{2}v_{2}+2c_{1}v_{1} \\
    \frac{1}{4}\Delta_{\boldsymbol{h}_{1}}v_{j} = -c_{j+1}v_{j+1}+2c_{j}v_{j}-c_{j-1}v_{j-1}, \ \ \ \ \ \ \ \ \ \ \ \ \ \ \ \  j=2, \dots, m-2 \\
    \frac{1}{4}\Delta_{\boldsymbol{h}_{1}}v_{m-1}= 3c_{m-1}v_{m-1}-\|\eta_{m-2}\|^{2} \leq  3c_{m-1}v_{m-1} .
    \end{cases}
\]
The only subtlety is that now the set $P$ may be empty. However, since $\eta_{m-2}$ is not identically zero by Theorem \ref{thm:hol_data}, Condition (1) in \cite[Lemma 3.1]{Dai_Li} can be used, which still implies that $v_{j}>0$ everywhere on $X$ for all $j=1, \dots, m-1$ and thus we can conclude as before. 
\end{proof}

\noindent We now are going to see that this analytical result implies that this family of isotropic $\p$-alternating surfaces in $\h^{2m}_{\tau}$ is stable. To do that, we first recall some well-known facts about variations of a spacelike maximal immersion $\sigma:S \rightarrow (M,\mathbf{g})$ into a pseudo-Riemannian manifold. A variation of $\sigma$ is a smooth $1$-parameter family $\{\sigma_{t}\}_{t\in (-\epsilon, \epsilon)}$ of immersions such that $\sigma_{0}=\sigma$. The derivative $\dot{\sigma}:=\frac{d}{dt}{\sigma_{t}}_{|_{t=0}}$ is a section of $\sigma^{*}TM$ and is called \emph{variational vector field}. We say that the variation is \emph{compactly supported} if $\sigma_{t}=\sigma$ outside a compact set of $S$ and we say it is \emph{normal} if $\dot{\sigma}_{0}$ is a section of the bundle $\mathcal{N}$ normal to $TS$ inside $\sigma^{*}TM$ for the pull-back metric $h=\sigma^{*}\mathbf{g}$. The following result is standard (\cite{anciaux2011minimal}):

\begin{theorem} \label{thm:variation_area} Let $\sigma: S \rightarrow (M, \mathbf{g})$ be a spacelike maximal immersion. Then for every compactly supported  variation $\{\sigma_{t}\}_{t\in (-\epsilon, \epsilon)}$
\begin{itemize}
    \item the area functional $        t \to Area(\sigma(S))$
    has a critical point at $t=0$;
    \item furthermore, if the variation is normal, the second derivative of the area functional at $t=0$ is given by
    \begin{equation}\label{eq:second_variation}
        \int_{S} \trace_{h} \left( h(\nabla^{\mathcal{N}}_{\cdot}\xi, \nabla^{\mathcal{N}}_{\cdot}\xi)- h(\mathbf{S}_{\xi}(\cdot), \mathbf{S}_{\xi}(\cdot)) + R^{\nabla}(\cdot, \xi, \cdot, \xi) \right)
    \end{equation}
    where $\xi=\dot{\sigma}$, $\nabla^{\mathcal{N}}$ denoted the normal component of the Levi-Civita connection of $h$ and $\mathbf{S}_{\xi}(\cdot)$ is the shape operator of $\sigma$ defined as the tangential component of $\nabla_{\cdot}\xi$;
    \item if every $\sigma_{t}$ is a spacelike maximal immersion then the normal component of $\dot{\sigma}$ is a \emph{Jacobi field}. 
\end{itemize}
\end{theorem}

\noindent Jacobi fields are vector fields $\xi \in \Gamma(\mathcal{N})$ that annihilate the integral in Equation \eqref{eq:second_variation}. Equivalently, these are vector fields in the kernel of the self-adjoint operator $L:\Gamma(\mathcal{N}) \rightarrow \Gamma(\mathcal{N})$ defined by
\[
    L\xi:=\Delta^{\mathcal{N}}\xi+ \sum_{i,j=1}^{n}h(\II(e_{i},e_{j}),\xi)\II(e_{i},e_{j})-\trace_{h}(R^{\nabla}(\cdot, \xi)\cdot),
\]
where $\Delta^{\mathcal{N}}$ is the Laplacian of the connection on the normal bundle. Indeed, unpacking the definitions of the second fundamental form and the shape operator together with a simple integration by parts, we have
\[
    -\int_{S} h(L\xi,\xi)dVol_{h} = \int_{S} \trace_{h} \left( h(\nabla^{\mathcal{N}}_{\cdot}\xi, \nabla^{\mathcal{N}}_{\cdot}\xi)- h(\mathbf{S}_{\xi}(\cdot), \mathbf{S}_{\xi}(\cdot)) + R^{\nabla}(\cdot, \xi, \cdot, \xi) \right)dVol_{h} 
\]
for any compactly supported normal variation $\xi$. \\

\noindent If we specialize to equivariant isotropic $\p$-alternating immersions into $\h^{2m}_{\tau}$, which are maximal by Remark \ref{rem:maximal_spacelike}, Theorem \ref{thm:variation_area} implies the following:

\begin{theorem}\label{thm:inf_rigidity}Let $\sigma:\widetilde{\Sigma} \rightarrow \h^{2m}_{\tau}$ be an isotropic $\p$-alternating immersion that is not contained in any para-complex hyperbolic subspace of real codimension $4$ and whose holomorphic data $(L_{1}, \dots, L_{m}, \boldsymbol{h}_{1}, \dots, \boldsymbol{h}_{m}, \eta_{2}, \dots, \eta_{m+1})$ satisfy the assumption in Equation \eqref{eq:assumption}. Let $\mathcal{N}^{+}=\mathscr{L}_{3} \oplus \mathscr{L}_{5} \oplus \dots \oplus \mathscr{L}_{2m-1}$ and $\mathcal{N}^{-}=\mathscr{L}_{2} \oplus \mathscr{L}_{4} \oplus \dots \oplus \mathscr{L}_{2m}$ denote the spacelike and timelike part of the normal bundle of $\sigma(\widetilde{\Sigma})$ respectively. Then, for any compactly supported variation of $\sigma$, the associated second derivative of the area functional is positive if the variational vector field $\xi_{+}=\dot{\sigma}$ is a non-zero section of $\mathcal{N}^{+}$ and negative if the variational vector field $\xi_{-}=\dot{\sigma}$ is a non-zero section of $\mathcal{N}^{-}$.    
\end{theorem}
\begin{proof}  The proof follows the same lines as \cite[Theorem 5.5]{Nie_alternating}. We first decompose the connection on the normal bundle $\mathcal{N}$ under the splitting $\mathcal{N}=\mathcal{N}^{+}\oplus \mathcal{N}^{-}$:
\[
    \nabla^{\mathcal{N}}=\begin{pmatrix}
        \nabla^{\mathcal{N}^{+}} & \Gamma_{0}^{-} \\
        \Gamma_{0}^{+} & \nabla^{\mathcal{N}^{-}} 
    \end{pmatrix} . 
\]
Here $\nabla^{\mathcal{N}^{\pm}}$ is a connection on $\mathcal{N}^{\pm}$ and $\Gamma_{0}^{\pm}$ is a $\Hom(\mathcal{N}^{\pm}, \mathcal{N}^{\mp})$-valued $1$-form. Let $\xi_{\pm}$ be a non-zero section of $\mathcal{N}^{\pm}$. We have to estimate the expression 
\[
    \underbrace{\trace_{h}(h(\nabla^{\mathcal{N}^{\pm}}_{\cdot}\xi_{\pm}, \nabla^{\mathcal{N}^{\pm}}_{\cdot}\xi_{\pm}))}_{(a)}+
    \underbrace{\trace_{h}(\Gamma_{0}^{\pm}(\cdot)\xi_{\pm}, \Gamma_{0}^{\pm}(\cdot)\xi_{\pm})}_{(b)}-
    \underbrace{\trace_{h}(\mathbf{S}_{\xi_{\pm}}(\cdot), \mathbf{S}_{\xi_{\pm}}(\cdot))}_{(c)}+\underbrace{\trace_{h}(R^{\nabla}(\cdot, \xi_{\pm}, \cdot, \xi_{\pm}))}_{(d)}
\]
where every $"\pm"$ is taken to be a $"+"$ when dealing with $\xi_{+}$ and a $"-"$ when dealing with $\xi_{-}$. We will deal with each term separately. \\
\underline{\emph{Term (a)}.} It is positive (resp. negative) when considering $\xi_{+}$ (resp. $\xi_{-}$) because $h$ is positive (resp. negative) definite on $\mathcal{N}^{+}$ (resp. $\mathcal{N}^{-}$). \\
\underline{\emph{Term (c)}.} We recall that the shape operator and the second fundamental form are related by the formula $h(\II(X,Y),\xi)=h(X,\mathbf{S}_{\xi}(Y))$ for all tangent vector fields $X, Y$. By Remark \ref{rem:maximal_spacelike}, the second fundamental form $\II(X,Y)$ coincides with $\eta_{2}(X)Y$, which lies in $\mathscr{L}_{2}$. Therefore, $h(X,\mathbf{S}_{\xi_{+}}(Y))=0$ for all $X,Y \in \Gamma(TS)$ because $\xi_{+}$ is orthogonal to $\mathscr{L}_{2}$ and $h(\mathbf{S}_{\xi_{-}}(X),\mathbf{S}_{\xi_{-}}(X))\geq 0$ for all $X\in \Gamma(TS)$ because $h$ is positive definite on $TS$. As a consequence, term (c) vanishes for $\xi_{+}$ and is non-negative for $\xi_{-}$. \\
\underline{\emph{Term (d)}.} Let $\{e_{1},e_{2}\}$ be an $h$-orthonormal basis of $TS$. Using the explicit expression of the curvature tensor of a para-K\"ahler manifold with constant para-holomorphic sectional curvature $-4$ (Lemma \ref{lem:para_hol_curvature}), we obtain
\begin{align*}
    R^{\nabla}(e_{i},\xi, e_{i}, \xi) &= h(e_{i},e_{i})h(\xi,\xi) - h(\xi,e_{i})h(e_{i},\xi) + h(e_{i},\p e_{i})h(\p\xi,\xi) \\
    & \ \ \ \ -h(\xi, \p e_{i})h(\p e_{i}, \xi) +2h(e_{i}, \p\xi)h(\p e_{i},\xi) \\
    &= \|\xi\|^{2}_{h}-3h(\xi, \p e_{i})^{2},
\end{align*}
for all sections $\xi$ of $\mathcal{N}$. In particular, $\trace_{h}(R^{\nabla}(\cdot, \xi_{+}, \cdot, \xi_{+}))= 2 \|\xi_{+}\|_{h}^{2}$ because $\p e_{i} \in \mathscr{L}_{2} \subset (\mathcal{N}^{+})^{\perp}$ and $\trace_{h}(R^{\nabla}(\cdot, \xi_{-}, \cdot, \xi_{-})) \leq 2 \|\xi_{-}\|_{h}^{2}$. \\
\underline{\emph{Term (b)}.} From the expression of the connection $\nabla$ in Equation \eqref{eq:nabla_decomposition}, we deduce that 
\[
    \Gamma_{0}^{+} = \begin{pmatrix}
        -\eta_{3}^{\dag} & & & \\
        \eta_{4} & -\eta_{5}^{\dag} & & \\
         & \ddots & \ddots & \\
         & & \ddots & -\eta_{2m-1}^{\dag} \\
         & & & \eta_{2m} 
    \end{pmatrix} \ \ \ \text{and} \ \ \ 
    \Gamma_{0}^{-} = \begin{pmatrix}
        \eta_{3} & -\eta_{4}^{\dag} & & & \\
         & \eta_{5} & -\eta_{6}^{\dag} & & \\
         &  & \ddots & \ddots & \\
         & & & \ddots & -\eta_{2m}^{\dag} \\
         & & & & \eta_{2m-1}
    \end{pmatrix}.
\]
In order to compute term (b), we fix an orthonormal frame $(u_{2i-1}, v_{2i-1})$ of $\mathscr{L}_{2i-1}$ and the corresponding frame $(u_{2m-2i+2},v_{2m-2i+2}):=(\p u_{2i-1}, \p v_{2i-1})$ of $\mathscr{L}_{2m-2i-1}=\p(\mathscr{L}_{2i-1})$, for $i=2, \dots, m$, so that $\eta_{i}$ is represented by $2\times 2$-matrices of $1$-forms $\boldsymbol{\eta}_{i}$ and $-\eta_{i}^{\dag}$ by its transpose $\boldsymbol{\eta}_{i}^{t}$, as in the proof of Theorem \ref{thm:structure_eq_hol}. We also use the local frame to view the section $\xi_{\pm}$ as a column vector, so that 
\[
    h(\Gamma_{0}^{\pm}(\cdot)\xi_{\pm},\Gamma_{0}^{\pm}(\cdot)\xi_{\pm}) = \mp\xi_{\pm}^{t}(\Gamma_{0}^{\pm}(\cdot))^{t}\Gamma_{0}^{\pm}(\cdot)\xi_{\pm}.
\]
Then, a simple matrix multiplication gives that these last expressions are equal to
\[  \renewcommand\arraystretch{1.5}
    -\xi_{+}^{t} \begin{pmatrix}
        \boldsymbol{\eta}_{3}\boldsymbol{\eta}_{3}^{t}  +\boldsymbol{\eta}_{4}^{t}\boldsymbol{\eta}_{4} & \boldsymbol{\eta}_{4}^{t}\boldsymbol{\eta}_{5}^{t} & &  \\
        \boldsymbol{\eta}_{5}\boldsymbol{\eta}_{4} & \boldsymbol{\eta}_{5}\boldsymbol{\eta}_{5}^{t} + \boldsymbol{\eta}_{6}^{t}\boldsymbol{\eta}_{6} & \boldsymbol{\eta}_{6}^{t}\boldsymbol{\eta}_{7}^{t} &  \\
        & \ddots & \ddots & \ddots \\
        & & \boldsymbol{\eta}_{2m}\boldsymbol{\eta}_{2m-1} & \boldsymbol{\eta}_{2m-1}\boldsymbol{\eta}_{2m-1}^{t} + \boldsymbol{\eta}_{2m}^{t}\boldsymbol{\eta}_{2m} \\
    \end{pmatrix}\xi_{+} 
\]
\[
    \ \ \xi_{-}^{t} \begin{pmatrix}
        \boldsymbol{\eta}_{3}\boldsymbol{\eta}_{3}^{t} & \boldsymbol{\eta}_{3}^{t}\boldsymbol{\eta}_{4}^{t} & &  & \\
        \boldsymbol{\eta}_{4}\boldsymbol{\eta}_{3} & \boldsymbol{\eta}_{4}\boldsymbol{\eta}_{4}^{t} + \boldsymbol{\eta}_{5}^{t}\boldsymbol{\eta}_{5} & \boldsymbol{\eta}_{5}^{t}\boldsymbol{\eta}_{6}^{t} &  & \\
        & \ddots & \ddots & \ddots & \\
        &  & \boldsymbol{\eta}_{2m-2}\boldsymbol{\eta}_{2m-3} & \boldsymbol{\eta}_{2m-2}\boldsymbol{\eta}_{2m-2}^{t} + \boldsymbol{\eta}_{2m-1}^{t}\boldsymbol{\eta}_{2m-1} & \boldsymbol{\eta}_{2m-1}^{t}\boldsymbol{\eta}_{2m}^{t} \\
        & & & \boldsymbol{\eta}_{2m}\boldsymbol{\eta}_{2m-1} & \boldsymbol{\eta}_{2m}\boldsymbol{\eta}_{2m}^{t} \\
    \end{pmatrix}\xi_{-}
\]
The explicit formula of $\boldsymbol{\eta}_{i}$ allows to compute the following traces:
\begin{itemize}
    \item $\trace_{h}(\boldsymbol{\eta}_{i}(\cdot)\boldsymbol{\eta}_{j}(\cdot))= \boldsymbol{\eta}_{i}(e_{1})\boldsymbol{\eta}_{j}(e_{1})+\boldsymbol{\eta}_{i}(e_{2})\boldsymbol{\eta}_{j}(e_{2})=0$ for all $i\neq j$;
    \item $\trace_{h}(\boldsymbol{\eta}_{i}^{t}(\cdot)\boldsymbol{\eta}_{j}^{t}(\cdot))= 0$ for all $i\neq j$
    \item $\trace_{h}(\boldsymbol{\eta}_{i}(\cdot)\boldsymbol{\eta}_{i}^{t}(\cdot))= \trace_{h}(\boldsymbol{\eta}^{t}_{i}(\cdot)\boldsymbol{\eta}_{i}(\cdot))=\|\eta_{i}\|^{2}\mathrm{Id}$ for all $i$.
\end{itemize}
We deduce that the quadratic form $\xi_{+} \to \trace_{h}(\Gamma_{0}^{+}(\cdot)\xi_{+}, \Gamma_{0}^{+}(\cdot)\xi_{+})$ has eigenvalues 
\[
    -(\|\eta_{3}\|^{2}+\|\eta_{4}\|^{2}), \ -(\|\eta_{5}\|^{2}+\|\eta_{6}\|^{2}), \ \dots, \  -(\|\eta_{2m-1}\|^{2}+\|\eta_{2m}\|^{2})
\]
and the quadratic form $\xi_{-} \to \trace_{h}(\Gamma_{0}^{-}(\cdot)\xi_{-}, \Gamma_{0}^{-}(\cdot)\xi_{-})$ has eigenvalues 
\[
    \|\eta_{3}\|^{2}, \ \|\eta_{4}\|^{2}+\|\eta_{5}\|^{2}, \ \|\eta_{6}\|^{2}+\|\eta_{7}\|^{2}, \ \dots, \  \|\eta_{2m-2}\|^{2}+\|\eta_{2m-1}\|^{2}, \ \|\eta_{2m}\|^{2} .
\]
From the proof of Theorem \ref{thm:estimates}, we learn that
\[
    \| \eta_{m+1} \|^{2} < \|\eta_{m}\|^{2} < \dots < \|\eta_{2}\|^{2} < 1
\]
and by Proposition \ref{prop:some_properties}, the inequality $\|\eta_{i}\|^{2}<1$ actually holds for all $i=1, \dots m$. Therefore, we obtain the following estimates for term (b):
\[
    \trace_{h}(\Gamma_{0}^{+}(\cdot)\xi_{+}, \Gamma_{0}^{+}(\cdot)\xi_{+}) > -2\|\xi_{+}\|^{2}_{h} \ \ \ \text{and} \ \ \ \trace_{h}(\Gamma_{0}^{-}(\cdot)\xi_{-}, \Gamma_{0}^{-}(\cdot)\xi_{-}) < 2\|\xi_{-}\|^{2}_{h}.
\]
Combining all four terms we obtain
\begin{itemize}
    \item for $\xi_{+}: (a)+(b)-(c)+(d) > -2\|\xi_{+}\|^{2}_{h}+ 2\|\xi_{+}\|^{2}_{h} = 0$;
    \item for $\xi_{-}: (a)+(b)-(c)+(d) < 2\|\xi_{-}\|^{2}_{h}+ 2\|\xi_{-}\|^{2}_{h} < 0$ because $\mathcal{N}^{-}$ is negative definite,
\end{itemize}
so the proof is complete.
\end{proof}

\noindent As a consequence we obtain:

\begin{cor}\label{cor:no_Jacobi} Under the assumptions of Theorem \ref{thm:inf_rigidity}, an isotropic $\p$-alternating immersion $\sigma:\widetilde{\Sigma} \rightarrow \h^{2m}_{\tau}$ does not admit any non-zero compactly supported Jacobi field.
\end{cor}
\begin{proof} Let $\xi$ be a compactly supported Jacobi field and write $\xi=\xi_{+}+\xi_{-}$ where $\xi_{\pm}$ are compactly supported sections of $\mathcal{N}^{\pm}$. Since $L\xi=0$, we have
\begin{equation}\label{eq:1}
    h(L\xi_{+}, \xi_{+})+h(L\xi_{-},\xi_{+})=h(L\xi, \xi_{+})=0
\end{equation}
\begin{equation}\label{eq:2}
h(L\xi_{-}, \xi_{-})+h(L\xi_{+},\xi_{-})=h(L\xi, \xi_{-})=0.
\end{equation}
Therefore, we obtain
\begin{align*}
    -\int_{\widetilde{\Sigma}} h(L\xi_{+}, \xi_{+})dVol_{h} & = \int_{\widetilde{\Sigma}}  h(L\xi_{-}, \xi_{+})dVol_{h} \tag{Equation \eqref{eq:1}}\\
    & = \int_{\widetilde{\Sigma}}  h(L\xi_{+}, \xi_{-})dVol_{h} \tag{$L$ is self-adjoint} \\
    & = -\int_{\widetilde{\Sigma}} h(L\xi_{-}, \xi_{-})dVol_{h}. \tag{Equation \eqref{eq:2}} 
\end{align*}
However, using the fact that the integral of $-h(L\xi_{\pm}, \xi_{\pm})$ coincides with the second derivative of the area functional induced by the variational vector field $\xi_{\pm}$, by Theorem \ref{thm:inf_rigidity} the first integral in the equation above is positive unless $\xi_{+}=0$ and the last integral is negative unless $\xi_{-}=0$. We thus conclude that $\xi_{\pm}=0$ and $\xi=0$, so that there are no non-zero compactly supported Jacobi fields.
\end{proof}

\section{Cyclic $\mathrm{SL}(2m+1, \mathbb{R})$-Higgs bundles}
\noindent In this section we recall some well-known facts about cyclic $\mathrm{SL}(2m+1, \mathbb{R})$-Higgs bundles over a Riemann surface $X$. We introduce their para-complexification and use it to prove the existence of an equivariant isotropic $\mathbf{P}$-alternating surface in $\mathbb{H}^{2m}_{\tau}$ for any given representation $\rho:\pi_{1}(X) \rightarrow \SL(2m+1,\R)$ associated to a family of stable cyclic $\mathrm{SL}(2m+1, \mathbb{R})$-Higgs bundles. In particular, we interpret Hitchin's equations in this context as structure equations for this type of surfaces in $\mathbb{H}^{2m}_{\tau}$. 

\subsection{Definition, stability and moduli space}\label{sec:4.1}
Let $X=(\Sigma, J)$ be a fixed closed Riemann surface of negative Euler characteristic and denote by $K$ its canonical bundle. An $\SL(2m+1,\R)$-Higgs bundle over $X$ is a triple $(E,\phi, Q)$, where
\begin{itemize}
    \item  $E$ is a holomorphic vector bundle of rank $2m+1$ over $X$ with $\det(E)=\mathcal{O}_{X}$;
    \item $Q$ is a non-degenerate holomorphic bilinear form on $E$;
    \item $\phi$ is a holomorphic $1$-form with values in the bundle $\mathrm{End}_{0}(E)$ of traceless endomorphisms of $E$ that is compatible with $Q$, in the sense that $Q(\phi(v),w)=Q(v,\phi(w))$ for all $v,w \in E$.
\end{itemize}
Note that the bilinear form $Q$ induces an isomorphism between $E$ and its dual $E^{*}$. The pair $(E,\phi)$ alone would instead define an $\SL(2m+1,\C)$-Higgs bundle.\\

\noindent We say that an $\SL(2m+1,\R)$-Higgs bundle is cyclic if $E$ decomposes into a direct sum of line bundles $E=L_{m}^{-1}\oplus \cdots \oplus L_{1}^{-1} \oplus \mathcal{O}_{X} \oplus L_{1} \oplus \cdots \oplus L_{m}$ and in this holomorphic splitting the Higgs field writes as
\[
    \phi=\begin{pmatrix}
    0 & & & & & & & &\gamma_{m}\\
    \gamma_{m-1} & 0 & & & & & & & \\
    &\ddots & \ddots & & & & & &\\
    & & \gamma_{1} & 0 & & & & &\\
    & & & \mu & 0 & & & &\\
    & & & & \mu & 0 & & &\\
    & & & & & \gamma_{1} & 0 & & \\
    & & & & & & \ddots & \ddots & \\
    & & & & & & & \gamma_{m-1} & 0
    \end{pmatrix}
\]
where $0 \neq \gamma_{i} \in H^{0}(X, L_{i}^{-1}L_{i+1}K)$ for $i=1, \dots, m-1$, $0 \neq \mu \in H^{0}(X, L_{1}K)$ and $\gamma_{m} \in H^{0}(X, L_{m}^{-2}K)$. Note that $(\mu\gamma_{1}\dots \gamma_{m-1})^{2}\gamma_{m} \in H^{0}(X,K^{2m+1})$, thus it is a holomorphic $(2m+1)$-differential that we will denote by $q_{2m+1}$. The bilinear form $Q$ in this case is represented by the matrix
\[
    Q=\begin{pmatrix}
     & & & 1\\
     & & \iddots & \\
     & \iddots &  &\\
    1 & & & 
    \end{pmatrix}.
\]
A cyclic $\SL(2m+1,\R)$-Higgs bundle is thus determined by the holomorphic data $(L_{1}, \dots, L_{m}, \mu, \gamma_{1}, \dots, \gamma_{m})$. \\

\noindent To form the moduli space of $\SL(2m+1,\R)$-Higgs bundles, we need to consider a notion of stability. There is an appropriate one for $G$-Higgs bundle for each Lie group $G$ (\cite{Simpson_quasiprojective, HK_correspondence}), however, in our setting, we can consider an $\SL(2m+1,\R)$-Higgs bundle as an $\SL(2m+1,\C)$-Higgs bundle with extra structure, specifically given by the holomorphic bilinear form $Q$. It turns out (\cite{HK_correspondence}) that in this case it is sufficient to check the stability condition for the $\SL(2m+1,\C)$-Higgs bundle, which we now recall. 

\begin{defi}\label{def:stability} An $\SL(2m+1,\C)$-Higgs bundle $(E,\phi)$ is stable if, for any proper $\phi$-invariant subbundle $F$, we have $\deg(F)<0$. Moreover, $(E,\phi)$ is polystable if $(E,\phi)$ is a direct sum of stable Higgs bundles of degree $0$. 
\end{defi}

\noindent For cyclic $\SL(2m+1,\C)$-Higgs bundles, the following result by Simpson (\cite[Lemma 6.4]{Simpson_Katz}) further simplifies the search for destabilizing subbundles:

\begin{prop}\label{prop:stability_cyclic} Let $(E,\phi)$ be a cyclic $\SL(2m+1,\C)$-Higgs bundle with $E=L_{1} \oplus \dots \oplus L_{m+1}$. Then $(E,\phi)$ is stable if and only if for all proper $\phi$-invariant holomorphic subbundles $F=F_{1} \oplus \dots \oplus F_{k} \subset E$ with $F_{i} \subset E_{i}$, we have $\deg(F)<0$.  
\end{prop}

\noindent By results of Hitchin (\cite{hitchin1987self}) and Simpson (\cite{Simpson_Higgs}), the polystability condition is equivalent to the existence of a Hermitian metric $H$ on $E$ that solves Hitchin's equations
\begin{equation}\label{eq:Hitchin_general}
    F_{H}+[\phi, \phi^{*H}]=0 \ ,
\end{equation}
where $F_{H}$ denotes the curvature of $H$ and $\phi^{*H}$ is the adjoint of $\phi$ with respect to the metric $H$. It is well-known that these equations arise by asking the connection
\begin{equation}\label{eq:connection}
    \nabla=D_{H}+\phi+\phi^{*H}
\end{equation}
to be flat, where $D_{H}$ denotes the Chern connection of $H$. For an $\SL(2m+1,\R)$-Higgs bundle $(E,\phi,Q)$ the Hermitian metric $H$ must also be compatible with $Q$, in the sense that $QH^{t}Q=H^{-1}$ (see for instance \cite{Collier-Li-asymptotics}). \\

\noindent When $(E,\phi,Q)$ is stable, there is a unique Hermitian metric $H$ on $E$ solving Hitchin's equation. By \cite{Collier_thesis}, the metric $H$ is diagonal in the above decomposition of $E$ into line bundles and can be written as
\[
    H=\diag(h_{m}^{-1}, \dots, h_{1}^{-1}, 1, h_{1}, \dots, h_{m}),
\]
where $h_{i}$ is a Hermitian metric on $L_{i}$ for each $i=1, \dots, m$. In particular, Hitchin's equations \eqref{eq:Hitchin_general} simplify to 
\begin{equation}\label{eq:Hitchin_cyclic}
    \begin{cases}
        \partial^{2}_{z\bar{z}}\log h_{m}-|\gamma_{m-1}|^{2}h_{m}h_{m-1}^{-1}+|\gamma_{m}|^{2}h_{m}^{-2}=0 \\
        \partial^{2}_{z\bar{z}}\log h_{j}-|\gamma_{j-1}|^{2}h_{j}h_{j-1}^{-1}+|\gamma_{j}|^{2}h_{j+1}h_{j}^{-1}=0, \ \ \ \ \ \ \ j=2, \dots, m-1 \\
        \partial^{2}_{z\bar{z}}\log h_{1}-|\mu|^{2}h_{1}+|\gamma_{1}|^{2}h_{2}h_{1}^{-1} =0 .
    \end{cases}
\end{equation}

\noindent The metric $H$ together with the bilinear form $Q$ defines a real structure on $E$: the bundle morphism $T:E\rightarrow E$ given by $T(v)=H^{-1}Q\bar{v}$ is a $\C$-antilinear involution and its fixed locus $\mathcal{E}$ is parallel for $\nabla$. \\

\noindent We recall the following distinguished family of cyclic $\SL(2m+1,\R)$-Higgs bundles: 
\begin{defi}\label{def:Higgs_bundles_Hitchin_component} A cyclic $\SL(2m+1,\R)$-Higgs bundle belongs to the Hitchin component (\cite{hitchin1992lie}) if $L_{i}=K^{-i}$ for all $i=1, \dots, m$ and the Higgs field is of the form
\[
    \phi=\begin{pmatrix}
    0 & & &  &q_{2m+1}\\
    1 & 0 &  & & \\
    &\ddots & \ddots &  &\\
    & & 1 & 0 & \\
    & & & 1 & 0
    \end{pmatrix},
\]
where $q_{2m+1}$ is a holomorphic $(2m+1)$-differential. 
\end{defi}

\noindent It turns out that all Higgs bundles in the Hitchin component are stable (\cite{hitchin1992lie}). \\

\noindent We finally recall the equivalence relation needed to form the moduli space of Higgs bundles.

\begin{defi} \label{defi:equivalent}Two polystable $\SL(2m+1,\C)$-Higgs bundles $(E,\phi)$ and $(E',\phi')$ are isomorphic if there is a holomorphic bundle automorphism $E \rightarrow E'$ that identifies $\phi$ with $\phi'$.
\end{defi}

\noindent In our setting, this equivalence relation becomes more explicit: 
\begin{lemma}\label{lm:isomorphic_Higgs} Two stable cyclic $\SL(2m+1,\R)$-Higgs bundles determined by the data \[
(L_{1}, \dots, L_{m}, \mu, \gamma_{1}, \gamma_{2}, \dots, \gamma_{m-1}, \gamma_{m}) \ \ \ \text{and} \ \ \ (L_{1}', \dots, L_{m}', \mu', \gamma_{1}', \gamma_{2}', \dots, \gamma_{m-1}', \gamma_{m}')
\]
are isomorphic if $L_{i}=L_{i}'$, $\mu'= \lambda_{1}^{-1}\mu$, $\gamma_{i}'=\lambda_{i}^{-1}\lambda_{i+1}\gamma_{i}$ for all $i=1, \dots m-1$ and $\gamma_{m}'=\lambda_{m}^{-2}\gamma_{m}$, for some $\lambda_{i}\in \C^{*}$.
\end{lemma}
\begin{proof} Such isomorphism is given by a diagonal gauge transformation $g$ with $\det(g)=1$ that preserves the bilinear form $Q$. As a result,
\[
    g=\diag(\lambda_{m}^{-1}, \dots, \lambda_{1}^{-1}, 1, \lambda_{1}, \dots, \lambda_{m}).
\]
Since $\gamma_{i} \in H^{0}(X, L_{i}^{-1}L_{i+1}K)$ for $i=1, \dots, m-1$, $\mu \in H^{0}(X,L_{1}K)$ and $\gamma_{m} \in H^{0}(X,L_{m}^{-2}K)$, the gauge transformation $g$ acts on $\gamma_{i}$ and $\mu$ as in the statement.
\end{proof}

\noindent By a result of Nitsure (\cite{Nitsure}) and Simpson (\cite{Simpson_quasiprojective}), isomorphism classes of polystable $\SL(2m+1,\C)$-Higgs bundles over a compact Riemann surface $X$ form a quasi-projective variety. \\

\noindent Finally, by the nonabelian Hodge correspondence, there is a homeomorphism between the moduli space of $\SL(2m+1,\C)$-Higgs bundles and the character variety 
\[
    \chi(\pi_{1}(X), \SL(2m+1,\C)) = \Hom(\pi_{1}(X), \SL(2m+1,\C)//\SL(2m+1,\C) \ 
\]
given by associating to an equivalence class of a polystable Higgs bundle $(E,\phi)$ the holonomy of the flat connection $\nabla$ in Equation \eqref{eq:connection}. In particular, polystable $\SL(2m+1,\R)$-Higgs bundles correspond to representations into a subgroup isomorphic to $\SL(2m+1,\R)$.

\subsection{Para-complexification of Higgs bundles}\label{sec:4.2}In this section we introduce the notion of \emph{para-complex} Higgs bundle over a Riemann surface $X$ and describe a "para-complexification" procedure that from an $\SL(n,\R)$-Higgs bundle gives a para-complex one with isomorphic holonomy.

\begin{defi}\label{def:para-complex_higgs} A para-complex Higgs bundle over a Riemann surface $X$ is a triple $(F, \varphi, \mathbf{P})$, where
\begin{itemize}
    \item $(F,\varphi)$ is a Higgs bundle of rank $2k$;
    \item $\mathbf{P}$ is a holomorphic section of $\mathrm{End}(F)$ such that $\mathbf{P}^{2}=\mathrm{Id}$. 
    \item the Higgs field $\phi$ is compatible with $\mathbf{P}$, in the sense that $\varphi(\mathbf{P}v)=\mathbf{P}\varphi(v)$ for all $v\in F$.
\end{itemize}
\end{defi}

\noindent Let $(E,\phi, Q)$ be an $\SL(n,\R)$-Higgs bundle. Denote by $\C_{\tau}=\R_{\tau} \otimes_{\R} \C$ the algebra of bi-complex numbers (see e.g. \cite{RT_bicomplex} for more details). We define the \emph{para-complexification} of $(E,\phi, Q)$ as follows.
We consider the tensor product of each complex vector bundle in the fiber of $E$ with $\C_{\tau}$ and obtain a holomorphic vector bundle $E^{\tau}$ of complex rank $2n$, in brief $E^{\tau}:=E \otimes_{\C} \C_{\tau}$. The holomorphic structure on $E^{\tau}$ can be easily understood via the decomposition $E^{\tau}= (E \otimes e_{+}) \oplus (E \otimes e_{-}$): holomorphic sections of $E^{\tau}$ are of the form $\sigma^{\tau}=\sigma^{+} \otimes e_{+} + \sigma^{-} \otimes e_{-}$, with $\sigma^{\pm}$ holomorphic sections of $E$. On $E^{\tau}$ there is a natural para-complex structure $\mathbf{P}$ given by multiplication by $\tau$ on each fiber. We note that the subbundles $E\otimes e_{\pm}$ are the eigenspaces for $\mathbf{P}$ relative to the eigenvalues $\pm 1$. The bilinear form $Q$ can be extended to a bilinear form on $E^{\tau}$ by setting $Q^{\tau}=Q\otimes \Id = Q\otimes e_{+} + Q\otimes e_{-}$. We then define the Higgs field on $E^{\tau}$ by
\[
    \phi^{\tau}=\phi \otimes \tau = \phi \otimes e_{+} - \phi \otimes e_{-} \ . 
\]
It is straightforward to verify that the triple $(E^{\tau}, \phi^{\tau}, \mathbf{P})$ is a para-complex Higgs bundle according to Definition \ref{def:para-complex_higgs}. With this particular choice of $\phi^{\tau}$, the Higgs field is compatible with the para-hermitian metric induced by $Q^{\tau}$ because
\[
    (\phi^{\tau})^{t}Q^{\tau}+Q^{\tau}\overline{\phi^{\tau}}^{\tau}=0.
\]
This guarantees that, after setting $H^{\tau}:=H\otimes e_{+} + H\otimes e_{-}$, where $H$ is the Hermitian metric solving Hitchin's equations on $(E,\phi)$, the connection 
\[
    \nabla^{\tau}:= D_{H^{\tau}}+\phi^{\tau}+(\phi^{\tau})^{*H^{\tau}} = (D_{H}+\phi+\phi^{*H}) \otimes e_{+} + (D_{H}-\phi-\phi^{*H}) \otimes e_{-}
\]
is flat and has the following properties:
\begin{enumerate}[i)]
    \item the para-complex structure $\mathbf{P}$ is parallel for $\nabla^{\tau}$;
    \item $\nabla^{\tau}$ preserves the para-complex real subbundle $\mathcal{E}^{\tau}:=(\mathcal{E} \otimes e_{+}) \oplus (\mathcal{E} \otimes e_{-})$. This is exactly the fixed locus of the $\C$-antilinear involution $T^{\tau}(v)= (H^{\tau})^{-1}Q^{\tau}\bar{v}$. 
    \item there is a $\nabla^{\tau}$-parallel para-Hermitian sesquilinear form defined by
    \[
        \mathbf{q}^{\C}(v,w) := v^{t} Q^{\tau}\overline{T^{\tau}(w)}^{\tau}.
    \]
    This restricts to a para-Hermitian metric on $\mathcal{E}$, which we will denote by $\mathbf{q}$. As a consequence, $\nabla^{\tau}$ preserves a fiber bundle over $X$ with fiber $\mathbb{H}^{n-1}_{\tau}$, obtained by considering vectors in $\mathcal{E}^{\tau}$ of unit norm with respect to $\mathbf{q}$.
    \item it follows from item iii) that the holonomy of $\nabla^{\tau}$ lies in a subgroup of $\SL(n,\C_{\tau})$ isomorphic to $\SL(n,\R)$. Indeed, since $\nabla^{\tau}$ preserves the real subbundle $\mathcal{E}^{\tau}$ and the Hermitian form $\mathbf{q}$, the holonomy of $\nabla^{\tau}$ is contained in $\SU(n,\R_{\tau}, \mathbf{q})$, which is isomorphic to $\SL(n.\R)$ by Proposition \ref{prop:iso_group}. More precisely, if $\Omega$ is the matrix connection of $\nabla$, then, using the fact that $\phi=Q\phi^{t}Q$ and $QH^{t}Q=H^{-1}$ (which implies that $-D_{H}=QD_{H}^{t}Q$), the matrix connection of $\nabla^{\tau}$ is
    \[
        \Omega^{\tau}= \Omega \otimes e_{+} - Q\Omega^{t}Q \otimes e_{-} 
    \]
    and thus lies in the Lie algebra of $\SU(n,\R_{\tau},Q)$. In particular, for any loop $\gamma$ on $X$, the holonomy of $\nabla^{\tau}$ around $\gamma$ is a matrix of the form
    \[
        A_{\gamma}e_{+}+Q(A_{\gamma}^{t})^{-1}Qe_{-} \in \SU(n,\R_{\tau},Q),
    \]
    where $A_{\gamma}$ is the holonomy of $\nabla$ along $\gamma$. Hence, the there is an explicit isomorphism between the holonomies of $\nabla^{\tau}$ and $\nabla$.
\end{enumerate}

\subsection{Construction of equivariant isotropic $\mathbf{P}$-alternating surfaces}\label{sec:existence_surfaces} This entire section is dedicated to the proof of the following result:
\begin{theorem}\label{thm:existence} Let $(E,\phi,Q)$ be a stable cyclic $\mathrm{SL}(2m+1,\R)$-Higgs bundle over $X$ with holonomy $\rho$. Assume further that $L_{1}=K^{-1}$ and $\mu=1$. Then there exists a $\rho$-equivariant isotropic $\mathbf{P}$-alternating surface $\sigma:\widetilde{X} \rightarrow \mathbb{H}^{2m}_{\tau}$. Moreover, the structural data of $\sigma$ is exactly the holomorphic data $(L_{1}, \dots, L_{m}, \gamma_{1}, \dots, \gamma_{m})$ used to construct the Higgs bundle along with the harmonic metric $h_{i}$ on each $L_{i}$.
\end{theorem}
\begin{proof} Let $(E^{\tau}, \phi^{\tau}, Q^{\tau})$ be the para-complexification of the Higgs bundle $(E,\phi,Q)$, as explained in the previous section. Fix a base point $x_{0}\in X$. We identify $\mathbb{H}^{n}_{\tau}$ with the set of all vectors $v\in \mathcal{E}^{\tau}_{|_{x_{0}}}$ such that $\mathbf{q}(v,v)=-1$. We define a $\rho$-equivariant map $\sigma:\widetilde{X}\rightarrow \mathbb{H}^{2m}_{\tau}$ by $\nabla^{\tau}$-parallel transport of the section 
\[
 1\otimes e_{+} - 1\otimes e_{-} \in (\mathcal{O}_{\R} \otimes e_{+}) \oplus (\mathcal{O}_{\R} \otimes e_{-}) \subset \mathcal{E}^{\tau}_{\R} . 
\]
We are now going to verify that $\sigma$ is harmonic, superconformal, maximal, isotropic and $\mathbf{P}$-alternating (see Remark \ref{rem:maximal_spacelike}). For this calculation, we let $\ell_{i}$ be a holomorphic section of $L_{i}$ such that $H(\ell_{i}, \ell_{i})=h_{i}$ for all $i=1, \dots, m$ and $\ell_{i}^{-1}$ be their dual sections of $L_{i}^{-1}$. \\
\noindent \underline{\emph{Conformality of $\sigma$.}} We compute
\begin{align}\label{eq:first_derivatives}
    \nabla_{\partial_{z}}^{\tau}\sigma &= \phi(1) \otimes e_{+} + \phi(1)\otimes e_{-} = \ell_{1} \otimes e_{+} + \ell_{1} \otimes e_{-}  \notag \\
    \nabla_{\partial_{\bar{z}}}^{\tau}\sigma &= \phi^{*H}(1) \otimes e_{+} + \phi^{*H}(1)\otimes e_{-} = h_{1}\ell_{1}^{-1} \otimes e_{+} + h_{1}\ell_{1}^{-1} \otimes e_{-}.
\end{align}
We note that
\begin{align*}
    \mathbf{q}^{\C}(\nabla_{\partial_{z}}^{\tau}\sigma,  \nabla_{\partial_{\bar{z}}}^{\tau}\sigma) &= (\ell_{1} \otimes e_{+} + \ell_{1} \otimes e_{-})^{t}Q^{\tau}\overline{T^{\tau}(h_{1}\ell_{1}^{-1} \otimes e_{+} + h_{1}\ell_{1}^{-1} \otimes e_{-})}^{\tau} \\
    &=(\ell_{1} \otimes e_{+} + \ell_{1} \otimes e_{-})^{t}Q^{\tau}(\ell_{1} \otimes e_{-} + \ell_{1} \otimes e_{+})=0 \\
    \mathbf{q}^{\C}(\nabla_{\partial_{z}}^{\tau}\sigma , \nabla_{\partial_{z}}^{\tau}\sigma)&=(\ell_{1} \otimes e_{+} + \ell_{1} \otimes e_{-})^{t}Q^{\tau}\overline{T^{\tau}(\ell_{1} \otimes e_{+} + \ell_{1} \otimes e_{-})}^{\tau} \\
    &=(\ell_{1} \otimes e_{+} + \ell_{1} \otimes e_{-})^{t}Q^{\tau}(h_{1}\ell_{1}^{-1}\otimes e_{-} + h_{1}\ell_{1}^{-1} \otimes e_{+})=h_{1} \\ 
    \mathbf{q}^{\C}(\nabla_{\partial_{\bar{z}}}^{\tau}\sigma,\nabla_{\partial_{\bar{z}}}^{\tau}\sigma) &= (h_{1}\ell_{1}^{-1} \otimes e_{+} + h_{1}\ell_{1}^{-1} \otimes e_{-})^{t}Q^{\tau}\overline{T^{\tau}(h_{1}\ell_{1}^{-1} \otimes e_{+} + h_{1}\ell_{1}^{-1} \otimes e_{-})}^{\tau} \\
    &=(h_{1}\ell_{1}^{-1} \otimes e_{+} + h_{1}\ell_{1}^{-1} \otimes e_{-})^{t}Q^{\tau}(\ell_{1} \otimes e_{-} + \ell_{1} \otimes e_{+})=h_{1},
\end{align*}
hence the map $\sigma$ is conformal and the induced metric on $\sigma(\widetilde{X})$ is locally $2h_{1} dzd\bar{z}$. \\
\noindent \underline{\emph{Harmonicity and maximality of $\sigma$.}}
The map is harmonic because
\begin{align*}
    \nabla^{\tau}_{\partial_{\bar{z}}}\nabla^{\tau}_{\partial_{z}}\sigma &= \nabla^{\tau}_{\partial_{\bar{z}}}( \ell_{1} \otimes e_{+} + \ell_{1} \otimes e_{-})\\
    &= \phi^{*H}(\ell_{1}) \otimes e_{+} - \phi^{*H}(\ell_{1}) \otimes e_{-} \\
    &= h_{1}(1\otimes e_{+} - 1\otimes e_{-}) = h_{1}\sigma,
\end{align*}
so $\nabla^{\tau}_{\partial_{\bar{z}}}\nabla^{\tau}_{\partial_{z}}\sigma$ does not have components tangential to $\mathbb{H}^{2m}_{\tau}$. Since $\sigma$ is also conformal, this shows that $\sigma$ parameterizes a maximal surface in $\mathbb{H}^{2m}_{\tau}$.\\
\noindent \underline{\emph{The surface $\sigma(\widetilde{X})$ is isotropic and $\mathbf{P}$-alternating.}} By construction, the tangent space to $\mathbb{H}^{2m}_{\tau}$ at any point on $\sigma(\widetilde{X})$ can be identified with 
\[
    (\mathcal{L}_{m} \oplus \dots \oplus \mathcal{L}_{1}) \otimes e_{+} \oplus (\mathcal{L}_{m} \oplus \dots \mathcal{L}_{1}) \otimes e_{-},
\]
where $\mathcal{L}_{i}$ is the fixed point of the $\C$-antilinear involution $T$ restricted to $L_{i}\oplus L_{i}^{-1}$. More explicitly, $\mathcal{L}_{i}$ is the real vector bundle over $X$ generated by the sections $u_{i}=\frac{1}{\sqrt{2}}(h_{i}^{-1/2}\ell_{i}+h_{i}^{1/2}\ell_{i}^{-1})$ and $v_{i}=\frac{1}{\sqrt{2}}(ih_{i}^{-1/2}\ell_{i}-ih_{i}^{1/2}\ell_{i}^{-1})$. We introduce the following notation
\[
    \mathcal{L}_{i}:= \mathcal{L}_{i} \otimes 1 \subset E^{\tau}\ \ \ \text{and} \ \ \ \tau\mathcal{L}_{i}:= \mathcal{L}_{i} \otimes \tau \subset E^{\tau}
\]
and we rearrange the decomposition above as
\begin{equation}\label{eq:splitting2}
    \mathcal{L}_{1} \oplus \tau\mathcal{L}_{2} \oplus \dots \oplus \tau\mathcal{L}_{1} \ .
\end{equation}
Note that this splitting is $\mathbf{P}$-alternating because $\{u_{i}, v_{i}\}$ is a $\mathbf{q}^{\C}$-orthonormal frame of $\mathcal{L}_{i}$ and $\mathbf{q}^{\C}(\tau w, \tau w)=-\mathbf{q}^{\C}(w,w)$ for all $w\in \mathcal{E}^{\tau}$. In particular, all subspaces $\mathcal{L}_{i}$ and $\tau\mathcal{L}_{i}$ are isotropic for the symplectic form on $\mathbb{H}^{2m}_{\tau}$ induced by the $\tau$-imaginary part of $\mathbf{q}^{\C}$. It remains to show that the Levi-Civita connection of $\mathbb{H}^{2m}_{\tau}$ is represented by a tridiagonal matrix as in Equation \eqref{eq:nabla_decomposition2} in this frame. The Levi-Civita connection of $\mathbb{H}^{2m}_{\tau}$ can be identified with the orthogonal projection of $\nabla^{\tau}$ onto $\mathcal{L}_{1} \oplus \tau\mathcal{L}_{2} \oplus \dots \oplus \tau\mathcal{L}_{1}$, so we can express its matrix connection with respect to the above splitting by a direct computation. Indeed, 
\begin{align*}
    \nabla^{\tau}(u_{1}\otimes 1) & = \nabla^{\tau}(u_{1}\otimes e_{+} + u_{1}\otimes e_{-}) \\
    & =(D_{H}+\phi+\phi^{*H})(u_{1}) \otimes e_{+} + (D_{H}-\phi-\phi^{*H})(u_{1}) \otimes e_{-} \\
    &= \left(\frac{1}{2}\partial_{x}\log(h_{1})dy-\frac{1}{2}\partial_{y}\log(h_{1})dx\right) v_{1} \otimes 1 \\
    & +h_{2}^{1/2}h_{1}^{-1/2}\Ree(\gamma_{1})u_{2} \otimes \tau -h_{2}^{1/2}h_{1}^{-1/2}\Ima(\gamma_{1})v_{2} \otimes \tau \\
   \nabla^{\tau}(v_{1}\otimes 1) & = \nabla^{\tau}(v_{1}\otimes e_{+} + v_{1}\otimes e_{-}) \\
    &= \left(-\frac{1}{2}\partial_{x}\log(h_{1})dy+\frac{1}{2}\partial_{y}\log(h_{1})dx\right) u_{1} \otimes 1 \\
   & +h_{2}^{1/2}h_{1}^{-1/2}\Ree(\gamma_{1})v_{2} \otimes \tau +h_{2}^{1/2}h_{1}^{-1/2}\Ima(\gamma_{1})u_{2} \otimes \tau \\   \nabla^{\tau}(u_{j}\otimes 1) & = \nabla^{\tau}(u_{j}\otimes e_{+} + u_{j}\otimes e_{-}) 
    \\    
    & = \left(\frac{1}{2}\partial_{x}\log(h_{j})dy-\frac{1}{2}\partial_{y}\log(h_{j})dx\right) v_{j} \otimes 1 \\
    &+h_{j}^{1/2}h_{j-1}^{-1/2}[\Ree(\gamma_{j-1})u_{j-1}-\Ima(\gamma_{j-1})v_{j-1}] \otimes \tau \\
    &+h_{j+1}^{1/2}h_{j}^{-1/2}[\Ree(\gamma_{j})u_{j+1}+\Ima(\gamma_{j})v_{j+1}] \otimes \tau \\
    \nabla^{\tau}(v_{j}\otimes 1) & = \nabla^{\tau}(v_{j}\otimes e_{+} + v_{j}\otimes e_{-}) 
    \\    
    & = \left(-\frac{1}{2}\partial_{x}\log(h_{j})dy+\frac{1}{2}\partial_{y}\log(h_{m})dx\right) u_{j} \otimes 1 \\
    &+h_{j}^{1/2}h_{j-1}^{-1/2}[\Ree(\gamma_{j-1})v_{j-1}+\Ima(\gamma_{j-1})u_{j-1}] \otimes \tau \\
    &+h_{j+1}^{1/2}h_{j}^{-1/2}[\Ree(\gamma_{j})v_{j+1}-\Ima(\gamma_{j})u_{j+1}] \otimes \tau 
\end{align*}
for all $j=2, \dots, m-1$, and 
\begin{align*}
    \nabla^{\tau}(u_{m}\otimes 1) & = \nabla^{\tau}(u_{m}\otimes e_{+} + u_{m}\otimes e_{-}) 
    \\    
    & = \left(\frac{1}{2}\partial_{x}\log(h_{m})dy-\frac{1}{2}\partial_{y}\log(h_{m})dx\right) v_{m} \otimes 1 \\
    &+h_{m}^{-1}\Ree(\gamma_{m}) u_{m} \otimes \tau - h_{m}^{-1}\Ima(\gamma_{m}) v_{m} \otimes \tau \\
    & +h_{m-1}^{-1/2}h_{m}^{1/2}\Ree(\gamma_{m-1})u_{m-1} \otimes \tau -h_{m-1}^{-1/2}h_{m}\Ima(\gamma_{m-1})v_{m-1} \otimes \tau \\
    \nabla^{\tau}(v_{m}\otimes 1) & = \nabla^{\tau}(v_{m}\otimes e_{+} + v_{m}\otimes e_{-}) 
    \\    
    & = \left(-\frac{1}{2}\partial_{x}\log(h_{m})dy+\frac{1}{2}\partial_{y}\log(h_{m})dx\right) u_{m} \otimes 1 \\
    & -h_{m}^{-1}\Ree(\gamma_{m}) v_{m} \otimes \tau - h_{m}^{-1}\Ima(\gamma_{m}) u_{m} \otimes \tau \\
    & +h_{m-1}^{-1/2}h_{m}^{1/2}\Ree(\gamma_{m-1})v_{m-1} \otimes \tau +h_{m-1}^{-1/2}h_{m}\Ima(\gamma_{m-1})u_{m-1} \otimes \tau . 
\end{align*}
Since the para-complex structure given by multiplication by $\tau$ is parallel for $\nabla^{\tau}$, we can conclude that the matrix connection of $\nabla^{\tau}$ in the chosen frame is given by
\[
    \begin{pmatrix}
        \Omega_{1} & \Gamma_{1}^{t} & & & & & & &\\
        \Gamma_{1} & \Omega_{2} & \Gamma_{2}^{t} & & & & & &\\
        & \Gamma_{2} & \Omega_{3} & \Gamma_{3}^{t} & & & & &\\
        & & & \ddots & & & & &\\
        & & & \Gamma_{m-1} & \Omega_{m} & \Gamma_{m} & & &\\
        & & & & \Gamma_{m} & \Omega_{m} & \Gamma_{m-1} & &\\
        & & & & & \Gamma_{m-1}^{t} & \Omega_{m-1} & \Gamma_{m-2} & & \\
        & & & & & & & \ddots & \\
        & & & & & & & \Gamma_{1}^{t} & \Omega_{1} 
    \end{pmatrix} ,
\]
where 
\[
    \Omega_{j} = \begin{pmatrix}
        0 & -\omega_{j} \\
        \omega_{j} & 0
    \end{pmatrix} \ \text{with $\omega_{j}=\frac{1}{2}\partial_{x}\log(h_{j})dy-\frac{1}{2}\partial_{y}\log(h_{j})dx$} 
\]
\[
    \Gamma_{j}=\sqrt{\frac{h_{j+1}}{h_{j}}}\begin{pmatrix}
            \Ree(\gamma_{j}) & -\Ima(\gamma_{j}) \\
            \Ima(\gamma_{j}) & \Ree(\gamma_{j})
    \end{pmatrix}   
\]
for all $j=1, \dots, m-1$, and 
\[
    \Gamma_{m}=h_{m}^{-1}\begin{pmatrix}
            \Ree(\gamma_{m}) & -\Ima(\gamma_{m}) \\
            -\Ima(\gamma_{m}) & -\Ree(\gamma_{m})
    \end{pmatrix} .
\]
We observe that this expression is consistent with Equation \eqref{eq:nabla_decomposition2} after matching the bundles in the decompositions
\[
    \mathcal{L}_{1} \oplus \tau\mathcal{L}_{2} \oplus \dots \oplus \mathcal{L}_{2}\oplus \tau\mathcal{L}_{1} 
\]
\[
    \mathscr{L}_1 \oplus \mathbf{P}(\mathscr{L}_{2m-1}) \oplus \cdots \oplus \mathscr L_{2m-1} \oplus\p(\mathscr L_1)
\]
place by place and setting $\Gamma_{j}=\eta_{j+1}$: in fact in local frames in which the pull-back metric is represented by plus or minus the identity matrix (depending on whether it is positive or negative definite), the operator $-\eta_{j+1}^\dag$ is represented by $\Gamma_{j}^{t}$. 

\noindent \underline{\emph{The map $\sigma$ is superconformal.}} This follows immediately from the explicit expressions of $\Gamma_{j}$ found before since they represent scaled rotations or reflections away from the zeros of $\gamma_{j}$ and the zero map at points where $\gamma_{j}=0$.  \\
\noindent \underline{\emph{Holomorphic data.}} As explained in Theorem \ref{thm:hol_data}, the holomorphic data associated to the isotropic $\mathbf{P}$-alternating surface $\sigma$ is obtained by observing that each real rank-$2$ bundle in Equation \eqref{eq:splitting2} can be identified with an Hermitian vector bundle so that the matrices $\Gamma_{j}$ correspond to holomorphic maps. Now, the bundle $\mathcal{L}_{i}$ can be identified with $L_{i}$ through the orthogonal projection from $L_{i}\oplus L_{i}^{-1}$ to $L_{i}$. This identifies the inner product on $\mathcal{L}_{i}$ with the real part of the Hermitian metric $h_{i}$. Since multiplication by $\tau$ is holomorphic, the bundle $\tau\mathcal{L}_{i}$ is identified with $L_{i}$ as well. Under these identifications, the matrices $\Gamma_{j}$ correspond to the holomorphic maps $\gamma_{j}$, thus the holomorphic data associated to $\sigma$ is exactly the data $(L_{1}, \dots, L_{m}, \gamma_{1}, \dots, \gamma_{m})$ used to construct the Higgs bundle $(E,\phi,Q)$.
\end{proof}\begin{remark}
A priori, the surfaces that we can construct starting from the Higgs bundle in Theorem \ref{thm:existence} do not necessarily satisfy assumption (\ref{eq:assumption}), which ensures the property of infinitesimal rigidity (see Section \ref{sec:inf_rigidity}). When the condition is satisfied, from Theorem \ref{thm:estimates}, it follows that the induced metric on such surfaces has strictly negative curvature, since it coincides with \( h_1 \). On the other hand, if the Higgs bundle is such that \( L_i \cong K^{-i} \) for every \( i=1,\dots,m \) (see Definition \ref{def:Higgs_bundles_Hitchin_component}), then property (\ref{eq:assumption}) is automatically satisfied, and the metric \( h_1 \) induced on the surface in \( \mathbb{H}_\tau^{2m} \) always has strictly negative curvature. This generalizes a well-known phenomenon for \( m=1 \), where the induced metric on the space-like maximal Lagrangian in \( \mathbb{H}_\tau^2 \) has negative curvature, as it coincides with the Blaschke metric of the associated hyperbolic affine sphere (\cite{simon1990local, changping1990some, li2019introduction, RT_bicomplex}).
\end{remark}
\noindent Recently, Baraglia (\cite{baraglia2010g2}) introduced the concept of a harmonic sequence starting with a minimal map from a Riemann surface into a symmetric space of indefinite signature, by adapting the Riemannian setting (\cite{calabi1967minimal, eells1992harmonic}). In this latter case, particularly when the ambient space is Kähler, it has been showed by Wood (\cite{wood1984holomorphic}) that given a minimal map whose harmonic sequence satisfies a certain condition, one can derive a holomorphic differential on the surface. In our case, starting from \(\sigma:\widetilde{X} \to \h_\tau^{2m}\) as in Theorem \ref{thm:existence}, we can construct a harmonic sequence in the following inductive manner:  
\[
\sigma_0 := \sigma, \quad \sigma_1 := \nabla^\tau_{\partial_z} \sigma, \quad \dots,\quad \sigma_{k+1} := \nabla^\tau_{\partial_z} \sigma_k=\underbrace{\nabla_{\partial_z}^\tau\cdots\nabla_{\partial_z}^\tau}_{\text{(k+1)-times}}\sigma .
\] It is important to note that this definition differs from the one given by Baraglia in \cite[\S 2.4.1]{baraglia2010g2}, but it similarly allows us to derive a holomorphic differential on the surface if a certain condition is satisfied, as we will explain. 
Using calculations similar to those in the proof of Theorem \ref{thm:existence}, one can compute the terms of the harmonic sequence and observe that the following identity holds:  
\begin{align*}
& \sigma_k=\gamma_1\cdots\gamma_{k-1}\big(l_k\otimes e_++(-1)^{k-1}l_k\otimes e_-\big)\quad (\text{mod \ } \C_\tau\sigma_{k-1}) \\
&\boldsymbol{q}^\mathbb C\big(\sigma_k,\overline{\sigma_j}\big)=0,\quad \text{for all} \ \ k+j\le 2m, \\ & \boldsymbol{q}^\mathbb C\big(\sigma_{m+1},\overline{\sigma_m}\big)=-q_{2m+1}, \qquad\nabla_{\partial_{\bar z}}\boldsymbol{q}^\mathbb C\big(\sigma_{m+1},\overline{\sigma_m}\big)=0.
\end{align*}where “mod \(\mathbb{C}_\tau \sigma_{k-1}\)” means that the equality holds up to additional linear combinations of \(\sigma_{k-1}\) with coefficients in \(\mathbb{C}_\tau\). In other words, we are stating that the isotropic order of the harmonic map \(\sigma\) is \(2m\) and that \(\boldsymbol{q}^\mathbb C\big(\sigma_{m+1},\overline{\sigma_m}\big)\) is holomorphic. 
The construction of holomorphic differentials from harmonic maps is, in fact, a much more general phenomenon that holds for immersions into para-Kähler manifolds that are not necessarily locally isomorphic to \(\h_\tau^{2m}\). We defer the general statement and the full proof with all details to Appendix \ref{sec:appendix}.

\subsection{Gauss map}\label{sec:4.4} In this section we are going to show that a natural Gauss map associated to an isotropic $\mathbf{P}$-alternating surface in $\mathbb{H}^{2m}_{\tau}$ gives rise to a conformal harmonic map to the symmetric space $\mathrm{SL}(2m+1,\R)/\SO(2m+1)$. In particular, we use the isomorphism \(\SL(2m+1,\R) \cong \SU(2m+1,\R_\tau)\) explained in Section \ref{sec:2.2} to provide a new description of the symmetric space in terms of para-complex geometry. \\

\noindent Classically, the symmetric space $\mathbb X_m=\mathrm{SL}(2m+1,\R)/\SO(2m+1)$ is identified with the space $\mathcal{M}$ of positive definite inner products on $\mathbb{R}^{2m+1}$ with unit determinant.
We give here a new interpretation in terms of the geometry of $\mathbb{H}^{2m}_{\tau}$. Let $\mathcal{N}$ be the space of totally geodesic Lagrangian subspaces of $\mathbb{H}^{2m}_{\tau}$ such that the restriction of the pseudo-Riemannian metric $g$ of $\mathbb{H}^{2m}_{\tau}$ is negative definite. Note that these are obtained by intersecting with $\mathbb{H}^{2m}_{\tau}$ Lagrangian, negative definite $(2m+1)$-planes in $\mathbb{R}^{2m+1}_{\tau}$ and then taking the quotient by the action of $\mathcal{U}$. 

\begin{theorem}\label{thm:symmetric_space} As homogeneous space, $\mathcal{N}$ is diffeomorphic to $\mathbb X_m$.    
\end{theorem}
\begin{proof} It is sufficient to show that $\SL(2m+1,\R)$ acts transitively on $\mathcal{N}$ with stabilizer $\SO(2m+1)$.
Let $N \in \mathcal{N}$. We observe that we have a $g$-orthogonal decomposition
\[
        T\mathbb{H}^{2m}_{\tau} = TN \perp \mathbf{P}TN
\]
because the restriction of $g$ to $\mathbf{P}TN$ is positive definite (and thus $TN \cap \mathbf{P}TN = \{0\}$) and 
\[
    g(v, \mathbf{P}w)=\omega(v,w) = 0
\]
for all $v,w \in TN$ since $N$ is a Lagrangian submanifold. As a consequence, we also find that for all $p\in N$
\[
    \mathbb{R}^{2m+1}_{\tau} = \mathbb{R} \cdot p \perp \tau(\mathbb{R} \cdot p) \perp T_{p}N \perp \tau(T_{p}N).
\]
Since isometries of $\mathbb{H}^{2m}_{\tau}$ commute with the multiplication by $\tau$, an isometry preserving $N$ is completely determined by its values on $p$ and on a fixed orthonormal frame of $T_{p}N$. Since the $(2m+1)$-plane generated by $T_{p}N$ and $p$ is negative definite, the stabilizer of $N$ is isomorphic to $\SO(2m+1)$. Moreover, the group of isometries acts transitively on $\mathcal{N}$ by the following argument: if $N'\in \mathcal{N}$ is another totally geodesic Lagrangian, negative definite subspace, then we have another orthogonal decomposition
\[
    \mathbb{R}^{2m+1}_{\tau} = \mathbb{R} \cdot p' \perp \tau(\mathbb{R} \cdot p') \perp T_{p'}N' \perp \tau(T_{p'}N')
\]
with $p'\in N'$. As isometry of $\mathbb{H}^{2m}_{\tau}$ sending $N$ to $N'$ can be constructed as the linear map sending $p$ to $p'$ and an orthonormal basis of $T_{p}N$ to one of $T_{p'}N'$.
\end{proof}

\noindent With this result in mind, we introduce the following definition of Gauss map:

\begin{defi} Let $\sigma:\widetilde{X} \rightarrow \mathbb{H}^{2m}_{\tau}$ be an isotropic $\mathbf{P}$-alternating immersion with Frenet splitting $ \sigma^{*}T \mathbb{H}^{2m}_{\tau} = \mathscr{L}_{1} \oplus \mathbf{P}(\mathscr{L}_{2}) \oplus \dots \oplus \mathscr{L}_{2} \oplus \mathbf{P}(\mathscr{L}_{1})$. The Gauss map of $\sigma$ is defined as $G_{\sigma}:\widetilde{X} \rightarrow \mathcal{N}$ such that $G_{\sigma}(x)$ is the totally geodesic, Lagrangian, negative definite subspace $N_{\sigma(x)}$ tangent to $\mathbf{P}(\mathscr{L}_{1}) \oplus \dots \oplus \mathbf{P}(\mathscr{L}_{m})$ at $\sigma(x)$. In particular, $N_{\sigma(x)}$ is orthogonal to the surface $\sigma(\widetilde{X})$ at $\sigma(x)$.
\end{defi}

\begin{theorem}\label{thm:Gauss_map} The Gauss map $G_{\sigma}$ of an isotropic $\mathbf{P}$-alternating immersion $\sigma:\widetilde{X} \rightarrow \mathbb{H}^{2m}_{\tau}$ is conformal and harmonic, thus it parameterizes a minimal surface in the symmetric space $\SL(2m+1,\R)/\SO(2m+1)$.
\end{theorem}
\begin{proof} By Theorem \ref{thm:existence}, we can assume that the immersion $\sigma$ comes from the data of an $\SL(2m+1,\R)$-Higgs bundle $(E,\phi,Q)$ over $X$. Fix a base point $z_{0} \in \widetilde{X}$ and, using the same notation as in the proof of the aforementioned theorem, fix the local $H^{\tau}$- and $\mathbf{q}$-orthonormal frame $\mathcal{F}(z)=\{\sigma(z)\otimes \tau, u_{1}(z) \otimes 1, v_{1}(z)\otimes 1, \dots, u_{m}(z)\otimes 1, v_{m}(z) \otimes 1\}$ of $\mathcal{E}^{\tau}$ in a neighborhood $U$ of $z_{0}$. Let $P_{z_{0}}^{z}$ denote the $\nabla^{\tau}$-parallel transport from $z_{0}$ to $z$. The map
\[
    \tilde{G}_{\sigma}: U \rightarrow \SU(2m+1,\R_{\tau})
\]
\[
    \tilde{G}_{\sigma}(z) = P_{z_{0}}^{z}\begin{pmatrix} \sigma(z_{0}) \otimes \tau & u_{1}(z_{0}) \otimes 1 & v_{1}(z_{0})\otimes 1 \dots & u_{m}(z_{0}) \otimes 1 & v_{m}(z_{0}) \otimes 1
            \end{pmatrix} ,
\]
obtained by constructing a matrix whose columns are the coordinates of the parallel transport of the vectors in $\mathcal{F}(z_{0})$ with respect to $\mathcal{F}(z)$, is a local lift of $G_{\sigma}$, in the sense that the diagram
\[
    \begin{tikzcd}
    U \arrow{r}{\tilde{G}_{\sigma}} \arrow{rd}[swap]{G_{\sigma}} & \SU(2m+1,\R_{\tau}) \arrow{d}{\pi}\\
    &  \mathcal{N}
    \end{tikzcd}
\]
commutes. Here $\pi:\SU(2m+1,\R_{\tau}) \rightarrow \mathcal{N}$ is the natural projection that associates to a matrix with columns $v_{1}, \dots, v_{2m+1}$ the totally geodesic, Lagrangian, negative definite subspace
\[
    \mathrm{Span}_{\R}(v_{1}\otimes \tau, \dots, v_{2m+1} \otimes \tau) \cap \mathbb{H}^{2m}_{\tau} \ . 
\]
On the other hand, if we post-compose $\tilde{G}_{\sigma}$ with the isomorphism 
\begin{align*}
    \epsilon_{+}:\SU(2m+1,\R_{\tau}) &\rightarrow \SL(2m+1,\R) \\   A&\mapsto Ae_{+} \ ,
\end{align*} 
we see that $\epsilon_{+} \circ \tilde{G}_{\sigma}$ also projects to 
\[
    \mathcal{H}:U \rightarrow \mathcal{M} \ \ \ \ \ \ \mathcal{H}(z):= [(\epsilon_{+} \circ \tilde{G}_{\sigma}(z))^{-1}]^{t}(\epsilon_{+} \circ \tilde{G}_{\sigma}(z))^{-1}, 
\]
which at each point $z$ represents the Hermitian metric $H$ on $(\mathcal{E}_{\R} \otimes e_{+}) \subset E^{\tau}$ with respect to the projection onto $\mathcal{E}_{\R} \otimes e_{+}$ of the parallel frame $P_{z_{0}}^{z}(\mathcal{F}(z_{0}))$. Again, this means that the diagram
\[
    \begin{tikzcd}
    U \arrow{r}{\epsilon_{+}\circ \tilde{G}_{\sigma}} \arrow{rd}[swap]{\mathcal{H}} & \SL(2m+1,\R) \arrow{d}{\pi}\\
    &  \mathcal{M}
    \end{tikzcd}
\]
commutes. This shows that if we identify the models $\mathcal{M}$ and $\mathcal{N}$ of the symmetric space $\mathbb X_m$ using the isometry induced by the isomorphism of Lie groups $\epsilon_{+}$, the maps $\mathcal{H}$ and $G_{\sigma}$ coincide:
\[
    \begin{tikzcd}
    U \arrow{r}{\tilde{G}_{\sigma}} \arrow{rd}[swap]{G_{\sigma}} \arrow[bend right=50]{drr}{\mathcal{H}}& \SU(2m+1,\R_{\tau}) \arrow{d}{\pi} \arrow{r}{\epsilon_{+}} & \SL(2m+1,\R) \arrow{d}{\pi}  \\
    &  \mathcal{N} \arrow{r}{\cong} & \mathcal{M}
    \end{tikzcd}
\]
Since by the non-abelian Hodge correspondence (\cite{Simpson_Higgs}, \cite{Corlette}), the map $\mathcal{H}$ is harmonic and conformal, the same holds for $G_{\sigma}$.
\end{proof}

\section{Application to higher rank Teichm\"uller theory} 
\noindent Let $\Sigma$ be a closed oriented smooth surface of genus $g\ge 2$. In this final section, we focus on the application to the study of surface group representations into \(\SL(2m+1,\R)\), considered as the isometry group of para-complex hyperbolic space. After a brief review of the general theory, we initially concentrate on the study of the moduli space of isotropic \(\p\)-alternating surfaces with a fixed Riemann surface structure, using the theory of cyclic Higgs bundles. Later, exploiting the infinitesimal rigidity of such surfaces, we show that the holonomy map is an immersion for \(m \geq 2\) and that it is a diffeomorphism onto a connected component of the representation space when \(m=1\), thus recovering the Labourie-Loftin parameterization of the Hitchin component. Finally, the last application concerns the construction of geometric structures in the case \(m=1\). In particular, we recover the structures defined by Guichard-Wienhard (\cite{guichard2012anosov}) with a completely different method.

\subsection{The representation space}
This section aims to introduce the space of surface group representations into \(\mathrm{SL}(2m+1,\mathbb{R})\), explaining how to topologically distinguish the various connected components discovered by Hitchin (\cite{hitchin1992lie}), and introduce some background materials on the topic. \\ \\ 
The Teichmüller space \(\mathcal{T}(\Sigma)\) of $\Sigma$ is defined as the set of (quasi-)complex structures on \(\Sigma\) compatible with the fixed orientation, modulo the action of the group of diffeomorphisms isotopic to the identity, denoted by \(\Diff_0(\Sigma)\). The Poincaré-Koebe uniformization theorem allows \(\mathcal{T}(\Sigma)\) to be identified with the space of representations \(\rho: \pi_1(\Sigma) \to \PSL(2, \R)\) that are discrete and faithful (also known as \emph{Fuchsian representations}), modulo conjugation. In fact, the Teichmüller space corresponds to one of the two connected components of this representation space.\\ \\
Much more in general, we can consider the representation space of surface groups into \(\mathrm{SL}(2m+1,\mathbb{R})\), namely the quotient of all completely reducible group homomorphisms \(\rho: \pi_1(\Sigma) \to \mathrm{SL}(2m+1,\mathbb{R})\) by the conjugation action, which will be denoted by \(\chi_m(\Sigma):=\chi(\pi_{1}(\Sigma), \SL(2m+1, \R))\). This space is endowed with an induced topology and a real analytic structure such that, at smooth points, it has dimension \(-4m(m+1)\chi(\Sigma)\) (\cite{goldman1984symplectic}). Using the non-abelian Hodge correspondence and a Morse type function on the moduli space of Higgs bundles, Hitchin (\cite{hitchin1992lie}) succeeded in counting the connected components of the space \(\chi_m(\Sigma)\). \begin{theorem}[\cite{hitchin1992lie}]
The variety $\chi_m(\Sigma)$ has three connected components: the one containing the class of the trivial representation, the one consisting of representations whose associated flat $\R^{2m+1}$-bundles have non-zero second Stiefel-Whitney class and the one consisting of representations connected to those arising as uniformization. Moreover, the third one is contained in the smooth locus of $\chi_m(\Sigma)$ and it is diffeomorphic to $\R^{-4m(m+1)\chi(\Sg)}$.
\end{theorem}
\noindent It must be noted that the there is no topological invariant which distinguishes the first component to the third one, as they are both formed by representations whose associated flat $\R^{2m+1}$-bundles have zero second Stiefel-Whitney class. The component containing the uniformizing representations is called the \emph{Hitchin component} and can be characterized as the connected component in $\chi_m(\Sigma)$ formed by deformations of the Fuchsian one, namely the post-composition of a Fuchsian representation $\pi_1(\Sigma)\to\PSL(2,\R)$ with the irreducible embedding $\PSL(2,\R)\hookrightarrow\SL(2m+1,\R)$. This component will be denoted by \(\mathrm{Hit}_m(\Sigma)\) and contains an isomorphic copy of the Teichmüller space $\mathcal{T}(\Sigma)$, as discussed above. 

\subsection{The moduli space of isotropic $\p$-alternating surfaces}\label{sec:moduli_space_fixed_X}
Let $\Sg$ be a closed oriented surfaces of genus $g\ge 2$ and let $X$ be a fixed Riemann surface structure. This section has two main objectives: first, we aim to study the moduli space of isotropic $\p$-alternating $\rho$-equivariant immersions into $\h_\tau^{2m}$ and show that there is a subspace which has a stratification labeled by the degree of a line bundle over $X$, and each strata has the structure of a holomorphic bundle over a certain symmetric product of $X$ (when $m=2$ we actually parameterize the whole moduli space); second, using polystable Higgs bundles, we study surfaces lying in a totally geodesic copy of $\h^{2m-2}_{\tau}$ in $\h^{2m}_{\tau}$ and show that they arise as limits of those described above, which correspond to stable bundles. As an interesting consequence, we obtain that the limiting representation \(\rho_0\) in \(\mathrm{SL}(2m+1,\mathbb{R})\) decomposes as a direct sum of two representations, each with a precise geometric meaning, and it is never contained in the Hitchin component. In other words, whenever $m\ge 2$, we identify a subspace of non-Hitchin representations in $\chi_m(\Sigma)$ for which there exists an equivariant isotropic $\p$-alternating immersion in \(\mathbb{H}_\tau^{2m}\).\\


\noindent Before starting the discussion, it is worth mentioning that in \cite{collier2023holomorphic} the authors studied surfaces immersed in a certain homogeneous space, equivariant for representations into the split real form $\mathrm{G}_2'$ of the exceptional complex Lie group \( \mathrm G_2 \). Using an approach similar to the Frenet splitting, they related such surfaces to cyclic $\mathrm G_2'$-Higgs bundles and investigated the corresponding moduli space. The stratification that we obtain for \( \mathrm{SL}(2m+1, \mathbb{R}) \) (a Lie group of real rank $2m$) resembles the same phenomenon for \( \mathrm{G}_2' \) (a Lie group of real rank $2$). 
\\ \\ For a fixed Riemann surface structure $X$, let us define the set of isomorphism classes of equivariant isotropic $\p$-alternating immersions into $\h_\tau^{2m}$ as $$\widetilde{\mathcal{ML}}(X)_m:=\bigslant{\left\{(\sigma,\rho) \ \Bigg| \ \parbox{15em}{$\sigma: \widetilde{X}\to\h_\tau^{2m}$ as in Definition \ref{def:P_alternating_surfaces} \\ $\rho:\pi_1(X)\to\SL(2m+1,\R)$} \right\}}{\sim}$$ where $(\sigma_1,\rho_1)\sim(\sigma_2,\rho_2)$ if and only if there exists $A\in\SL(2m+1,\R)$ such that $\sigma_2(p)=A\cdot\sigma_1(p)$ for any $p\in \widetilde{X}$ and $\rho_2(\gamma)=A\rho_1(\gamma)A^{-1}$ for any $\gamma\in\pi_1(X)$.

\begin{prop}\label{prop:bijection_modulispace_representations}
There is a bijection between the space $\widetilde{\mathcal{ML}}(X)_m$ and isomorphism classes of stable cyclic $\SL(2m+1,\R)$-Higgs bundles as in Theorem \ref{thm:existence}.
\end{prop}
\begin{proof}
One direction has already been proven in Theorem \ref{thm:existence}, namely given a stable cyclic $\SL(2m+1,\R)$-Higgs bundles $(L_1,\dots,L_m,\mu,\gamma_1,\dots,\gamma_m)$ as defined in Section \ref{sec:4.1} with $L_{1}=K^{-1}$ and $\mu=1$, we showed the existence of an isotropic $\p$-alternating immersion $\sigma:\widetilde{X}\to\h_\tau^{2m}$ equivariant with respect to the holonomy $\rho:\pi_1(X)\to\SL(2m+1,\R)$ of the flat connection $\nabla^\tau$ (see Section \ref{sec:4.2}). \newline Conversely, using the Frenet splitting and its holomorphic interpretation explained in Section \ref{sec:3.2}, it is possible to define a stable cyclic $\SL(2m+1,\R)$-Higgs bundles of the appropriate form.\newline It remains only to show that two pairs $(\sigma_1,\rho_1)$ and $(\sigma_2,\rho_2)$ are equivalent if and only if the associated Higgs bundles are isomorphic. In this respect, according to Lemma \ref{lm:isomorphic_Higgs}, two stable cyclic $\SL(2m+1,\R)$-Higgs bundles $(L_1,\dots,L_m,\gamma_1,\dots,\gamma_m)$ and $(L_1',\dots,L_m',\gamma_1',\dots,\gamma_m')$ are isomorphic if and only if $L_i=L_i'$ for any $i=1,\dots,m$ and there is a diagonal gauge transformation $\mathcal{G}$ sending $(\gamma_1,\dots,\gamma_m)$ to $(\gamma_1',\dots,\gamma_m')$, with unit determinant and such that $Q\mathcal G^tQ=\mathcal G^{-1}$, where $Q$ is the matrix associated with the non-degenerate bi-linear form of the Higgs bundles (see Section \ref{sec:4.1}). Given that $\mu\equiv 1$, such a transformation is necessarily $\mathcal G=\diag(\lambda_m^{-1},\dots,\lambda_1^{-1},1,\lambda_1,\dots,\lambda_m)$, with $\lambda_1=1$ and $\lambda_i\in\C^*$ for any $i=2,\dots,m$. Using the holomorphic interpretation of the Frenet splitting given in Theorem \ref{thm:hol_data} and the construction of the isotropic $\p$-alternating surface in Theorem \ref{thm:existence}, it follows that $\mathcal G$ acts on $\sigma$ and changes the Frenet splitting to be consistent with the description involving the Higgs bundles. Moreover, the two representations are conjugate one to other by $\mathcal G$. The argument can clearly be reversed.
\end{proof}
\begin{remark}\label{rem:moduli_m=1}
When \( m = 1 \), the surface \( \sigma: \widetilde{X} \to \mathbb{H}_\tau^2 \) is a space-like maximal Lagrangian immersion that is equivariant under a Hitchin representation \( \rho: \pi_1(X) \to \mathrm{SL}(3, \mathbb{R}) \), which has been already studied in \cite{RT_bicomplex}, \cite{hildebrand2011cross}. The associated cyclic Higgs bundles all belong to the Hitchin component (because $L_{1}=K^{-1}$), hence their isomorphism class is completely determined by the holomorphic cubic differential $\gamma_{1}$. In other words, the space \( \widetilde{\mathcal{ML}}(X)_1 \) is biholomorphic to the complex space of holomorphic cubic differentials on \( X \).  
\end{remark}
\noindent Let us assume $m\ge 2$ and let us restrict our attention to cyclic $\SL(2m+1,\R)$-Higgs bundles as in Theorem \ref{thm:existence} such that \(L^{-1}_i \cong K^i\) and \(\gamma_{i-1} \equiv 1\) for \(i = 1, \dots, m-1\).\footnote{Note that, when \( m = 2 \), the condition just imposed is always satisfied.} In other words, these can be represented as:  
\begin{equation}\label{eq:Higgs_bundles_SL(2m+1,R)}\begin{tikzcd}[column sep=1.5em, row sep=1em]
L^{-1} \arrow[r, "\gamma_{m-1}"', bend right] & K^{m-1}  \arrow[r, "1"', bend right] & \dots \arrow[r, "1"', bend right] & K \arrow[r, "1"', bend right] & \mathcal{O}_X \arrow[r, "1"', bend right] & K^{-1} \arrow[r, "1"', bend right] & \dots\arrow[r, "1"', bend right] & K^{1-m} \arrow[r, "\gamma_{m-1}"', bend right] & L \arrow[llllllll, "\gamma_m"', bend right]
\end{tikzcd}\end{equation}
where each section in the diagram above should be considered twisted for a copy of $K$ in the target holomorphic bundle and $L:=L_m$. By assumption, $\gamma_{m-1}\in H^0(X,LK^m)$ can not be identically zero and this forces the degree of $LK^m$ to be non-negative, or in other words, $\deg(L^{-1})\le m(2g-2)$. \begin{lemma}\label{lem:stability_Higgs_bundles}
Let $(L,\gamma_{m-1},\gamma_m)$ be a cyclic $\SL(2m+1,\R)$-Higgs bundles of the form (\ref{eq:Higgs_bundles_SL(2m+1,R)}) and let $d:=\deg(L^{-1})$. Then, \begin{enumerate}
    \item[(i)] if $\gamma_m\equiv 0$ the Higgs bundle is polystable if and only if $0<d\le m(2g-2)$ and in these cases it is actually stable;
    \item[(ii)] if $\gamma_m$ is not identically zero, the Higgs bundle is polystable if and only if $1-g\leq d \leq m(2g-2)$, in which cases it is actually stable.
\end{enumerate}
\end{lemma}\begin{proof}
Starting with the case where $\gamma_m$ is identically zero, we observe that the only proper $\phi$-invariant sub-bundles are all ascending direct sums of the line bundles in which \(L^{-1}\) does not appear, for example: $$L,\quad K^{1-m}\oplus L\quad, \ \dots \ ,\quad \mathcal{O}_X\oplus K^{-1}\dots\oplus K^{1-m}\oplus L\quad, \ \dots \ , \quad K^{m-1}\oplus\dots\oplus K^{1-m}\oplus L.$$ Therefore, if $0<d\le m(2g-2)$ all the holomorphic bundles appearing above have strictly negative degree, and therefore the associated $\mathrm{SL}(2m+1, \mathbb{R})$-Higgs bundle (\ref{eq:Higgs_bundles_SL(2m+1,R)}) is stable. Conversely, if the Higgs bundle is stable, we must have $-d=\deg(L)< 0$ which is equivalent to $d>0$. The inequality $d\le m(2g-2)$ is always true according to the discussion prior to the statement of the lemma. When $d=0$, $L$ is a degree $0$ $\phi$-invariant subbundle without $\phi$-invariant complement, thus $(E,\phi)$ is not polystable.  Finally, whenever $d<0$ the line bundle $L$ is a destabilizing proper $\phi$-invariant sub-bundle having strictly positive degree. Therefore, for any $d<0$ the cyclic $\SL(2m+1,\R)$-Higgs bundle (\ref{eq:Higgs_bundles_SL(2m+1,R)}) is not polystable. \newline 
If $\gamma_{m}$ is not identically zero, there is no proper $\phi$-invariant subbundle compatible with the cyclic splitting (see Proposition \ref{prop:stability_cyclic}), so $(E,\phi)$ is automatically stable. However, since $\gamma_{m}\in H^{0}(X,L^{-2}K)$, such $\gamma_{m}$ exists if and only if $d \geq 1-g$. 
\end{proof} 
\noindent As a consequence of Proposition \ref{prop:bijection_modulispace_representations}, we obtain that the set of equivalence classes of Higgs bundles as in (\ref{eq:Higgs_bundles_SL(2m+1,R)}) is in bijection with a subspace of \(\widetilde{\mathcal{ML}}(X)_m\), denoted by \(\mathcal{ML}(X)_m\). It is easily deduced that the two spaces coincide only when \(m = 2\). Moreover, the map $\boldsymbol\deg:\mathcal{ML}(X)_m\to\Z$ which takes the equivalence class $[(L,\gamma_{m-1},\gamma_m)]$ to $d:=\deg(L^{-1})$ is continuous, hence we obtain a decomposition $$\mathcal{ML}(X)_m=\bigsqcup_{d\in\Z}\mathcal{ML}_d(X)_m$$ labeled by the degree of $L^{-1}$, where $\mathcal{ML}_d(X)_m=\boldsymbol{\deg}^{-1}(d)$. 
\begin{remark}
From a geometric point of view, when $m\ge 3$, the space \(\mathcal{ML}(X)_m\) can be thought of as the set of equivariant isotropic $\p$-alternating immersions with an additional property: the $\mathrm{Hom}$-valued $1$-forms $\eta_{i}$, seen as holomorphic $1$-forms as in Theorem \ref{thm:hol_data}, induce isomorphisms between $L_{i-1}$ and $L_{i}\otimes K$ for all $i$, except possibly for $i=m,m+1$.
\end{remark}
\noindent The subsequent step is to present a parameterization of the spaces introduced above as holomorphic bundles over a symmetric product of the surface, a situation similar to the one studied by Hitchin (\cite[Theorem 10.8]{hitchin1987self}) for $\PSL(2,\R)$ using the Euler number (see also \cite{collier2018psl} for the $\mathrm{SO}_0(2,1)$ version). We first recall a technical lemma that will be useful during the proof. \begin{lemma}\label{lem:twisted_line_bundle}
Let $\mathcal N\to X$ be a holomorphic line bundle and let $l\ge 1$ be an integer such that $\deg(\mathcal N)\in[0,l(2g-2))\cap\mathbb Z$, then there exists an effective divisor $D$ such that $\mathcal{N}\cong K^l(-D)$.
\end{lemma}\begin{proof}
Using the correspondence between line bundles and divisors on $X$ we know there exists a divisor $D_{\mathcal N}$ such that $\mathcal N\cong\mathcal O(D_\mathcal N)$ and $\deg(D_\mathcal N)=\deg(\mathcal N)\in[0,l(2g-2))\cap\Z$. Moreover, $K^l\cong\mathcal O(\divi(q_l))$ where $q_l$ is a meromorphic $l$-differential over $X$. In particular, $\mathcal N\cong K(-D')$ if and only if $\mathcal O(D_\mathcal N)\cong\mathcal O(\divi(q_l)-D')$ if and only if there exists a meromorphic function $f$ on $X$ such that $\divi(f)=D_\mathcal N-\divi(q_l)+D'$. Therefore, $\deg(D')=l(2g-2)-\deg(\mathcal N)$ and it is strictly positive by hypothesis. The divisor $D'$ is effective if and only if $\divi(f\cdot q_l)-D_\mathcal N$ is effective. For this purpose, if the meromorphic function $f$ is such that the divisor $\divi(f\cdot q_l)-D_\mathcal N$ is effective then we conclude the proof. Otherwise, we can choose an other non-zero meromorphic function $h$ such that $\divi(h\cdot f\cdot q_l)-D_\mathcal N$ is effective, which is equivalent to being effective for $\widetilde D:=D'+\divi(h)$. We observe that such a non-zero function $h$ always exists being a holomorphic section $\mathcal{O}(\divi(f\cdot q_l)-D_\mathcal N)$, which has strictly positive degree by hypothesis. We conclude by noting that $\mathcal N\cong K^l(-\widetilde D)$. 
\end{proof} \begin{theorem}\label{prop:parameterization_moduli_space}
Let $X$ be a fixed Riemann surface structure on a closed connected surface $\Sg$ of genus $g\ge 2$, then \begin{enumerate}
    \item[(i)] if $0<d\le m(2g-2)$ the set $\mathcal{ML}_d(X)_m$ is biholomorphic to the total space of a holomorphic vector bundle of complex rank $2d+g-1$ over the $(m(2g-2)-d)$-th symmetric product of $X$;
    \item[(ii)] if $1-g\le d\le 0$ the set $\mathcal{ML}_d(X)_m$ is biholomorphic to a bundle over an $H^{1}(X,\Z_{2})$-cover of the $(2d+2g-2)$-symmetric product of $X$ whose fiber is $(\C^{(2m-1)(g-1)-d} \setminus \{0\}) / \{\pm \mathrm{Id}\}$.
    \item[(iii)] if $d\notin[1-g,\dots,m(2g-2)]\cap\Z$ then $\mathcal{ML}_d(X)_m$ is empty.
\end{enumerate}Moreover, the space $\mathcal{ML}_d(X)_m$ is smooth for any $d\in[1-g,m(2g-2)]\cap\Z$.
\end{theorem}
\begin{proof}
$(i)$ Let us start with the case $d=m(2g-2)$, so that $L^{-1}$ is necessarily isomorphic to $K^m$, $\gamma_{m-1}\equiv 1$, and $\gamma_m=q_{2m+1}$ is a holomorphic differential on $X$ of order $2m+1$. As a consequence, the set of equivalence classes $[(K^m,q_{2m+1})]$ is in bijection with $H^0(X,K^{2m+1})$ and therefore $\mathcal{ML}_{m(2g-2)}(X)_m$ is biholomorphic to the space of holomorphic $(2m+1)$-differentials over $X$, which has complex dimension equal to $(1+4m)(g-1)$. \newline Let us assume $d\in (0,m(2g-2))$ and define the set $$\widetilde{\mathcal{E}}_d^m:=\{(\mathcal N,\alpha,\beta) \ | \deg(\mathcal{N}^{-1})=d, \alpha\in H^0(X,K^m\mathcal N)\setminus\{0\}, \ \beta\in H^0(X,\mathcal{N}^{-2}K)\},$$then applying Lemma \ref{lem:twisted_line_bundle} to any such $\mathcal{N}^{-1}$ we get $\mathcal{N}^{-1}\cong K^m(-D)$ for some effective divisor $D$ with $\deg(D)=m(2g-2)-d$. In particular, $\alpha$ can be thought of as an element in $H^0(X,\mathcal{O}(D))$, namely is the holomorphic section defining $D$, and $\beta$ can be thought of as an element in $H^0(X,K^{2m+1}(-2D))$. There is a map $\widetilde\Phi:\widetilde{\mathcal{E}}_d^m\to\mathcal{ML}_d(X)_m$ sending the triple $(K^{-m}(D),\alpha,\beta)$ to the equivalence class of the Higgs bundle in (\ref{eq:Higgs_bundles_SL(2m+1,R)}) with $L\cong K^{-m}(D)$, $\gamma_{m-1}=\alpha$ and $\gamma_m=\beta$, which is surjective by construction. Moreover, two points $(K^{-m}(D),\alpha,\beta)$ and $(K^{-m}(D'),\alpha',\beta')$ in $\widetilde{\mathcal{E}}_d^m$ have the same image under $\widetilde\Phi$ if and only if $D=D'$, $\alpha=\lambda\alpha'$ and $\beta'=\lambda^{-2}\beta$ for some $\lambda\in\C^*$ (see Lemma \ref{lm:isomorphic_Higgs}). Therefore, the quotient $\mathcal{E}_d^m:=\widetilde{\mathcal{E}}_d^m/\C^*$ is diffeomorphic to the moduli space $\mathcal{ML}_d(X)_m$.
Let us now denote by $S^{m(2g-2)-d}(X)$ the $(m(2g-2)-d)$-th symmetric product of the surface, which is naturally in bijection with the set of effective divisors of degree $m(2g-2)-d$. The projection map $\pi:\mathcal{E}_d^m\to S^{m(2g-2)-d}(X)$ sending the triple $[(K^{-m}(D),\alpha,\beta)]$ to $\mathrm{div}(\alpha)=D\in S^{m(2g-2)-d}(X)$ is surjective and its fiber, over a degree $m(2g-2)-d$ effective divisor $D'$ is isomorphic to $H^0(X,K^{m}(-2D')$: a complex vector space of complex dimension $2d+g-1$ by a simple application of Riemann-Roch theorem. Hence, $\mathcal{E}_{d}^m$, and consequently $\mathcal{ML}_{d}(X)_{m}$ is biholomorphic to the total space of a holomorphic vector bundle of complex rank $2d+g-1$ over $S^{m(2g-2)-d}(X)$.
\newline $(ii)$ The $(2d+2g-2)$-symmetric product $S^{2d+2g-2}(X)$ of $X$ has an $H^{1}(X,\Z_{2})$-cover $\widehat{S}^{2d+2g-2}(X)$ which parameterizes pairs $(\mathcal{V}, [\beta])$, where $\mathcal{V}$ is a line bundle of degree $d+g-1$ and $[\beta] \in \mathbb{P}(H^{0}(X, \mathcal{V}^{2}) \setminus \{0\})$. After fixing a square root of $K$, there is a well-defined injective map
\begin{align*}
    \Phi: \mathcal{ML}_{d}(X)_{m} &\rightarrow \widehat{S}^{2d+2g-2}(X) \times H^{0}(X,K^{2m+1}) \\
     [(L,\gamma_{m-1},\gamma_{m})] &\mapsto ( (L^{-1}K^{1/2}, [\gamma_{m}]), \gamma_{m-1}^{2}\gamma_{m}) . 
\end{align*}
Let $W$ denote the image of $\Phi$. Projection onto the first factor defines a surjective map $\pi: W \rightarrow \widehat{S}^{2d+2g-2}(X)$. The fiber $\pi^{-1}((\mathcal{V},[\beta]))$ consists of all holomorphic $(2m+1)$-differentials $q$ that can be written as $q=\alpha^{2}\beta$ for some $\alpha \in H^{0}(X, \mathcal{V}^{-1}K^{m+1/2})$. Therefore, the fiber can be identified with the image of the quadratic map $H^{0}(X, \mathcal{V}^{-1}K^{m+1/2}) \setminus \{0\} \rightarrow H^{0}(X,\mathcal{V}^{-2}K^{2m+1})$ and, in particular, it is biholomorphic to $(\C^{(2m-1)(g-1)-d} \setminus \{0\}) / \{\pm \mathrm{Id}\}$. 
\newline $(iii)$ If $d>m(2g-2)$ then $\deg(LK^m)=-d+m(2g-2)<0$. Therefore, the holomorphic section $\gamma_{m-1}$ is zero everywhere and $L^{-1}$ becomes a proper $\phi$-invariant sub-bundle with positive degree. In other words, the corresponding $\SL(2m+1,\C)$-Higgs bundle is not (poly)stable and there is no solution to Hitchin's equations. Following the same reasoning, if $d<1-g$ then $\gamma_m\equiv 0$. In this case, $L$ is a proper $\phi$-invariant sub-bundle with positive degree $-d>g-1$.\newline Finally, by Proposition \ref{prop:bijection_modulispace_representations} and Lemma \ref{lem:stability_Higgs_bundles} any point in $\mathcal{ML}_d(X)_m$ for $d\in[1-g,m(2g-2)]\cap\Z$ corresponds to an equivalence class of stable Higgs bundles, which represents smooth points in the corresponding moduli space (see \cite[Proposition 3.18]{HK_correspondence}).
\end{proof}
\begin{remark}
As previously mentioned, when \(m = 2\), the subspace \(\mathcal{ML}(X)_2\) coincides with the entire space \(\widetilde{\mathcal{ML}}(X)_2\) of equivariant isotropic $\p$-alternating immersions in \(\mathbb{H}_\tau^4\), to which all the results of Proposition \ref{prop:parameterization_moduli_space} apply.
\end{remark}
\noindent In the final part of this section, we study a family of equivariant isotropic \(\p\)-alternating surfaces that lie in a totally geodesic copy of $\h^{2m-2}_{\tau}$ inside $\h^{2m}_{\tau}$. We consider Higgs bundles as in Equation \eqref{eq:Higgs_bundles_SL(2m+1,R)} with \(\gamma_{m-1} \equiv 0\) and \(1-g \le \deg(L^{-1}) = d \leq 0\): they are strictly polystable because they can be written as 
\begin{equation}\label{eq:Higgs_bundles_SL(2m+1,R)_totally_geodesic}\begin{tikzcd}[column sep=1em, row sep=1em]
K^{m-1} \arrow[r, "1"', bend right] & \dots \ \  K \arrow[r, "1"', bend right] & \mathcal{O}_X \arrow[r, "1"', bend right]  & K^{-1} \dots \arrow[r, "1"', bend right] & K^{1-m} & \oplus & L \arrow[r, "\gamma_m"', bend right] & L^{-1}
\end{tikzcd}\end{equation} with $\gamma_m\in H^0(X,L^{-2}K)\setminus\{0\}$. By polystability, the representations $\rho$ arising from these Higgs bundles are reducible and split as $\rho=\rho_1\oplus\rho_2$ where $\rho_{1}:\pi_1(S)\to\SL(2m-1,\R)$ is the uniformizing Hitchin representation and $\rho_2:\pi_1(S)\to\SL(2,\R)$ has Euler number equal to $-d\in[0,g-1]\cap\Z$. In particular, the representation \(\rho\) is never contained in the $\SL(2m+1,\R)$-Hitchin component (being reducible, for example), but it belongs to one of the other two components depending on the parity of the degree of \(L^{-1}\): if \(d\) is even (resp. odd), then it is contained in the component with \(sw_2 = 0\) (resp. \(sw_2 \neq 0\) \cite[\S 10]{hitchin1992lie}). As stated in the proof of Theorem \(\ref{thm:hol_data}\), since \(\gamma_{m-1} \equiv 0\), and therefore \(\eta_m \equiv 0\), the immersed surface \(\sigma: \widetilde{X} \to \mathbb{H}_\tau^{2m}\) is contained in a totally geodesic copy of a para-complex hyperbolic space \(\mathbb{H}_\tau^{2m-2}\) of real codimension 4. Using the decomposition induced by the Frenet splitting of the surface (see Definition \ref{def:Frenet}), we can decompose $\sigma^*T\mathbb{H}_\tau^{2m}$ as follows: 
\[
T\widetilde{X} \oplus \mathscr{L}_2 \oplus \dots \oplus \dots\mathscr{L}_{m-1}\oplus\p(\mathscr{L}_{m-1})\oplus\dots\oplus\p(T\widetilde X) \ \oplus \ \mathscr{L}_m\oplus\p(\mathscr{L}_m)
\] where the first $2m-2$ summands correspond to $\sigma^{*}T\h^{2m-1}_{\tau}$ and $\mathscr{L}_m\oplus\p(\mathscr{L}_m)$ gives its normal bundle inside $\sigma^{*}T\h^{2m}_{\tau}$. Using the holomorphic interpretation of the \(\mathscr{L}_i\)'s given in Theorem \(\ref{thm:hol_data}\), it follows that the Euler class of the orientable bundle \(\mathscr{L}_m\) equals the Chern class of its holomorphic counterpart \(L\equiv L_m\). In other words, the Euler class of the representation \(\rho_2\) is determined by the topology of the component of the normal bundle to $\sigma(\widetilde{X})$ that is also normal to \(\mathbb{H}_\tau^{2m-2}\). 
\begin{remark} When \(m=2\), the equivariant immersion for \(\rho_1 \oplus \rho_2\) is contained in the totally geodesic copy of \(\mathbb{H}_\tau^2\) inside \(\mathbb{H}_\tau^4\). In this case, \(\gamma_1 \equiv 0\) and coincides precisely with the \((1,0)\)-part of the second fundamental form $\II$ of the immersion. Consequently, we obtain a totally geodesic space-like maximal Lagrangian surface in \(\mathbb{H}_\tau^2\), which is nothing but a copy of \(\mathbb{H}^2\) (see Example \(\ref{ex:maximal_Lagrangian}\)).
\end{remark} 
\noindent Let us now explain how these particular surfaces can be obtained as limits of those corresponding to stable Higgs bundles. Given a Higgs bundle \((L, \gamma_{m-1}, \gamma_{m})\) as in (\ref{eq:Higgs_bundles_SL(2m+1,R)}) with \(1-g \leq \deg(L^{-1}) = d \leq 0\), the natural action of \(\mathbb{C}^*\) leaves \(L\) invariant and maps \((\gamma_{m-1}, \gamma_{m})\) to \((t\gamma_{m-1}, t\gamma_{m})\) for \(t \in \mathbb{C}^*\). Using the gauge transformation \(\mathcal{G} := \operatorname{diag}(t^{-\frac{1}{2}}, 1, \dots, 1, t^{\frac{1}{2}})\), the holomorphic structure of the bundle remains unchanged, but \(\mathcal{G} \cdot (\gamma_{m-1}, \gamma_{m}) = (t^{\frac{1}{2}} \gamma_{m-1}, t^{-1} \gamma_m)\) (see the proof of Proposition \ref{prop:bijection_modulispace_representations}). In particular, the bundle \((L, t\gamma_{m-1}, t\gamma_m)\) is isomorphic to \((L, t^{\frac{3}{2}} \gamma_{m-1}, \gamma_m)\), whose equivalence class corresponds to a continuous family of equivariant isotropic $\p$-alternating surfaces \([\sigma_t, \rho_t] \in \mathcal{ML}_d(X)_m\). Taking the limit \(t \to 0\), the equivalence class \([\rho_t, \sigma_t]\) converges to an isotropic $\p$-alternating surface in \(\mathbb{H}_\tau^{2m-2} \subset \h^{2m}_{\tau}\). Moreover, if \(1-g \leq d \leq 0\), thanks to Proposition \ref{prop:parameterization_moduli_space}, we know that the space \(\mathcal{ML}_d(X)_m\) is biholomorphic to the quotient of the total space of the complement of the zero section of a complex rank \( (2m-1)(g-1)-d\) bundle by \(\pm\mathrm{Id}\). As described above, the zero section of this bundle corresponds precisely to isotropic $\p$-alternating immersions lying in a totally geodesic copy of $\h^{2m-2}_{\tau} \subset \h^{2m}_{\tau}$ (see also \cite{collier2023holomorphic} for the $\mathrm{G}_2'$ case).

\subsection{Geometry of cyclic representations}
By means of infinitesimal rigidity (see Section \ref{sec:inf_rigidity}), we study the holonomy map $\mathrm{Hol}$ that associates to each $\rho$-equivariant isotropic $\p$-alternating surface the conjugacy class of $\rho$. In particular, when \(m = 1\), the map $\mathrm{Hol}$ is precisely the Labourie-Loftin parametrization of the Hitchin component, which we recover using only the geometry of the para-complex hyperbolic space. When \(m \geq 2\), on the one hand, we are able to study $\mathrm{Hol}$ on a large subset so that its image is not only contained in the Hitchin component, but on the other hand, we obtain partial results, as in the case of Labourie (\cite{labourie2017cyclic}) and Collier-Toulisse (\cite{collier2023holomorphic}) using \emph{cyclic surfaces}, or in the case of Nie (\cite{Nie_alternating}) using \emph{A-surfaces}. \\ \\ 
We define the moduli space of equivariant isotropic \(\p\)-alternating immersions in \(\h_\tau^{2m}\) as $$\widetilde{\mathcal{ML}}(\Sigma)_m:=\bigslant{\left\{(\sigma,\rho) \ \Bigg| \ \parbox{15em}{$\sigma: \widetilde{\Sigma}\to\h_\tau^{2m}$ as in Definition \ref{def:P_alternating_surfaces} \\ $\rho:\pi_1(\Sigma)\to\SL(2m+1,\R)$} \right\}}{\sim}$$ where $(\sigma_1,\rho_1)\sim(\sigma_2,\rho_2)$ if and only if there exists $A\in\SL(2m+1,\R)$ and $\varphi\in\Diff_0(\Sigma)$ such that $\sigma_2(p)=A\cdot(\sigma_1\circ\varphi)(p)$ for any $p\in \widetilde{\Sigma}$ and $\rho_2(\gamma)=A\rho_1(\gamma)A^{-1}$ for any $\gamma\in\pi_1(\Sigma)$. There is a natural projection map \begin{align*}
    \pi: \ &\widetilde{\mathcal{ML}}(\Sigma)_m\to\mathcal{T}(\Sigma) \\ & [(\sigma,\rho)]\longmapsto [J] 
\end{align*}where \(J\) is the complex structure induced by the conformal class of the first fundamental form of the immersion. The fiber of this map over $J$ is exactly the space $\widetilde{\mathcal{ML}}(X)_{m}$, with $X=(S,J)$, studied in the previous subsection. This projection is equivariant with respect to the action of the mapping class group \(\mathrm{MCG}(\Sigma) := \Diff_+(\Sigma)/\Diff_0(\Sigma)\), which acts naturally on the space \(\widetilde{\mathcal{ML}}(\Sigma)_m\) by lift of the action at the level of the Teichm\"uller space. \\

\noindent We denote by $\widetilde{\mathcal{ML}}(\Sigma)_{m}^{*}$ the subset of equivariant isotropic $\p$-alternating surfaces in $\h^{2m}_{\tau}$ whose holomorphic data satisfy the assumptions in Equation \eqref{eq:assumption}. 

\begin{theorem}\label{thm:holonomy} For $t\in (-\epsilon, \epsilon)$ let $[(\sigma_{t}, \rho_{t})]$ be a smooth path in $\widetilde{\mathcal{ML}}(\Sigma)_{m}^{*}$ such that $\dot{\rho}=\frac{d}{dt}\rho_{|_{t=0}}$ is zero in $T_{\rho_{0}}\chi_{m}(\Sigma)$. Then the tangent vector $[(\dot{\sigma}, \dot{\rho})]$ is trivial in $T\widetilde{\mathcal{ML}}(\Sigma)_{m}^{*}$.
\end{theorem}
\begin{proof} By assumption, there exists a smooth path $A_{t} \in \SL(2m+1,\R)$ such that $\dot{\rho}':=\frac{d}{dt} {A_{t}\rho_{t}A_{t}^{-1}}_{|_{t=0}}=0$. We consider the smooth $1$-parameter family of equivariant immersions $(\sigma_{t}'= A_{t}\sigma_{t}, \rho_{t}':=A_{t}\rho_{t}A_{t}^{-1})$, which is equivalent to $(\sigma_{t}, \rho_{t})$ in $\widetilde{\mathcal{ML}}(\Sigma)_{m}^{*}$ by definition. Taking the derivative at $t=0$ of the equivariance relation
\[
    \sigma_{t}'(\gamma \cdot p) = \rho'_{t}(\gamma)\sigma'_{t}(p) \ \ \  \ \ \text{for all $\gamma\in \pi_{1}(\Sigma)$ and $p\in \widetilde{\Sigma}$,}  
\]
we obtain that the variational vector field $\dot{\sigma}'$ is $\pi_{1}(\Sigma)$-invariant. As a consequence, we can see $\dot{\sigma}'$ as a vector field on $\Sigma$ with normal component that is a Jacobi vector field by Theorem \ref{thm:variation_area}. By Corollary \ref{cor:no_Jacobi}, we conclude that $\dot{\sigma}$ is purely tangential, thus there exists a smooth $1$-parameter family $\varphi_{t}$ of diffeomorphisms of $\Sigma$ with $\varphi_{0}=\mathrm{Id}$ such that $\dot{\varphi}=-\dot{\sigma}'$. The new family of immersions $\sigma_{t}'':= A_{t} \cdot (\sigma_{t} \circ \varphi_{t})$ is still $\rho_{t}'$-equivariant and in the same equivalence class as $(\sigma_{t}, \rho_{t})$. By construction $\dot{\sigma}''=0$, hence the tangent vector $[(\dot{\sigma}, \dot{\rho})]$ is trivial in $T\widetilde{\mathcal{ML}}(\Sigma)_{m}^{*}$.
\end{proof}
\begin{cor}\label{cor:immersion}
The holonomy map $\mathrm{Hol}:\widetilde{\mathcal{ML}}(\Sigma)_{m}^{*}\longrightarrow\chi_m(\Sigma)$ given by $\mathrm{Hol}\big([\sigma,\rho]\big):=[\rho]$ is an immersion for any $m\ge 1$.
\end{cor}
\noindent A particular class of equivariant isotropic $\p$-alternating immersions in $\mathbb{H}^{2m}_{\tau}$ belonging to $\widetilde{\mathcal{ML}}(\Sigma)_{m}^{*}$ is the one associated to cyclic $\SL(2m+1,\R)$-Higgs bundles in the Hitchin component (see Section \ref{sec:4.1}). In this setting, Theorem \ref{thm:holonomy} recovers the infinitesimal rigidity of the cyclic locus in the Hitchin components proved by Labourie with completely different techniques (\cite{labourie2017cyclic}). Finally, when \( m = 1 \), we recover the Labourie-Loftin parametrization of the Hitchin component of \( \mathrm{SL}(3,\mathbb{R}) \) as a holomorphic vector bundle over the Teichmüller space of the surface (\cite{loftin2001affine, Labourie_cubic}) by adapting Labourie's argument from Labourie's cyclic surfaces to our setting of isotropic $\p$-alternating immersions in $\h^2_\tau$. 

\begin{cor}\label{cor:Labourie-Loftin} The map $\mathrm{Hol}:\mathcal{ML}(\Sigma)_{1}\longrightarrow \mathrm{Hit}_3(\Sigma)\subset \chi_1(\Sigma)$ is a diffeomorphism. In particular, $\mathrm{Hit}_3(\Sigma)$ is diffeomorphic to the bundle of holomorphic cubic differentials over the Teichm\"uller space of $\Sigma$.
\end{cor}
\begin{proof} Let $P=\mathrm{Teich}(\Sigma) \times \mathrm{Hit}_3(\Sigma)$ and let $p$ be the projection onto the second factor. Let $F$ be the function that associates to every $(J,\rho) \in P$ the energy of the unique $\rho$-equivariant harmonic map from $X=(\Sigma, J)$ to the symmetric space $\SL(3,\R)/\SO(3)$. By \cite{labourie2008cross}, the map $F$ is smooth, positive, and the restriction $F_\delta$ of $F$ to each fiber $p^{-1}(\delta)$ is proper. From Corollary \ref{cor:immersion}, the map 
\begin{align*}
    \Psi:  \mathcal{ML}(\Sigma)_1 &\rightarrow P \\
            [(\sigma, \rho)] &\mapsto (J, \mathrm{Hol}(\sigma, \rho))
\end{align*}
where $J$ is the complex structure compatible with the first fundamental form of $\sigma$, is transverse to the fibers of $p$. Note that $\mathrm{Hol}$ indeed takes values in the Hitchin component because $\mathrm{Hol}(\sigma, \rho)$ is also the holonomy of the cyclic $\SL(3,\R)$-Higgs bundle associated to the equivariant isotropic $\p$-alternating immersion in $\h^2_\tau$, which belongs to the Hitchin component by construction. Moreover, $\Psi$ is an embedding and its image $N$ is precisely given by the pairs $(J,\rho) \in P$ such that the $\rho$-equivariant harmonic map from $X=(S,J)$ to $\SL(3,\R)/\SO(3)$ is conformal: in fact, such harmonic map can be obtained as Gauss map of the equivariant isotropic $\p$-alternating immersion $\h^2_\tau$ from Section \ref{sec:4.4}. Therefore, the subbundle $N \subset P$ satisfies the assumptions of \cite[Theorem 8.1.1]{labourie2017cyclic}, hence the restriction of $p$ to $N$ is a diffeomorphism. In other words the map $\mathrm{Hol}:\mathcal{ML}(\Sigma)_{1}\longrightarrow \mathrm{Hit}_3(\Sigma)$ is a diffeomorphism.
\end{proof}

\subsection{Geometric structures}
In this section, we construct geometric structures on a fiber bundle over the surface \(\Sigma\) starting from a Hitchin representation in $\SL(3,\R)$ and show that they coincide with those defined by Guichard-Wienhard (\cite{guichard2012anosov}) using the Anosov property. We will briefly recall the classical definition of an \((X, G)\)-structure (see \cite{goldman2022geometric} for a more detailed exposition) and their equivalence in terms of an \((X,G)\)-\emph{developing pair}. We will also make use of the notion of the graph of a \((X,G)\)-structure (see the survey \cite{alessandrini2019higgs}), which will provide a method for defining such a structure based on a $\rho$-equivariant space-like maximal Lagrangian in $\h_\tau^2$. \\ \\ Let us consider a smooth manifold $X$ endowed with an effective and transitive action by a real (or complex) Lie group $G$. An $(X,G)$ \emph{geometric structure} on a smooth manifold $M$ with $\dim(M)=\dim(X)$ is the datum of a maximal atlas $\big\{(U_i,\varphi_i)\big\}_{i\in I}$ for which the transitions function $\varphi_i\circ\varphi_j^{-1}:\varphi_j\big(U_i\cap U_j\big)\to\varphi_i\big(U_i\cap U_j\big)$ are \emph{locally in $G$}, namely they restricts to some $g_{ij}\in G$ on any connected component of $\varphi_j\big(U_i\cap U_j\big)$. The notion of a \((X, G)\) geometric structure just described is equivalent to the existence of a pair of maps \((\text{dev}, h)\), where \(h: \pi_1(M) \to G\) is a group homomorphism called the \emph{holonomy representation} and \(\text{dev}: \widetilde{M} \to X\) is an \(h\)-equivariant map called the \emph{developing map}, where \(\widetilde{M}\) denotes the universal cover of \(M\). To obtain a well-defined parameter space of \((X, G)\)-structures, two of them must be declared equivalent (or \emph{isotopic}) if there exists a diffeomorphism \(\phi: M \to M\) isotopic to the identity that is also an isomorphism of \((X, G)\)-structures, meaning that it is compatible with the earlier definition given by the maximal atlas. The set \(\mathcal{D}_{(X,G)}(M)\) of all \((X, G)\)-structures up to isotopy can be endowed with the compact-open topology using its equivalence with developing pairs \((\text{dev}, h)\). In particular, if \(\chi(M,G)\) denotes the space of representations of \(\pi_1(M)\) into \(G\) up to conjugation, the well-known \emph{Ehresmann-Thurston principle} (\cite{thurston_notes}) states that the holonomy map  
\[
\text{hol}: \mathcal{D}_{(X,G)}(M) \to \chi(M,G)
\]  
is continuous, locally injective, and the following property holds: given an \((X, G)\)-structure with holonomy \(\rho\), any other representation \(\rho'\) sufficiently close to \(\rho\) is the holonomy of an \((X, G)\)-structure (nearby) to the original one. Finally, we briefly explain a different method for defining an $(X,G)$-structure on a closed manifold $M$ based on the theory of flat bundles. Let \(\rho: \pi_1(M) \to G\) be a representation, and let \(\pi: B \to M\) be a flat bundle with fiber \(X\) and holonomy \(\rho\). A section \(s \in \Gamma(M, B)\) is called \emph{transverse} if it is transverse to the parallel foliation of the flat bundle. Since smooth sections of \(B\) correspond uniquely to \(\rho\)-equivariant maps from \(\widetilde{M}\) to \(X\), transverse sections correspond to equivariant maps which are local diffeomorphisms, whenever \(\dim(M) = \dim(X)\). In particular, they allow the construction of a developing pair \((\text{dev}_\rho, h)\) and thus an \((X, G)\)-structure on \(M\). \\ \\ Let \(\Sigma\) be a closed, oriented surface of genus \(g \geq 2\). We now focus on the case of interest, namely when \(G = \mathrm{SL}(3, \mathbb{R})\) and the space \(X\) is given by the flag variety $$\mathcal{F}:=\{(F_1,F_2) \ | \ F_1\subset F_2\subset\R^3, \ \dim F_i=i, \ \text{for} \ i=1,2\}$$realized as the quotient of \(\mathrm{SL}(3, \mathbb{R})\) by the parabolic subgroup \(P_{1,2}\), which stabilizes pairs of transverse lines and planes in \(\mathbb{R}^3\). Given a Hitchin representation \(\rho: \pi_1(\Sigma) \to \mathrm{SL}(3, \mathbb{R})\), Guichard-Wienhard (\cite{guichard2012anosov}) proved that there exists an open \emph{domain of discontinuity} \(\Omega_\rho \subset \mathcal{F}\), on which the action of \(\Gamma := \rho(\pi_1(\Sigma))\) is properly discontinuous and co-compact. This open set \(\Omega_\rho\) consists of three connected components (\cite{barbot2010three}), and the one we will focus on is given by  
\[
\Omega_\rho^{dS} = \{(p, \Ker(\varphi)) \in \mathcal{F} \ | \ p \notin \overline{\mathcal{C}}, \ \Ker(\varphi)\cap \mathcal{C} \neq \emptyset\},
\] where $\mathcal{C}$ is the properly convex subset of $\R\mathbb P^2$ preserved by the action of $\rho$ (\cite{goldman1990convex}) and $\varphi$ is a functional such that $\varphi(p)=0$. The manifold \(M\) on which we aim to define the \((\mathcal{F}, \mathrm{SL}(3, \mathbb{R}))\)-structure is the projectivezed tangent bundle of the surface, denoted by \(\Pp T\Sigma\). It is already known that the quotient of each connected component of the Guichard-Wienhard domain \(\Omega_\rho\) by \(\Gamma\) is homeomorphic to the projectivized tangent bundle of the surface. Our main result consists in showing that an \((\mathcal{F}, \mathrm{SL}(3, \mathbb{R}))\)-structure on \(\Pp T\Sigma\) can be defined from the isotropic $\p$-alternating immersions we introduced in Section \ref{sec:definition_surfaces}. Moreover, using a continuity argument based on the Ehresmann-Thurston principle, we show that this construction recovers exactly the structure defined by Guichard-Wienhard. A similar approach was used in \cite{CTT} for Hitchin representations in \(\mathrm{SO}_0(2,3)\), and more recently, Nolte-Reistenberg (\cite{nolte2024concave}) studied \((\mathcal{F}, \mathrm{SL}(3, \mathbb{R}))\)-structures on \(\Pp T\Sigma\) with a completely different approach. \\ \\ Let \(\sigma: \widetilde{\Sigma} \to \mathbb{H}_\tau^2\) be the \(\rho\)-equivariant space-like maximal Lagrangian immersion associated with a Hitchin representation $\rho$ in $\SL(3,\R)$ (see Remark \ref{rem:moduli_m=1} and Corollary \ref{cor:Labourie-Loftin}), and consider one of its horizontal lifts \(\hat{\sigma}: \widetilde{\Sigma} \to \mathbb{R}^3_\tau\) such that \(\boldsymbol{q}(\hat{\sigma}, \hat{\sigma}) = -1\). Since everything is \(\rho\)-equivariant, we can consider the objects defined on the closed surface \(\Sigma = \widetilde{\Sigma} / \rho\). In particular, we have the following orthogonal splitting:  
\[
E_\rho := \hat{\sigma}^* T\mathbb{R}_\tau^3 = \mathbb{R} \hat{\sigma} \oplus T\Sigma \oplus \p\big(T\Sigma\big) \oplus \p\big(\mathbb{R} \hat{\sigma}\big)
\] where $E_\rho$ is the flat bundle over $\Sigma$ whose holonomy equals $\rho$. Moreover, the matrix of $1$-forms associated with the flat connection $D$ on $E_\rho$ is given by $$A=\begin{pmatrix}
    0 & \theta^t & 0 & 0 \\ \theta & A_1 & \eta_2^\dag & 0 \\ 0 & \eta_2 & A_2 & \theta \\ 0 & 0 & \theta^t & 0
\end{pmatrix}$$ where $\theta^t=(\theta_1,\theta_2)$ is the frame of $T^*\Sigma$ dual to a local orthonormal frame $\{u,v\}$ of $T\Sigma$ (see the proof of Theorem \ref{thm:hol_data}) with respect to which the matrices \(A_1, A_2, \eta_2, \eta_2^\dag\) are computed\footnote{A different notation has been used compared to Theorem \ref{thm:hol_data} to avoid confusion with the domain \(\Omega_\rho\).}. Let $\pi:T^1\Sigma\to\Sigma$ be the natural projection from the unit tangent bundle to the surface. There is a tautological section $s_1$ of the flat bundle $\pi^*E_\rho$ over $T^1\Sigma$ with values in the sub-bundle $\pi^*T^1\Sigma$, so that $\boldsymbol{q}(s_1,s_1)=1$. Moreover, one can define a section \( s_2 \) of \( \pi^* \p(T^1 \Sigma) \) obtained by first applying a rotation of \( \frac{\pi}{2} \) on the tangent plane to \( \Sigma \), then the section \( s_1 \), and finally the para-complex structure \( \p \) on the unit vector part. It is clear from the definition that $\boldsymbol{q}(s_2,s_2)=-1$, hence the sum $s:=s_1+s_2$ defines an isotropic line sub-bundle of $\pi^*E_\rho$. In other words, the map \( s \) can be seen as a section of the sub-bundle \( \text{Iso}_1(\pi^*E_\rho) \), whose fiber over a point \( (p,v) \in T^1\Sigma \) can be identified with the space of isotropic vectors in \( (E_\rho)_p \). At this point, it is important to clarify the relationship between this space and \( \mathbb{H}_\tau^2 \). In Section \ref{sec:2.4}, we introduced the notion of boundary of the para-complex hyperbolic space, which was defined as  \[
\partial_\infty \mathbb{H}_\tau^2 =\mathbb{P}_\tau\Big(\{ z\in\R_\tau^3 \ | \ \boldsymbol{q}(z,z) = 0 \}\Big),
\]where $\mathbb P_\tau$ means we are taking the quotient with respect to the para-complex numbers with non-zero absolute value.  Using the description in terms of incidence geometry (see Section \ref{sec:2.4}), the space \( \partial_\infty \mathbb{H}_\tau^2 \) is precisely identifiable with the flag variety \( \mathcal{F} \). Moreover, for each \( p \in \Sigma \), after identifying \( (E_\rho)_p \) with \( \mathbb{R}_\tau^3 \) and using the idempotent decomposition \( \mathbb{R}_\tau^3 = \mathbb{R}^3 e_+ \oplus \mathbb{R}^3 e_- \), we can construct a homeomorphism: \begin{align*}
    & \ \ \ \ \ \ \ \ \ \ \ \mathbb P_\tau\Big(\text{Iso}_1(\R_\tau^3)\Big)\longrightarrow\mathcal{F} \\ &[z]=[z^+e_++z^-e_-]\mapsto [(z^+,\Ker(\varphi_{z^-}))]
\end{align*}where $\varphi_{z^-}$ is the functional determined by the identification $\R^3\cong(\R^3)^*$ using the bilinear form $Q=\begin{psmallmatrix}
    0 & 0 & 1 \\ 0 & 1 & 0 \\ 1 & 0 & 0
\end{psmallmatrix}$ (see Section \ref{sec:2.4})
\begin{prop}\label{prop:transverse_section}
The isotropic section $s:T^1\Sigma\to\pi^*E_\rho$ is transverse and induces an $\big(\mathcal{F},\SL(3,\R)\big)$-structure on $\Pp T\Sigma$.
\end{prop}
\begin{proof} Let $\alpha$ be a coordinate angle in the fiber of $T^1\Sigma$. To enlighten the notation, we will still denote by $D$ the flat connection on $\pi^{*}E_{\rho}$. In order to show that $s$ is transverse to the flat structure, it is sufficient to prove that the vectors $s, D_{\alpha}s, D_{X}s$ and $D_{Y}s$ are linearly independent at all points, whenever the vector fields $X$ and $Y$ are chosen to provide a local frame of $T\Sigma$. Since $s=s_{1}+s_{2}$, it is actually sufficient to show that the sections $s_{1}, s_{2}, D_{\alpha}s_{1}, D_{\alpha}s_{2}, D_{X}s$ and $D_{Y}s$ are linearly independent at every point. By definition, since $s_{1}$ takes values in the unit circle in the fiber of $\pi^{*}(T^1\Sigma)\subset \pi^{*}E_{\rho}$, the derivative $D_{\alpha}s_{1}$ is orthogonal to $s_{1}$, thus at every point $(p,\beta)\in T^{1}\Sigma$ they generate $\pi^{*}T_{p}\Sigma$. By the same argument, using the fact that $s_{2}$ only differs from $s_{1}$ by pre-composition by a rotation and post-composition by $\p$, the sections $s_{2}$ and $D_{\alpha}s_{2}$ generate $\pi^{*}(\p T_{p}\Sigma)$. This implies that $D_{X}s$ and $D_{Y}s$ are linearly independent if and only if their projections onto $\pi^{*}(\R\hat{\sigma} \oplus \p\R\hat{\sigma})$ are. Using the explicit formula of the connection $D$, we see that those projections are
\begin{align*}
    \mathrm{pr}(X)&:=\theta^{t}(X)s_{1}+\p(D_{X}(s_{1}\circ R_{\pi/2})) \\
    \mathrm{pr}(Y)&:=\theta^{t}(Y)s_{1}+\p(D_{Y}(s_{1}\circ R_{\pi/2})).
\end{align*}
In local coordinates, after choosing a local orthonormal frame $\{u,v\}$ of $T\Sigma$, we can write
\begin{align*}
    s_{1}(z, \beta)&=\cos(\beta)u(z)+\sin(\beta)v(z) \\
    s_{2}(z, \beta)&=-\sin(\beta)\p u(z) + \cos(\beta)\p v(z).
\end{align*}
Therefore, if we choose $X=u$ and $Y=v$ and $\{u,v\}$ to be the dual frame of $\theta^{t}=(\theta_{1}, \theta_{2})$, we get
\begin{align*}
    \mathrm{pr}(X)(z,\beta)&=\cos(\beta)\hat{\sigma}(z)-\sin(\beta)\p\hat{\sigma}(z) \\
    \mathrm{pr}(Y)(z,\beta)&=\sin(\beta)\hat{\sigma}(z)+\cos(\beta)\p \hat{\sigma}(z), 
\end{align*}
which are evidently independent for all $(z,\beta) \in T^{1}\Sigma$. \\
Finally, in order for $s$ to induce a developing map with values in $\mathcal{F}$, we need to consider projective classes of $\mathbf{q}$-isotropic vectors, and thus consider the section $s$ up to scalar multiplication by unitary para-complex numbers. From the expression in local coordinates of $s(z,\beta)=s_{1}(z,\beta)+s_{2}(z,\beta)$, it is straightforward to verify that the $[s(z,\beta)]=[s(z,\beta+\pi)]$, thus $s$ descends to a section of the bundle $\Pp_{\tau}(\mathrm{Iso}_{1}(\pi^{*}E_{\rho}))$ over $\Pp T\Sigma$. 
\end{proof}

\noindent This result implies the existence of a developing pair \((\text{dev}_\rho, h)\), where \(\text{dev}_\rho: \Pp T\widetilde{\Sigma} \to \mathcal{F}\) and \(h: \pi_1(\Pp T\Sigma) \to \SL(3,\mathbb{R})\) are, respectively, the developing map and the holonomy of the \((\mathcal{F}, \SL(3,\mathbb{R}))\)-structure determined by the section \(s\). Moreover, the representation \(h\) strongly factors through \(\rho\), meaning that \(h = \rho \circ \pi_*\). The next objective is to show that if the representation \(\rho\equiv\rho_0\) is Fuchsian, meaning that it factors through the irreducible embedding \(\PSL(2,\mathbb{R}) \hookrightarrow \SL(3,\mathbb{R})\), then the developing map \(\text{dev}_{\rho_0}\) is a homeomorphism onto the connected component \(\Omega_{\rho_0}^{dS}\) of the Guichard-Wienhard domain. In this case, the space-like maximal Lagrangian surface in $\h_\tau^2$ preserved by $\rho_0$ is a totally geodesic copy of the hyperbolic plane $\h^2$ defined in the hyperboloid model as a subset of $\R^{2,1}$. Then, the following lemma holds  \begin{lemma}\label{lm:domain_GW}
Let $\rho_0:\pi_1(\Sigma)\to\SL(3,\R)$ be a Fuchsian representation, then $$\Omega_{\rho_0}^{dS}=\{(v,\Ker(\varphi))\in\mathcal{F} \ | \ v\in\R^{2,1} \ \text{and} \ \Ker(\varphi)^{\perp}\in\R^{2,1} \ \text{are spacelike} \}$$
\end{lemma}
\begin{proof}
In the Fuchsian case, the convex subset $\mathcal{C}_{\rho_0}$ is given by the projective copy of $\h^2$ inside $\R\mathbb P^2$. Therefore, $p\in\mathcal C_{\rho_0}$ if and only if $\langle v,v\rangle_{2,1}<0$, where $v\in\R^{2,1}$ is understood as a vector generating the line representing the point $p\in\mathbb{RP}^2$. Since points in the boundary of $\mathcal{C}_{\rho_0}$ corresponds to isotropic vectors in $\R^{2,1}$ we conclude that $p\notin\overline{\mathcal{C}_{\rho_0}}$ if and only if $v$ is space-like. Using a similar argument, we see that \(\Ker(\varphi) \cap \mathcal{C}_{\rho_0} \neq \emptyset\) if and only if there exists a time-like vector in \(\Ker(\varphi)\). By assumption, we know that \(v \in \Ker(\varphi)\), and we have proved that \(v\) is space-like. This implies that \(\Ker(\varphi)\) is a plane of signature \((1,1)\) in \(\mathbb{R}^{2,1}\), and therefore, its orthogonal complement must necessarily be space-like.
\end{proof}
\begin{theorem}\label{thm:geometric_structure}
The developing map \(\dev_{\rho_0}: \Pp T\widetilde\Sigma \to \mathcal{F}\) of the \((\mathcal{F}, \SL(3,\R))\)-structure on \(\Pp T\Sigma\) with \(\rho_0\) Fuchsian is a homeomorphism onto \(\Omega_{\rho_0}^{dS}\).
\end{theorem}
\begin{proof} By Proposition \ref{prop:transverse_section}, we already know that $\dev_{\rho_{0}}$ is a local homeomorphism, thus we only need to prove that $\dev_{\rho_{0}}$ takes indeed values in $\Omega_{\rho_{0}}^{dS}$ and is bijective. \\
\underline{$\Ima(\dev_{\rho_{0}}) \subset \Omega_{\rho_{0}}^{dS}$.} The $\rho_{0}$-equivariant isotropic $\p$-alternating surface is the diagonal embedding $\hat{\sigma}: \mathbb{H}^{2} \hookrightarrow \mathbb{R}^{2,1}e_{+} \oplus \mathbb{R}^{2,1}e_{-}$ of the hyperboloid model of the hyperbolic plane with respect to the bilinear form $Q$ (\cite{RT_bicomplex}). Let $p\in \mathbb{H}^{2}$ and $\{u,v\}$ be an orthonormal frame of $\hat{\sigma}^{*}T_{p}\h^{2}$. We can write $u=u^{+}e_{+}+u^{-}e_{-}$ and $v=v^{+}e_{+}+v^{-}e_{-}$ with $u^{+}=u^{-}$ and $v^{+}=v^{-}$ because the embedding is diagonal. By definition of the developing map $\dev_{\rho_{0}}$, for every $\alpha \in S^{1}$, we have
\begin{align}\label{eq:dev}
    \dev_{\rho_{0}}(p,\alpha)&= \cos(\alpha)u+\sin(\alpha)v-\sin(\alpha)\tau u + \cos(\alpha)\tau v \notag \\ 
    &=[(\cos(\alpha)-\sin(\alpha))u^{+}+(\sin(\alpha)+\cos(\alpha))v^{+}]e_{+} \\
    &+[(\cos(\alpha)+\sin(\alpha))u^{-}+(\sin(\alpha)-\cos(\alpha))v^{-}]e_{-}. \notag
\end{align}
Since $u^{\pm}, v^{\pm}$ are tangent vectors to the hyperboloid in $\R^{2,1}$, the vectors $z^{+}:=(\cos(\alpha)-\sin(\alpha))u^{+}+(\sin(\alpha)+\cos(\alpha))v^{+}$ and $\Ker(\varphi_{z_{-}})^{\perp}=z_{-}:=(\cos(\alpha)+\sin(\alpha))u^{-}+(\sin(\alpha)-\cos(\alpha))v^{-}$ are both spacelike, and the claim follows from Lemma \ref{lm:domain_GW}.

\noindent \underline{$\dev_{\rho_{0}}$ is injective.} Assume that $\dev_{\rho_{0}}(p_1,\alpha_1)=\dev_{\rho_{0}}(p_2,\alpha_2)$, we shall show that $(p_1,\alpha_1)=(p_2,\alpha_2)$ in $\Pp T\widetilde{\Sigma}$. By construction $\dev_{\rho_{0}}(p_{j},\alpha_{j}) \in \mathrm{Iso}_{1}(\hat{\sigma}_{j}^{\perp_{\mathbf{q}}})$ for $j=1,2$. If $p_1 \neq p_2$, the intersection $P=\hat{\sigma}(p_{1})^{\perp_\mathbf{q}}\cap \hat{\sigma}(p_{2})^{\perp_\mathbf{q}}$ is a $\tau$-invariant plane and thus has signature $(1,1)$ for the inner product $\Ree_{\tau}(\mathbf{q})$. Therefore, there is a unique $\mathbf{q}$-isotropic vector in $P$, which we denote by $v=v^{+}e_{+}+v^{-}e_{-}$, and we must have $\dev_{\rho_{0}}(p_{j}, \alpha_{j})=v$. Since $\rho_{0}$ is Fuchsian, we know that $\hat{\sigma}$ is a diagonal embedding, and we have already shown that $v^{+}$ is spacelike and tangent to the hyperboloid in $\R^{2,1}$ at the points $\hat{\sigma}(p_{j})^{\pm}$, where $\hat{\sigma}(p_{j})=\hat{\sigma}(p_{j})e^{+}+\hat{\sigma}(p_{j})e_{-}$. Furthermore, the computation above shows that $v^{+}$ and $v^{-}$ are linearly independent. This gives a contradiction because if $p_{1}\neq p_{2}$, the intersection of the tangent planes at $\hat{\sigma}(p_{1})^{+}$ and $\hat{\sigma}(p_{2})^{+}$ is a spacelike line, which cannot contain the two linearly independent vectors $v^{+}$ and $v^{-}$. On the other hand, if $p_{1}=p_{2}=p$, then, after fixing an orthonormal basis of $\hat{\sigma}^{*}T\h^{2}$ as before, we can write
\begin{align*}
    \dev_{\rho_{0}}(p,\alpha_j)&=[(\cos(\alpha_j)-\sin(\alpha_j))u^{+}+(\sin(\alpha_j)+\cos(\alpha_j))v^{+}]e_{+}\\
    &+[(\cos(\alpha_j)+\sin(\alpha_j))u^{-}+(\sin(\alpha_j)-\cos(\alpha_j))v^{-}]e_{-}.
\end{align*}
and it is straightforward to check that $\dev_{\rho_{0}}(p,\alpha_1)=\dev_{\rho_{0}}(p,\alpha_2)$ in $\Pp_{\tau}(\mathrm{Iso}_{1}(\R^{3}_\tau))$ if and only if $\alpha_{1}=\alpha_{2}+\pi$, hence $[(p,\alpha_{1})]=[(p,\alpha_2)]$ in the projectivized tangent bundle.

\noindent \underline{$\dev_{\rho_{0}}$ is surjective.} Let $(v, \Ker(\varphi)) \in \Omega_{\rho_{0}}^{dS}$. Since $\Ker(\varphi)^{\perp}$ is spacelike, the plane $\Ker(\varphi)^{\perp}$ has signature $(1,1)$. The intersection $\Ker(\varphi)^{\perp}\cap \mathbb{H}^{2}$ is a geodesic $\gamma \in \mathbb{H}^{2}$, which we parameterize by arclength $\gamma(t)=\cosh(t)p+\sinh(t)w$ for some $p\in \mathbb{H}^{2}$ and $w\in T_{p}^{1}\mathbb{H}^{2}$. Up to scaling, we can assume that $v$ is unitary and with positive inner product with $p$, so that there is a unique time $t_{0}$ such that $\gamma'(t_{0})=v$. Let $p_{0}=\gamma(t_{0})$ and let $w$ be unit vector spanning $\Ker(\varphi)^{\perp}$ so that $\{v,w\}$ is an oriented orthonormal basis of $T_{p_{0}}\h^{2}$. Then, by the explicit formula of $\dev_{\rho_{0}}$ in Equation \eqref{eq:dev}, we see that $\dev(p_{0}, \frac{3}{4}\pi)= -\sqrt{2}ve_{+}+\sqrt{2}we_{-}$, which corresponds exactly to the pair $(v, \Ker(\varphi))$ under the homeomorphism between $\Pp_{\tau}(\mathrm{Iso}_{1}(\R^{3}_\tau))$ and $\mathcal{F}$.    
\end{proof}

\noindent This last result implies that the geometric structure defined by Guichard-Wienhard using the domain of discontinuity \(\Omega_\rho^{dS}\) and the one we construct using maximal Lagrangian immersions in \(\h_\tau^2\) coincide when \(\rho\) is Fuchsian. By applying a continuity argument and the Ehresmann-Thurston principle, we obtain the following immediate consequence.
\begin{cor}
For any $\rho:\pi_1(\Sigma)\to\SL(3,\R)$ in the Hitchin component, the $(\mathcal{F},\SL(3,\R))$-structure defined on $\Pp T\Sigma$ by Proposition \ref{prop:transverse_section} coincides with the one obtained as the quotient $\Omega_\rho^{dS}/\rho\big(\pi_1(\Sigma)\big)$.
\end{cor}
\begin{proof}
Theorem \ref{thm:geometric_structure} shows, in particular, that when \(\rho = \rho_0\) is Fuchsian, the space \(\Omega_{\rho_0}^{dS}/\Gamma_0\) is homeomorphic to \(\Pp T\Sigma\) and the $(\mathcal{F},\SL(3,\R))$ obtained in Proposition \ref{prop:transverse_section} coincides with the Guichard-Wienhard's one. Since the topology of \(\Omega_\rho^{dS}/\Gamma\) remains unchanged when \(\rho\) varies continuously (\cite[Theorem 9.12]{guichard2012anosov}), the quotient must have the same topology for every Hitchin representation. Moreover, given that the $(\mathcal{F},\SL(3,\R))$-structure varies continuously, by the Ehresmann-Thurston principle (\cite{thurston_notes}) we conclude the two geometric structures coincide for any representation in the Hitchin component. 
\end{proof}
    
\appendix
\section{Harmonic immersions in para-K\"ahler manifolds}\label{sec:appendix}
\noindent We present here an extension of a well-known result for harmonic immersions of Riemann surfaces in K\"ahler manifolds of constant holomorphic sectional curvature (\cite{wood1984holomorphic}) to the para-K\"ahler setting.\\

\noindent Let $X$ be a Riemann surface and $(M,\mathbf{P},g)$ a para-K\"ahler manifold of dimension $2n$ and constant para-holomorphic sectional curvature $K$. We denote by $\mathbf{q}$ the para-K\"ahler metric
\[
    \mathbf{q}(V,W)=g(V,W)+\tau g(V,\mathbf{P}W) \ \ \ \text{for $V,W\in \Gamma(TM)$}.
\]
Let $\C TM = TM \otimes_{\R}\C$ be the complexified tangent space of $M$. We extend $\mathbf{q}$ to a bilinear form $\mathbf{q}^{\C}$ on $\C TM$ that is $\C$-linear in the first entry and $\C$-antilinear in the second. We also extend the endomorphism $\mathbf{P}$ to $\C TM$ by $\C$-linearity and the Riemannian metric $g$ to a complex bilinear form $g_{\C}$ on $\C TM$. \\

\noindent Given a smooth immersion $\varphi:X\rightarrow M$, we extend its differential by $\C$-linearity to obtain a map $d_{\C}\varphi$ from the complex tangent space $T^{\C}X$ of $X$ to $\C TM$. If $\nabla$ denotes the Levi-Civita connection of $g$, the pull-back bundle $\varphi^{*}\C TM$ on $X$ inherits a natural connection $\overline{\nabla}$ that is $\C$-bilinear and satisfies
\[
    \overline{\nabla}_{Z}V = \nabla_{d^{\C}\varphi(Z)} V
\]
for all sections $Z\in \Gamma(TX)$ and $V \in \Gamma(\varphi^{*}\C TM)$, where we identified $(\varphi^{*}\C TM)_{x}$ with $T_{\varphi(x)}M$ for all $x \in X$. For a complex coordinate $z$ on $X$, we set
\[
    \varphi_{z,1}=\varphi_{z}:=d_{\C}\varphi(\partial_{z}) \ \ \ \text{and} \ \ \  \varphi_{\bar{z}}:=d_{\C}\varphi(\partial_{\bar{z}}) 
\]
and, more in general, for all $m\in \mathbb{N}_{\geq 2}$,
\[
    \varphi_{z,m}=\underbrace{\overline{\nabla}_{\partial_{z}}\cdots \overline{\nabla}_{\partial_{z}}}_{\text{$(m-1)$}}\varphi_{z} \ . 
\]

\begin{defi} The isotropic order of a smooth map $\varphi:X\rightarrow M$ is the largest integer $\gamma \in \mathbb{N}_{\geq 1} \cup \{\infty\}$ such that, in all local charts, 
\[
    \eta_{\alpha,\beta}:=\mathbf{q}^{\C}(\varphi_{z,\alpha}, \varphi_{\bar{z}, \beta})=0 
\]
for all natural numbers $\alpha, \beta \geq 1$ such that $\alpha+\beta \leq \gamma$. 
\end{defi}  

\noindent The following theorem generalizes a well-known result for harmonic immersions in K\"ahler manifolds:

\begin{theorem} Let $\varphi:X \rightarrow M$ be a harmonic map with isotropic order $\gamma+1$. Then
\begin{enumerate}[i)]
    \item for all integers $\alpha,\beta \geq 1$ with $\alpha+\beta=\gamma+1$, we have $\eta_{\alpha,\beta}=(-1)^{\beta-1}\eta_{\gamma,1}$;
    \item the tensor $\eta_{\gamma,1}dz^{\gamma+1}$ is holomorphic.
\end{enumerate}
\end{theorem}
\begin{proof} The first part of the statement follows from an elementary integration by part: details are left to the reader. \\
In order to prove ii), we first show by induction that the following property holds: for all integers $1\leq \alpha \leq \gamma$ we have
\begin{equation}\label{eq:span}
    \overline{\nabla}_{\partial_{\bar{z}}} \varphi_{z,\alpha} \in \mathrm{Span}\{\varphi_{z}, \mathbf{P}\varphi_{z}, \dots, \varphi_{z, \alpha-1}, \mathbf{P}\varphi_{z, \alpha-1}\} \ .
\end{equation}
For $\alpha=1$, the statement holds because
\[
    \overline{\nabla}_{\partial_{\bar{z}}}\varphi_{z}=0
\]
by harmonicity of $\varphi$. Now, we assume the statement holds for $\alpha-1$, and we compute
\[
    \overline{\nabla}_{\partial_{\bar{z}}}\varphi_{z,\alpha} = \overline{\nabla}_{\partial_{z}}\overline{\nabla}_{\partial_{\bar{z}}}\varphi_{z,\alpha-1}-R^{\nabla}_{\C}(\varphi_{z}, \varphi_{\bar{z}})(\varphi_{z,\alpha-1}) \ ,
\]
where $R^{\nabla}_{\C}$ denotes the $\C$-bilinear extension of the Riemann tensor on $(M,g)$. We analyse the two terms in the above expression separately:
\begin{itemize}
    \item by assumption $\overline{\nabla}_{\partial_{\bar{z}}} \varphi_{z,\alpha-1} \in \mathrm{Span}\{\varphi_{z}, \mathbf{P}\varphi_{z}, \dots, \varphi_{z, \alpha-2}, \mathbf{P}\varphi_{z, \alpha-2}\}$ hence, since $\mathbf{P}$ is parallel, $\overline{\nabla}_{\partial_{z}}\overline{\nabla}_{\partial_{\bar{z}}}\varphi_{z,\alpha-1} \in  \mathrm{Span}\{\varphi_{z}, \mathbf{P}\varphi_{z}, \dots, \varphi_{z, \alpha-1}, \mathbf{P}\varphi_{z, \alpha-1}\}$, as desired. 
    \item for the second term, using the formula for the Riemann tensor of a para-K\"ahler manifold with constant para-holomorphic sectional curvature $\kappa$ (Lemma \ref{lem:para_hol_curvature}), we get
    \begin{align*}
        R^{\nabla}_{\C}(\varphi_{z}, \varphi_{\bar{z}})\varphi_{z,\alpha-1}&=-\frac{\kappa}{4}\Bigg[g_{\C}(\varphi_{z}, \varphi_{z,\alpha-1})\varphi_{\bar{z}}-g_{\C}(\varphi_{\bar{z}},\varphi_{z,\alpha-1})\varphi_{z}+g_{\C}(\varphi_{z}, \mathbf{P}\varphi_{z,\alpha-1})\mathbf{P}\varphi_{\bar{z}} \\
        & \ \ \ \ \ \ \ \ \ \ -g_{\C}(\varphi_{\bar{z}},\mathbf{P}\varphi_{z,\alpha-1})\mathbf{P}\varphi_{z}+2g_{\C}(\varphi_{z},\mathbf{P}\varphi_{\bar{z}})\mathbf{P}\varphi_{z,\alpha-1}\Bigg] 
    \end{align*}
    and this belongs to $\mathrm{Span}\{\varphi_{z}, \mathbf{P}\varphi_{z}, \dots, \varphi_{z, \alpha-1}, \mathbf{P}\varphi_{z, \alpha-1}\}$ if and only if 
    \[
        g_{\C}(\varphi_{z}, \varphi_{z,\alpha-1})=g_{\C}(\varphi_{z}, \mathbf{P}\varphi_{z,\alpha-1})=0.
    \]
    These last conditions are satisfied because
    \begin{align*}
        g_{\C}(\varphi_{z}, \varphi_{z,\alpha-1})&= g_{\C}(\varphi_{z,\alpha-1},   \varphi_{z}) \\ &=\Ree_{\tau}(q^{\C}(\varphi_{z,\alpha-1},\varphi_{\bar{z}}))=\Ree_{\tau}(\eta_{\alpha-1,1})=0
    \end{align*}
    since $\varphi$ has isotropic order $\gamma+1$ and $(\alpha-1)+1=\alpha \leq \gamma$. Similarly,
    \begin{align*}
        g_{\C}(\varphi_{z}, \mathbf{P}\varphi_{z,\alpha-1})&= g_{\C}(\mathbf{P}\varphi_{z,\alpha-1},   \varphi_{z})\\ &=-\Ima_{\tau}(q^{\C}(\varphi_{z,\alpha-1},\varphi_{\bar{z}}))=-\Ima_{\tau}(\eta_{\alpha-1,1})=0 .
    \end{align*}
\end{itemize}
We can now prove that the tensor $\eta_{\gamma,1}dz^{\gamma+1}$ is holomorphic:
\begin{align*}
    \partial_{\bar{z}}\eta_{\gamma,1}&=\partial_{\bar{z}} (q^{\C}(\varphi_{z,\gamma}, \varphi_{\bar{z}})) \\
    &=q^{\C}(\overline{\nabla}_{\partial_{\bar{z}}}\varphi_{z,\gamma}, \varphi_{\bar{z}})+q^{\C}(\varphi_{z,\gamma}, \overline{\nabla}_{\partial_{z}}\varphi_{\bar{z}}) \\
    &=q^{\C}(\overline{\nabla}_{\partial_{\bar{z}}}\varphi_{z,\gamma}, \varphi_{\bar{z}}) \tag{$\overline{\nabla}_{\partial_{z}}\varphi_{\bar{z}}=0$ by harmonicity} \\
    &=q^{\C}\left(\sum_{j=1}^{\gamma-1} a_{j}\varphi_{z,j}, \varphi_{\bar{z}}\right) \tag{Equation \eqref{eq:span}}\\
    &=\sum_{j=1}^{\gamma-1}a_{j}q^{\C}(\varphi_{z,j},\varphi_{\bar{z}})=\sum_{j=1}^{\gamma-1}a_{j}\eta_{j,1}=0.
\end{align*}


\end{proof}

\bibliographystyle{alpha}
\bibliography{biblio}

\end{document}